\documentclass[reqno,a4paper]{amsart}
\usepackage{amsmath, amssymb, bm, graphicx, nicefrac,a4wide}
\usepackage{mathtools}
\mathtoolsset{showonlyrefs}

\usepackage{psfrag} \usepackage[usenames,dvipsnames]{xcolor}
\usepackage[all,cmtip,line]{xy} \usepackage[normalem]{ulem}
\usepackage{tikz,tikz-cd} \usetikzlibrary{cd} \usetikzlibrary{matrix,
  arrows} \usepackage[utf8]{inputenc} \usepackage{pgfplots}
\usepackage{subfig}
\usepackage[pdfpagelabels]{hyperref}
\hypersetup{
  colorlinks = true, linktoc = all,
  linkcolor = black, citecolor = black}

\usepackage{cancel}
\usepackage{upgreek}

\usepackage[T3,T1]{fontenc} \DeclareSymbolFont{tipa}{T3}{cmr}{m}{n}
\usepackage[utf8]{inputenc}
\usepackage{diagbox}

\DeclareMathAccent{\invbreve}{\mathalpha}{tipa}{16}

\newcommand{\sw}{\mathfrak{w}}

\newcommand{\parm}{\xi}

\newcommand{\SW}{\mathfrak{W}}

\newcommand{\wrk}{{\rm work}}
\newcommand{\wk}{w}
\newcommand{\rK}{\breve{K}}
\newcommand{\Vwk}{{\boldsymbol{w}}}

\newtheorem{theorem}{Theorem}[section]
\newtheorem{lemma}[theorem]{Lemma}
\newtheorem{corollary}[theorem]{Corollary}
\newtheorem{proposition}[theorem]{Proposition}
\theoremstyle{definition} \newtheorem{problem}[theorem]{Problem}
\newtheorem{definition}[theorem]{Definition}

 \theoremstyle{remark}
\newtheorem{remark}[theorem]{Remark}
\newtheorem{assumption}[theorem]{Assumption}

\newcommand{\bx}{\boldsymbol{x}}
\newcommand{\bzeta}{\boldsymbol{\zeta}}
 
\theoremstyle{definition}

\newcommand{\normz}[2][]{\lVert #2\rVert_{#1}}
 \newcommand{\bbP}{\mathbb{P}}
\newcommand{\curl}{\operatorname{\mathbf{curl}}}
\renewcommand{\div}{\operatorname{div}}
\newcommand{\p}[1]{\langle #1\rangle}  
\newcommand{\modulo}[1]{\left\vert #1\right\vert}

\newcommand{\hcurl}[1]{\bm{H}( \curl; #1)}

\newcommand{\hocurl}[1]{\bm{H}_0(\curl; #1 )}
\newcommand{\hdiv}[1]{\bm{H}(\div;#1)}

\newcommand{\Hsob}[2]{\bm{H}^{#1}( #2 )}
\newcommand{\Wsob}[2]{\bm{W}^{#1}( #2 )}
\newcommand{\Wsobpw}[2]{\bm{W}^{#1}_{\rm pw}( #2 )}

\newcommand{\hsob}[2]{{H}^{#1}( #2 )}
\newcommand{\wsob}[2]{{W}^{#1}( #2 )}

\newcommand{\Lp}[2]{\bm{L}^{#1}( #2 )}
\newcommand{\lp}[2]{{L}^{#1}(
  #2 )} 
\newcommand{\half}{\frac{1}{2}}

\newcommand{\KL}{\mbox{Karh\'{u}nen-Lo\`{e}ve }}

\newcommand{\be}{\begin{equation}} \newcommand{\ee}{\end{equation}}
\newcommand{\eps}{{\varepsilon}} %
 \newcommand{\set}[2]{\{#1\,:\,#2\}}

\newcommand{\C}{\mathbb{C}}
\newcommand{\cC}{\mathcal{C}}

\newcommand{\cF}{\mathcal F}

\newcommand{\cI}{\mathcal I}
\newcommand{\CF}{\mathcal F}

 \newcommand{\cL}{\mathcal L}
\newcommand{\cN}{\mathcal N}

\newcommand{\frakT} {{\mathfrak
    T}} %
\newcommand{\frakE} {{\mathfrak E}}

\newcommand{\bnul}{{\boldsymbol 0}} \newcommand{\bsnu}{{\boldsymbol
    \nu}} \newcommand{\bsmu}{{\boldsymbol \mu}}

\newcommand{\bsy}{{\boldsymbol y}} 

\newcommand{\bse}{{\boldsymbol e}}

\newcommand{\bT} {\mathbf{T}}
\newcommand{\bmT} {\bm{T}}
 \newcommand{\D}{\mathrm{D}}
\newcommand{\N}{\mathbb{N}} 
 
\newcommand{\dd}{\,{\rm d}}

\newcommand{\dup}[2]{\langle #1, #2\rangle}

\renewcommand{\bar}[1]{\overline#1} \renewcommand{\div}{{\rm div}}

\newcommand{\R}{{\mathbb{R}}} 

\newcommand{\IC}{{\mathbb C}} 
 \newcommand{\IN}{{\mathbb N}}
 \newcommand{\IR}{{\mathbb R}}

\newcommand{\norm}[2][]{\| #2 \|_{#1}}
\newcommand{\seminorm}[2][]{| #2 |_{#1}}
\newcommand{\normc}[2][]{\left\| #2 \right\|_{#1}}

\newcommand{\Dnul}{{\widehat{\D}}} 
\newcommand{\Dcomp}{{\widetilde{\D}}}

\newcommand{\Id}{\operatorname{\mathsf{I}}}
 \renewcommand{\d}{\operatorname{d}}
\newcommand{\ii}{\imath}

\newcommand{\ao}[2]{a(#1,#2)}
\newcommand{\fo}[1]{F(#1)}
\newcommand{\aot}[2]{a_{\bT}(#1,#2)}
\newcommand{\ant}[2]{\hat{a}_{\bT}(#1,#2)}
\newcommand{\act}[2]{\tilde{a}_{\bT}(#1,#2)}

\newcommand{\fot}[1]{F_{\bT}(#1)}
\newcommand{\fnt}[1]{\hat{F}_{\bT}(#1)}
\newcommand{\fct}[1]{\tilde{F}_{\bT}(#1)}

\newcommand{\aht}[2]{\tilde a_{h;\bT}(#1,#2)}
\newcommand{\fht}[1]{\tilde F_{h;\bT}(#1)}

\newcommand{\hcurlbf}[2][]{{{\boldsymbol{H}}_{#1}}(\curl; #2)}
\newcommand{\hdivbf}[2][]{{{\boldsymbol{H}}_{#1}}(\div; #2)}
\newcommand{\hscurlbf}[2]{{{\boldsymbol{H}}^{#1}}(\curl; #2)}
\newcommand{\hsdivbf}[2]{{{\boldsymbol{H}}^{#1}}(\div; #2)}
\newcommand{\hscurlbfpw}[2]{{{\boldsymbol{H}}_{\rm pw}^{#1}}(\curl; #2)}
\newcommand{\hcurlcurlbf}[2][]{{{\boldsymbol{H}}_{#1}}(\curl\curl; #2)}
\newcommand{\hncurlbf}[2][]{{{\boldsymbol{H}}_0^{#1}}(\curl; #2)}

\renewcommand{\div}{\operatorname{div}}
\renewcommand{\Re}{\operatorname{Re}}

\newcommand{\cP}{\mathcal{P}}

\newcommand{\scurl}{\operatorname{curl}}

\newcommand{\bA}{\mathbf{A}} 
 \newcommand{\bE}{\mathbf{E}}
 \newcommand{\bH}{\mathbf{H}}
 
 \newcommand{\bn}{\mathbf{n}}

\newcommand{\by}{\mathbf{y}} \newcommand{\bJ}{\mathbf{J}} 
 \newcommand{\bV}{\mathbf{V}}
 
\newcommand{\bU}{\mathbf{U}}

\newcommand{\tD}{\gamma_{\mathrm{D}}}

\DeclareMathOperator*{\essinf}{ess\,inf}

\numberwithin{equation}{section}

\title[Multilevel Shape UQ in CEM]{%
Multilevel Domain Uncertainty Quantification 
\\
in Computational Electromagnetics}

\author[R.~Aylwin]{Rub\'en Aylwin} \address{Faculty of Engineering and Sciences, Universidad Adolfo Ib\'{a}a\~nez, 7941169 Santiago, Chile}
\email{rbayl1993@gmail.com}

\author[C.~Jerez-Hanckes]{Carlos Jerez-Hanckes} \address{Faculty of Engineering and Sciences, Universidad Adolfo Ib\'{a}\~nez, 7941169 Santiago, Chile} \email{carlos.jerez@uai.cl}

\author[Ch.~Schwab]{Christoph Schwab} 
  \address{Seminar for Applied Mathematics, ETH Z\"urich, 8092 Z\"urich, Switzerland}
\email{schwab@math.ethz.ch}

\author[J.~Zech]{Jakob Zech} \address{
  Interdisciplinary Center for Scientific Computing (IWR), Universit\"at Heidelberg, 69120 Heidelberg, Germany}
\email{jakob.zech@uni-heidelberg.de}

\thanks{This work was supported in part by Fondecyt Regular 1171491,
  Conicyt-PFCHA/Doctorado Nacional/2017-21171791 and by the Swiss
  National Science Foundation under grant SNF149819. %
}


\date{\bf \today}


\pgfplotsset{yminorticks=false, legend style={/tikz/every even
    column/.append style={column sep=0.5cm}}, legend style={at={(axis
      description cs:0.5,-0.26)},anchor=north}, title style={at={(axis
      description cs:0.5,0.98)},anchor=south}, legend
  style={font=\fontsize{10}{8}\selectfont}, tick label
  style={font=\fontsize{10}{8}\selectfont}, x label style={at={(axis
      description cs:0.5,-0.02)},anchor=north}, y label
  style={at={(axis description cs:0,0.5)},anchor=south}}

\begin{document}
\normalem

\begin{abstract}
  We continue our study [Domain Uncertainty Quantification in
    Computational Electromagnetics, JUQ (2020), {\bf 8}:301--341] of
  the numerical approximation of time-harmonic electromagnetic fields
  for the Maxwell lossy cavity problem for uncertain geometries. We
  adopt the same affine-parametric shape parametrization framework,
  mapping the physical domains to a nominal polygonal domain with
  piecewise smooth maps.  The regularity of the pullback solutions on
  the nominal domain is characterized in piecewise Sobolev spaces.  We
  prove error convergence rates and optimize the algorithmic
  steering of parameters for edge-element discretizations in the
  nominal domain combined with: (a) multilevel Monte Carlo sampling,
  and (b) multilevel, sparse-grid quadrature for computing the
  expectation of the solutions with respect to uncertain domain
  ensembles.  In addition, we analyze sparse-grid interpolation to
  compute surrogates of the domain-to-solution mappings. All
  calculations are performed on the polyhedral nominal domain, which
  enables the use of standard simplicial finite element
  meshes. We provide a rigorous fully discrete error analysis and
  show, in all cases, that dimension-independent algebraic convergence
  is achieved. For the multilevel sparse-grid quadrature methods, we
  prove higher order convergence rates which are free from the
  so-called curse of dimensionality, i.e.~independent of the number of
  parameters used to parametrize the admissible shapes.  Numerical
  experiments confirm our theoretical results and verify the
  superiority of the sparse-grid methods.
\end{abstract}
\maketitle

\section{Introduction}
\label{sec:intro}
In recent years, \emph{computational uncertainty quantification}
(computational UQ for short)
has emerged as a sub-discipline of computational science and engineering.
A broad theme within computational UQ is the efficient numerical analysis
of parametric partial differential equation models in science and engineering.
They are usually based on parametric families of domains which are homeomorphic,
in particular, to one fixed {\em nominal} reference domain via parametric families
of diffeomorphisms. In the common case that the parametric dependence 
involves possibly infinitely many parameters, the solution families 
become likewise infinite-parametric. The numerical approximation of such
parametric solution sets is costly due to the usually large number of 
numerical approximations of parametric solutions that must be generated
to approximate the entire parametric solution family. In the presently
considered time-harmonic Maxwell equation model, the physical domain 
is necessarily three-dimensional, which is an additional source of 
computational complexity.

The present paper addresses the formulation and error analysis of
multilevel Monte-Carlo (MLMC) Finite Element (FE) schemes for
efficient computational domain uncertainty quantification of
time-harmonic, electromagnetic scattering from parametric families of
lossy cavities in a bounded domain in three-dimensional space.
\subsection{Previous work}
\label{sec:Prevwork}

The question of \emph{shape recovery in time-harmonic acoustic and
  electromagnetic scattering} subject to noisy measurements of
far-fields has received increasing attention in recent
years. Classical shape calculus suggests that under certain
conditions, the dependence of the \emph{forward-map} from scatterer
shape to far-field is continuous and differentiable in suitable
topologies (cf.~ \cite{DelfZol_2ndEd_2011} and references therein).
This continuous dependence implies strong measurability of solution
families corresponding to random ensembles of admissible domains,
thereby justifying the use of Monte-Carlo sampling to explore the
solution manifold.

In the discipline of computational uncertainty quantification, one is
interested in \emph{numerically approximating parametric solution
  families} corresponding to parametric representations of admissible
shapes, e.g., by Fourier-, wavelet- or \KL-expansions. Naturally,
many-parametric representations of shapes will imply many-parametric
solution families, thereby mandating \emph{sampling and interpolation
  of parametric solutions in high-dimensional parameter domains}.  A
broad class of forward data-to-solution maps which arise in shape
uncertainty quantification have been shown to be holomorphic as maps
between---possibly complexified---Banach spaces.  We refer to
\cite{HSSS2018,Jerez-Hanckes:2017aa,Schwab:2016aa} and the references
there for proofs of this so-called \emph{Shape Holomorphy} of PDEs.
This property has been used in \cite{AJSZ20} for the numerical
analysis of UQ in the time-harmonic propagation of electromagnetic
waves in lossy cavities in ${\mathbb R}^3$.  In this last reference,
we developed a single-level algorithm whose convergence properties
were analyzed; in particular, dimension-independent convergence rates
of suitable sparse-grid interpolation and approximation algorithms for
the many-parametric dependence were shown. Moreover, we showed that
these algorithms outperform the corresponding Monte-Carlo sampling
considerably, also for many-parametric shape representations.

The present paper develops the corresponding multilevel algorithms and
their error analysis. As is well-known from previous work, e.g.,
\cite{ABChS_2011,KCMG_2011}, 
for scalar, parametric PDEs, the analysis
of the multilevel FE algorithms does require additional regularity of
the parametric solution families. Specifically, we require holomorphy
of the parametric, time-harmonic solutions for a Maxwell-like,
non-homogeneous problem on a so-called {\em
  nominal}---reference---domain, with non-homogeneous coefficients
resulting from the pullback of the parametric, physical domain to this
nominal domain.  
The requited regularity results for solutions of 
Maxwell-like equations in the nominal domain, for \emph{non-constant
  coefficients of low differentiability}, were recently obtained in
\cite{AlbertiRoughMaxw}. 
The shape holomorphy results required for the
error analysis of the multilevel Smolyak quadrature algorithms 
is developed in the present paper.

\subsection{Paper Layout}
\label{sec:layout}
The structure of this text is as follows.  In
Section~\ref{sec:mslsscav}, we present the governing equations and the
problem formulation of Maxwell's equations on a lossy cavity in three
space dimensions. Specifically, we focus on the parametric model of
the cavity's uncertain shape, and of the variational form of the
governing equations.  We set in particular notation, and recapitulate
basic or known results on the unique solvability of the forward model;
particular attention is paid to {\em a priori} bounds which are
uniform over all realizations of domains. The presented approach opts
for the so-called {\em domain-mapping formulation}, whereby all
realizations of domains which occur in numerical simulation are
mathematically formulated on one fixed, so-called {\em nominal
  domain}. As we show here, these entails the need to numerically
solve a potentially large number of Maxwell-like PDEs with variable,
metric-dependent coefficients in the nominal domain.

Section~\ref{sec:DiscSol} addresses the discretization of the
parametric forward problem, in particular edge-elements in the
nominal domain (cf.~Section~\ref{ssec:fe}). Section~\ref{sec:ParamSol}
describes the mathematical setting of regularity of the parametric
solution families. In comparison to the error analysis of the
single-level Smolyak quadrature algorithms in \cite{AJSZ20}, stronger
parametric regularity results are required. Section~\ref{ssec:mlmc}
recapitulates the MLMC algorithm, and Section~\ref{sec:MLSmol} the
corresponding multilevel Smolyak quadrature one. Finally,
Section~\ref{sec:experiments} contains several numerical experiments
to illustrate our theoretical results, confirming %
in particular the
superior accuracy versus cost performance of the multilevel
  version compared to the single-level algorithms from \cite{AJSZ20}.
\subsection{General notation}
\label{sec:Notation}
 Let $m\in\IN_0:=\{0\}\cup\IN$ and let $\D\subset\IR^d$ be an open and bounded Lipschitz domain,
then $\cC^m(\D)$ and $\cC^m(\D;\IC)$ denote the set of $m$-times
continuously differentiable functions with real and complex values in
$\overline{\D}$, respectively. Furthermore, $\cC^\infty(\D)$ denotes the
space of infinitely differentiable functions in $\overline{\D}$ and we
write $\cC^m_0(\D)$ for the set of elements of $\cC^m(\D)$ with compact
support on $\D$---analogous definitions apply for $\cC^\infty(\D;\IC)$
and $\cC^m_0(\D;\IC)$. 

For $k\in\N_0$, we write $\bbP_k(\D;\IC^d)$
  the space of functions from $\D$ to $\IC^d$, such that each of their
  $d$ components is a polynomial with complex coefficients of total
  degree $k$ or smaller. $\widetilde\bbP_k(\D;\IC^d)$
  represents the space of functions in $\bbP_k(\D;\IC^d)$ which have
  polynomials of degree exactly $k$ on each component.

The set of all bounded antilinear mappings from a Banach space
$X$ to $\IC$ is written as $X^*$ and is referred to as the dual
space of $X$. The norm and duality product of a Banach space shall be
denoted by the use of subscript ($\norm[X]{\cdot}$ and
$\p{\cdot,\cdot}_{X\times X^*}$, respectively).
For a pair of Banach spaces $X$ and $Y$, the set
of all linear mappings from $X$ to $Y$ is denoted $\cL(X;Y)$. For $s\geq 0$ and $p\geq 1$, $\lp{p}{\D}$ denotes the
Banach space of $p$-integrable functions over $\D$, while
$\wsob{s,p}{\D}$ denotes the standard Sobolev spaces
of order $s$ as defined in \cite[Chap.~3]{McLean2000}, where
we use the convention $\wsob{0,p}{\D}=\lp{p}{\D}$. If
$p=2$, we shall use the standard notation
$\hsob{s}{\D}=\wsob{s,2}{\D}$. Furthermore, the
norm and semi-norm of $\hsob{s}{\D}$ are denoted as
$\norm[s,\D]{\cdot}$ and $\seminorm[s,\D]{\cdot}$, respectively.

Let $p\geq 1$ and let $(\Omega,\mathfrak{F},\upnu)$ be a
probability space, where we take $\Omega$ to be a sample
space, $\mathfrak{F}$ a $\sigma$-algebra on $\Omega$ and $\upnu$ a probability measure.
Given a separable Hilbert space $X$, we say that a function $f:\Omega\to X$ is
measurable (or a random variable) if for every Borel set $B\subset X$ there holds that
$\set{\zeta\in\Omega}{f(\zeta)\in B}\in\mathfrak{F}$, we say that $f$ is
strongly measurable if it is the pointwise limit of a sequence of simple functions $\{f_n\}_{n\in\IN}$,
and we say that $f$ is Bochner integrable if it is strongly measurable and $\lim_{n\to\infty}\int_\Omega\norm[X]{f(\zeta)-f_n(\zeta)}\d\!\upnu(\zeta)=0$. Moreover, if $f:\Omega\to X$ is Bochner integrable, its integral over $\Omega$ is
defined as $\int_\Omega f(\zeta)\d\!\upnu(\zeta):=\lim_{n\to\infty}\int_\Omega f_n(\zeta)\d\!\upnu(\zeta)$. We
denote the Bochner space
of $p$ integrable, measurable mappings in
$(\Omega,\mathfrak{F},\upnu)$ with values in $X$ as
$\lp{p}{\Omega,\mathfrak{F},\upnu;X}$.
When the $\sigma$-algebra on $\Omega$ and the
probability measure are clear from the context, we will denote the norm of
$\phi\in\lp{p}{\Omega,\mathfrak{F},\upnu;X}$ as
$\norm[\lp{p}{\Omega;X}]{\phi}$. For further details we refer to
\cite{prato_zabczyk_2014,Diestel_1977,hytonen2016analysis}.

In general, boldface symbols will be used to differentiate
the vector-valued counterparts of scalar functions and
functional spaces, e.g., $\Lp{2}{\D}$ denotes the functional
space of vector-valued functions with each of its
$d$-components in $\lp{2}{\D}$. In particular, the
$\lp{2}{\D}$-inner product is denoted $(\cdot,\cdot)_{\D}$. 
Functional spaces built of tensor quantities are to be
identified by indicating the range of the elements it contains,
eg., $\Lp{2}{\D;\IC^{d\times d}}$  represents the space of
tensor-valued functions with each of their $d^2$ entries
belonging to $\lp{2}{\D}$.

Euclidean norms in $\IR^d$ are denoted by
$\norm[\IR^d]{\cdot}$, while the induced matrix norm
is denoted by $\norm[\IR^{d\times d}]{\cdot}$
with analogous versions when in $\IC^d$ and $\IC^{d\times d}$. 
The Jacobian of a differentiable function $\bU:\D\to\IC^d$ is written as
$\d\!\bU:\D\to\IC^{d\times d}$. For a general square matrix
$\bA\in\IC^{d\times d}$, we write the transpose matrix of
$\bA$ as $\bA^\top$, its determinant as $\det(\bA)$ and its
inverse as $\bA^{-1}$, when it exists. Finally, the overline notation
will be used to represent complex conjugation as well as
the closure of a set.
\section{Maxwell's equations on a lossy cavity}
\label{sec:mslsscav}
We begin this section %
by stating the time-harmonic Maxwell's lossy cavity problem and its functional
framework.
\subsection{Functional spaces for Maxwell's equations}
\label{sec:FncSpc}
Let $\D$ be an open and bounded Lipschitz domain in $\IR^3$
with simply connected boundary $\partial\D$, exterior
$\D^c:=\IR^3\setminus\overline\D$ and with exterior unit normal
vector $\bn$---pointing from $\D$ to $\D^c$.
We recall the standard functional spaces required to formulate
Maxwell problems:
\begin{align*}
\hcurlbf{\D}&:=\{\bU\in\Lp{2}{\D}\; : \; \curl \bU\in\Lp{2}{\D}\},\\
\hcurlcurlbf{\D}&:=\{\bU\in\hcurlbf{\D}\; :\; \curl\bU\in\hcurlbf{\D}\},\\
\hdivbf{\D}&:=\{\bU\in\Lp{2}{\D}\; :\; \div\bU\in\lp{2}{\D}\},
\end{align*}
and introduce, for $m\in\IN_0$ and $p\geq1$, extensions of
$\hcurlbf{\D}$ and $\hdivbf{\D}$ to spaces with additional regularity
and arbitrary integrability, 
\begin{align}
\Wsob{m,p}{\curl;\D}:=\{\bU\in\Wsob{m,p}{\D}\; :\; \curl\bU\in\Wsob{m,p}{\D}\},\quad
\hscurlbf{m}{\D}:=\Wsob{m,2}{\curl;\D},\\
\Wsob{m,p}{\div;\D}:=\{\bU\in\Wsob{m,p}{\D}\; :\; \div\bU\in\wsob{m,p}{\D}\},\quad
\hsdivbf{m}{\D}:=\Wsob{m,2}{\div;\D},
\end{align}
with associated norms
\begin{gather}
\norm[\Wsob{m,p}{\curl;\D}]{\bU}:=(\norm[\Wsob{m,p}{\D}]{\bU}^p+\norm[\Wsob{m,p}{\D}]{\curl\bU}^p)^\frac{1}{p},\\
\norm[\Wsob{m,p}{\div;\D}]{\bU}:=(\norm[\Wsob{m,p}{\D}]{\bU}^p+\norm[\wsob{m,p}{\D}]{\div\bU}^p)^\frac{1}{p},
\end{gather}
for $p\in[1,\infty)$ and the usual modification for $p=\infty$.
We point out that $\Wsob{0,2}{\curl;\D}\equiv\hcurlbf{\D}$
and $\Wsob{0,2}{\div;\D}\equiv\hdivbf{\D}$.
\begin{definition}
\label{def:traces}
For $\bU\in\bm{\cC}^{\infty}(\D)$ we define the following
trace operators:
\begin{align*}
\gamma_D\bU:=\bn\times(\bU\times\bn)\vert_{\partial\D},\quad
\gamma_{D}^{\times}\bU:=(\bn\times\bU)\vert_{\partial\D}\quad\mbox{and}\quad
\gamma_N\bU:=(\bn\times\curl\bU)\vert_{\partial\D},
\end{align*}
as the Dirichlet trace, flipped Dirichlet trace and Neumann
trace operators, respectively.
\end{definition}

The trace operators in Definition \ref{def:traces} may be
extended to continuous linear functionals from $\hcurlbf{\D}$
and $\hcurlcurlbf{\D}$ to subsets of
$\Hsob{-\frac{1}{2}}{\partial\D}:={\Hsob{\frac{1}{2}}{\partial\D}}^*$.
Specifically, we consider the following trace spaces
({cf.}~\cite{BuCo00,BufHipTvPCS_NM2003}):
\begin{align}
\boldsymbol{H}^{-\half}_{\div}(\partial\D) &:= \{\bU\in
(\bn\times(\Hsob{\half}{\partial\D}\times\bn))^{*}:\div_{\partial\D}\bU \in
\hsob{-\half}{\partial\D} \},\\
\boldsymbol{H}^{-\half}_{\scurl}(\partial\D) &:= \{\bU\in
(\Hsob{\half}{\partial\D}\times\bn)^*:\scurl_{\partial\D}\bU
\in \hsob{-\half}{\partial\D} \},
\end{align}
where $\div_{\partial\D}$ and $\scurl_{\partial\D}$ are,
respectively, the surface  divergence and surface scalar
curl operators and
$\boldsymbol{H}^{-\half}_{\scurl}(\partial\D)={\boldsymbol{H}^{-\half}_{\div}(\partial\D)}^*$
(cf.~\cite[Thm.~2]{BuHip} and \cite[Rmk.~3.32]{Monk:2003aa}).
Then, the operators in Definition \ref{def:traces} may be
continuously extended as
\begin{alignat*}{2}
\gamma_D: &\hcurlbf{\D} &&\to\boldsymbol{H}^{-\frac{1}{2}}_{\scurl}(\partial\D),\\
\gamma_D^{\times}: &\hcurlbf{\D} &&\to\boldsymbol{H}^{-\half}_{\div}(\partial\D), \\
\gamma_N: &\hcurlcurlbf{\D} &&\to\boldsymbol{H}^{-\half}_{\div}(\partial\D).
\end{alignat*} 
We also introduce the space of functions with
well defined curl and null flipped Dirichlet trace:
\begin{align*}
\hncurlbf{\D}:=\{\bU\in\hcurlbf{\D}\; :\; \gamma_{\D}^{\times}\bU=\bnul\;\mbox{on}\;\partial\D\},
\end{align*} 
which is a closed subspace of $\hcurlbf{\D}$.
Finally, for $\bU$ and $\bV\in \hcurlbf{\D}$
there holds the following integration by parts formula
\cite[Eq.~(27)]{BuCo00}:
\begin{align}\label{eq:green}
(\bU,\curl\bV)_\D-(\curl\bU,\bV)_\D =
-\dup{\tD^\times\bU}{\tD\bV}_{\partial\D}\;
\end{align}
where $\dup{\cdot}{\cdot}_{\partial\D}$ denotes the
duality between $\boldsymbol{H}_{\mathrm{div}}^{-\half}(\partial \D)$
and $\boldsymbol{H}^{-\half}_{\scurl}(\partial\D)$.
\subsection{Lossy cavity problem}
\label{sec:EM}
We consider the EM cavity problem for a time-harmonic dependence
$e^{\ii\omega t}$ with circular frequency $\omega>0$ and $\ii^2=-1$
on $\D$. The electric permittivity is denoted
$\varepsilon(\bx)\in \Lp{\infty}{\D;\C^{3\times 3}}$ and the magnetic
permeability is denoted as $\mu(\bx)\in \Lp{\infty}{\D;\C^{3\times 3}}$, where
losses are represented by their respective imaginary parts. With the current density
$\bJ\in \Lp{2}{\D}$, Maxwell's equations in $\D$ read
\begin{align}\label{eq:firstorder}
\begin{aligned}
\curl \bE + \ii\omega \mu \bH &= \bnul &&\text{in }\D,\\
\ii\omega\eps\bE -\curl \bH &=-\bJ&&\text{in }\D.
\end{aligned}
\end{align}
Assuming the pointwise inverse
$\mu^{-1}\in \Lp{\infty}{\D;\C^{3\times 3}}$ to be
well defined,  the system in \eqref{eq:firstorder} can
be reduced to
\begin{equation}\label{eq:MaxE}
\curl \mu^{-1}\curl \bE - \omega^2\eps \bE =
-\ii\omega\bJ\qquad\text{in }\D.
\end{equation}
We further impose perfect electric conductor (PEC)
boundary conditions on $\partial\D$,
\begin{equation}\label{eq:bc}
\tD^\times\bE=\bnul\quad\text{on }\partial\D.
\end{equation}

\begin{remark}\label{rmk:firstorder}
Under the assumption that the pointwise inverse
$\mu^{-1}:\D\to\C^{3\times 3}$ belongs to
$\Lp{\infty}{\D;\C^{3\times 3}}$, any weak solution
$\bE$, $\bH\in \hcurlbf{\D}$ of \eqref{eq:firstorder}
gives a weak solution of \eqref{eq:MaxE} and viceversa by setting
\begin{align}\label{eq:bH}
\bH:=\frac{\ii}{\omega}{\mu}^{-1}\curl\bE.
\end{align}
\end{remark}
\subsection{Existence and uniqueness of solutions}
\label{sec:ex_uniq}
By multiplying \eqref{eq:MaxE} by a test function
$\bV\in\hncurlbf{\D}$ and integrating by parts, using
\eqref{eq:green}, one derives the usual weak formulation
of the Maxwell lossy cavity problem \eqref{eq:MaxE}--\eqref{eq:bc}.

\begin{problem}[Maxwell cavity problem]
\label{prob:weak}
We seek $\bE\in\hncurlbf{\D}$ such that, with
\begin{align}\label{eq:aF}
\begin{aligned}
\ao{\bU}{\bV}&:= \int_\D \mu^{-1}\curl\bU\cdot
\curl\overline{\bV}\dd \bx
-\int_\D\omega^2\eps\bU\cdot\overline{\bV}\dd \bx,\\
\fo{\bV}&:=-\ii\omega\int_{\D}\bJ\cdot
\overline{\bV}\dd \bx,
\end{aligned}
\end{align}
for all $\bU$, $\bV\in\hncurlbf{\D}$, it holds that
\begin{align}\label{eq:weak}
\ao{\bE}{\bV}=\fo{\bV}\qquad\forall\;
\bV\in\hncurlbf{\D}.
\end{align}
\end{problem}

We will work under a positivity assumption for the parameters
defining $\ao{\cdot}{\cdot}$ ({cf.}~\cite{AJSZ20,Ern:2018aa}).
\begin{proposition}
\label{prop:existence}
Assume that $\bJ\in\Lp{2}{\D}$, that $\eps,\ \mu^{-1}\in\Lp{\infty}{\D;\IC^{3\times 3}}$ and that
there exist $\theta\in\R$ and $\alpha>0$ such that
\begin{align}\label{eq:alpha}
\inf_{0\neq \bzeta\in\C^3}
\essinf_{\bx\in\D}\;
\min\left\{
\frac{\Re(\bzeta^\top e^{\ii\theta}\mu(\bx)^{-1}\bar\bzeta)}{\norm[\C^3]{\bzeta}^2},
\frac{-\Re(\bzeta^\top e^{\ii\theta}\omega^2\eps(\bx)\bar\bzeta)}{\norm[\C^3]{\bzeta}^2}
\right\} \ge\alpha, 
\end{align}
holds. Then, Problem
\ref{prob:weak} has a unique solution
$\bE\in\hncurlbf{\D}$ and
\begin{align}
\label{eq:apriori}
\norm[\hcurlbf{\D}]{\bE}\le \frac{1}{\alpha}\norm[\hncurlbf{\D}^*]{F},
\end{align}  
for $F$ in \eqref{eq:aF} and $\alpha>0$ as in
\eqref{eq:alpha}.
\end{proposition}
\begin{proof}
Under our assumptions, the sesquilinear
form $e^{\ii\theta} \ao{\cdot}{\cdot}$ is coercive and continuous on the space $\hncurlbf{\D}$, i.e.,
\begin{align}
&\Re(e^{\ii\theta} \ao{\bU}{\bU}) \geq \alpha\norm[\hcurlbf{\D}]{\bU}^2
\\
&\Re(e^{\ii\theta} \ao{\bU}{\bV}) 
  \leq C(\norm[\Lp{\infty}{\D;\IC^{3\times 3}}]{\mu^{-1}}+\norm[\Lp{\infty}{\D;\IC^{3\times 3}}]{\varepsilon})
                                                           \norm[\hcurlbf{\D}]{\bU}\norm[\hcurlbf{\D}]{\bV},
\end{align}
for all $\bU$, $\bV$ in $\hncurlbf{\D}$, 
and 
with a constant $C>0$ independent
of the parameters $\mu^{-1}$ and $\varepsilon$.
The complex
Lax-Milgram lemma (see, e.g., \cite[Chap.~VI, Thm.~1.4]{MR564653})
implies
\begin{align}
\bE\mapsto e^{\ii\theta}\ao{\bE}{\cdot}:
\hncurlbf{\D}\to\hncurlbf{\D}^*
\end{align}
to be an isomorphism, so that
\begin{align}
\bE\mapsto \ao{\bE}{\cdot}:
\hncurlbf{\D}\to\hncurlbf{\D}^*
\end{align}
is an isomorphism as well. Additionally, the Lax-Milgram lemma
gives the {\em a priori} bound on the solution in \eqref{eq:apriori}.
\end{proof}

Due to Remark \ref{rmk:firstorder}, we may recover the magnetic field
as $\bH:=\frac{\ii}{\omega}\mu^{-1}\curl\bE$, from where the pair
$\bE$ and $\bH$ belong to $\hcurlbf{\D}$ and solve
\eqref{eq:firstorder}.
\subsection{Domain perturbations}
\label{ssec:pert_dom}
We consider Maxwell's lossy cavity problem (Problem \ref{prob:weak})
on a family of domains given as perturbations of $\Dnul\subset\IR^3$,
an open and bounded Lipschitz domain henceforth referred to as the
\emph{nominal domain}. The set of admissible domain perturbations is
$\frakT$ and we set $\D_\bT:=\bT(\Dnul)$ for every $\bT\in\frakT$. In
order to consider Maxwell's equations on the family of domains
$\{\D_\bT\}_{\bT\in\frakT}$, we will require suitable extensions of
the data $\eps$, $\mu$ and $\bJ$ to every perturbed domain, as well as
assumptions on $\frakT$. Furthermore, in order to prove uniform
convergence rates of finite element solutions of Maxwell's equations
on these perturbed domains we enforce smoothness conditions on
$\Dnul\subset\IR^3$, the perturbations $\bT\in\frakT$ and the data
$\eps$, $\mu$ and $\bJ$.

\begin{assumption}
\label{ass:material_trafo_smooth}
Fix $N\in\IN$, $q>3$, $\vartheta\in (0,1)$, $\theta\in\IR$ and $\alpha,\alpha_s>0$.
We assume the existence of an open, convex and bounded domain
$\D_H\subset\IR^3$ such that $\D_\bT:=\bT(\Dnul)\subset\D_H$
for all $\bT\in\frakT$
and assume that the following conditions hold:
\begin{enumerate}
\item the \emph{nominal domain} $\Dnul\subset\IR^3$ 
is a bounded and path connected domain of class $\cC^{N,1}$,
\item \label{it:Assum_1_frakT} the set of admissible domain perturbations $\frakT$
is a compact subset of $\bm\cC^{N,1}(\Dnul)$ such that every
$\bT\in\frakT$ is bijective and such that $\bT^{-1}\in\bm\cC^{N,1}(\D_\bT)$,
$\det(\d\!\bT)>0$ everywhere on $\Dnul$ and
\begin{align}
\label{eq:vartheta_smooth}
\vartheta\leq\norm[\lp{\infty}{\Dnul}]{\det\d\!\bT},\quad
\norm[\lp{\infty}{\Dnul}]{\det\d\!\bT^{-1}},\quad
\norm[\bm\cC^{0,1}(\Dnul)]{\bT},\quad
\norm[\bm\cC^{0,1}(\D_{\bT})]{\bT^{-1}}\leq\vartheta^{-1},
\end{align}
\item the magnetic permeability is invertible everywhere on $\D_H$ and there holds that
$\eps$, $\mu$ and $\mu^{-1}$ belong to $\bm{W}^{N,\infty}(\D_H;\IC^{3\times 3})$, and
that $\bJ$ belongs to $\mathbf{W}^{N,q}(\div;\D_H)$,
\item \label{it:Assum_1_ellip_1} $\eps$ and $\mu^{-1}$ satisfy
\begin{align}
\label{eq:alpha_holdall_smooth}
\inf_{0\neq \bzeta\in\C^3}
\essinf_{\bx\in\D_H}\;
\min\left\{
\frac{\Re(\bzeta^\top
e^{\ii\theta}\mu(\bx)^{-1}\bar\bzeta)}{\norm[\C^3]{\bzeta}^2},
\frac{-\Re(\bzeta^\top
e^{\ii\theta}\omega^2\eps(\bx)\bar\bzeta)}{\norm[\C^3]{\bzeta}^2}\right\}\ge\alpha,
\end{align}
\item \label{it:Assum_1_ellip_2}  $\eps$ and $\mu$ satisfy
\begin{align}
\inf_{0\neq \bzeta\in\C^3}
\essinf_{\bx\in\D_H}\;
\min\left\{
\frac{\Re(\bzeta^\top
\mu(\bx)\bar\bzeta)}{\norm[\C^3]{\bzeta}^2},
\frac{\Re(\bzeta^\top
\eps(\bx)\bar\bzeta)}{\norm[\C^3]{\bzeta}^2}\right\}\ge\alpha_s.
\end{align}
\end{enumerate}
\end{assumption}
\begin{remark}
Note that item (\ref{it:Assum_1_ellip_1}) in Assumption \ref{ass:material_trafo_smooth} implies a rotated positivity
property on the permittivity $\mu$ that could be used to replace item (\ref{it:Assum_1_ellip_2}) in the same assumption.
We choose, however, to include both
conditions for brevity and simplicity, since they will be required for different purposes. Item (\ref{it:Assum_1_ellip_1}) allows us
to ensure existence and uniqueness of Maxwell's equations on each one of the uncertain domains (\emph{cf.}~Proposition \ref{prop:existence}), while item (\ref{it:Assum_1_ellip_2}) is required in \cite{AlbertiRoughMaxw} to ensure the unique solvability of an auxiliary problem and in order to prove the smoothness properties of solutions to Maxwell's equations (\emph{cf.}~Theorem \ref{thm:reg} below). 
\end{remark}

\begin{remark}
We recall the identity, $\bm\cC^{N,1}(\Dnul)=\Wsob{N+1}{\Dnul}$ (\emph{cf.}~\cite[Sec.~2.6.4]{DelfZol_2ndEd_2011}),
since we will often employ Sobolev norms of the transformations $\bT\in\frakT$.
\end{remark}

We can then consider the following family of
$\bT$-dependent problems.

\begin{problem}[Maxwell cavity problem on perturbed domains]
\label{prob:weak_T}
For each $\bT\in\frakT$, we seek
$\bE_\bT \in \hncurlbf{\D_\bT}$ such that, with
\begin{align}
\aot{\bU}{\bV}&:= \int_{\D_{\bT}} \mu^{-1}\curl\bU\cdot
\curl\overline{\bV}-\omega^2\eps\bU\cdot\overline{\bV}\dd \bx,\\
\fot{\bV}&:=-\ii\omega\int_{\D_\bT}\bJ\cdot
\overline{\bV}\dd \bx,
\end{align}
for all $\bU$, $\bV\in\hncurlbf{\D_\bT}$, it holds that
\begin{align}
\label{eq:weak_T}
\aot{\bE_\bT}{\bV}=\fot{\bV}\quad\forall\ \bV\in\hncurlbf{\D_\bT}.
\end{align}
\end{problem}

Under Assumption \ref{ass:material_trafo_smooth}, and arguing
as in Section \ref{sec:ex_uniq}, there exists a unique
solution $\bE_\bT\in\hncurlbf{\D_\bT}$ 
to Problem \ref{prob:weak_T} for each $\bT\in\frakT$
satisfying an {\em a priori} bound.

\begin{proposition}
\label{prop:existence_T}
Under Assumption \ref{ass:material_trafo_smooth} and for each
$\bT\in\frakT$, Problem \ref{prob:weak_T} has a
unique solution $\bE_\bT\in\hncurlbf{\D_\bT}$ satisfying
\begin{align}
\label{eq:apriori_T}
\norm[\hcurlbf{\D_\bT}]{\bE_\bT}\le
\frac{\omega}{\alpha}\norm[\Lp{2}{\D_\bT}]{\bJ},
\end{align}  
for $\alpha>0$ as in \eqref{eq:alpha_holdall_smooth}.
\end{proposition}

\begin{proof}
The uniqueness and existence of a solution to Problem
\ref{prob:weak_T} follows from Assumption \ref{ass:material_trafo_smooth}
and the Lax-Milgram lemma as in the proof of Proposition
\ref{prop:existence}. 
Furthermore, the reasoning
in the proof of Proposition \ref{prop:existence} 
also yields
\begin{align}
\norm[\hcurlbf{\D_\bT}]{\bE_\bT}\le
\frac{1}{\alpha}\norm[\hncurlbf{\D_\bT}^*]{F_{\bJ,\bT}}.
\end{align}
The estimate \eqref{eq:apriori_T} then follows from
\begin{align}
\norm[\hncurlbf{\D_\bT}^*]{F_{\bJ,\bT}}\leq
\omega\norm[\Lp{2}{\D_\bT}]{\bJ}
\end{align}
since $\bJ\in\Lp{2}{\D_\bT}$ by assumption.
\end{proof}

\subsection{Pullback to the nominal domain}
\label{ssec:pullback}
As in \cite{AJSZ20,Jerez-Hanckes:2017aa}, rather than considering Maxwell's equations 
on each domain $\{\D_{\bT}\}_{\bT\in\frakT}$, we pull back Problem 
\ref{prob:weak_T} to a family of variational problems set on $\Dnul$
through an appropriate curl-conforming pullback given, for any $\bT\in\frakT$,  as the
extension to $\hcurlbf{\D_\bT}$ of
\begin{align}\label{eq:pullback}
\Phi_{\bT}(\bU):=\d\!\bT^\top(\bU\circ\bT)
\end{align}
for $\bU\in\bm\cC^\infty(\D_\bT;\IC^3)$.

\begin{lemma}[Lemma 2.2 in \cite{Jerez-Hanckes:2017aa}]
\label{lemma:trafo}
Under Assumption \ref{ass:material_trafo_smooth}, for each $\bT\in\frakT$ the
mapping in \eqref{eq:pullback}
can be extended to an
isomorphism
$\Phi_\bT:\hcurlbf{\D_\bT}\to\hcurlbf{\Dnul}$.
In addition,
$\Phi_\bT:\hncurlbf{\D_\bT} \to \hncurlbf{\Dnul}$
is an isomorphism. 
Furthermore, for $\bU\in\hcurlbf{\D_{\bT}}$, it
holds that
\begin{align}
\curl\Phi_\bT(\bU)=\det(\d\!\bT)\d\!\bT^{-1}\curl\bU\circ\bT
\end{align}
in $\Lp{2}{\Dnul}$.
\end{lemma}

We introduce the following family of $\bT$-dependent problems over
$\hncurlbf{\Dnul}$.

\begin{problem}[Nominal Maxwell cavity problem]
\label{prob:weak_nom_T}
For each $\bT\in\frakT$, we seek
$\widehat\bE_\bT\in\hncurlbf{\Dnul}$ such that, with
\begin{align}
\label{eq:ant_fnt_def}
\begin{aligned}
\ant{\widehat\bU}{\widehat\bV}&:=\int_{\Dnul}\left[
\frac{{\mu}^{-1}\circ\bT}{\det(\d\!\bT)}\d\!\bT\curl{\widehat\bU}
\cdot \d\!\bT\curl{\overline{{\widehat\bV}}}
-\omega^2(\eps\circ \bT)\det(\d\!\bT)\d\!\bT^{-\top}{\widehat\bU}\cdot
\d\!\bT^{-\top}\overline{{\widehat\bV}}\right] \d\!\widehat{\bx},\\
\fnt{{\widehat\bV}}&:=-\ii\omega\int_{\Dnul}
\det(\d\!\bT)(\bJ\circ \bT)\cdot
\d\!\bT^{-\top}\overline{{\widehat\bV}}\d\!\widehat{\bx},
\end{aligned}
\end{align}
for all $\widehat\bU$, $\widehat\bV\in\hncurlbf{\Dnul}$, it holds that
\begin{align}
\label{eq:weak_T_nom}
\ant{\widehat\bE_\bT}{\widehat\bV}=\fnt{{\widehat\bV}}
\quad\forall\ \widehat\bV\in\hncurlbf{\Dnul}.
\end{align}
\end{problem}

\begin{remark}\label{rmk:equiv}
Note that the sesquilinear and antilinear forms in \eqref{eq:ant_fnt_def}
may be written as
\begin{align}
\label{eq:ant_fnt_modifed}
\begin{aligned}
\ant{\widehat\bU}{\widehat\bV}&=\int_{\Dnul}\left[\mu_\bT^{-1}\curl{\widehat\bU}
\cdot\curl{\overline{{\widehat\bV}}}
-\omega^2\eps_\bT{\widehat\bU}\cdot
\overline{{\widehat\bV}}\right] \d\!\widehat{\bx},\\
\fnt{{\widehat\bV}}&=-\ii\omega\int_{\Dnul}
\bJ_\bT\cdot
\overline{{\widehat\bV}}\d\!\widehat{\bx},
\end{aligned}
\end{align}
where
\begin{equation}\label{eq:param_p}
\begin{split}
\mu_\bT&:=\det(\d\!\bT)\d\!\bT^{-1}\left(\mu\circ\bT\right)\d\!\bT^{-\top},\\
\eps_\bT&:=\det(\d\!\bT)\d\!\bT^{-1}\left(\eps\circ\bT\right)\d\!\bT^{-\top},\\
\bJ_\bT&:=\det(\d\!\bT)\d\!\bT^{-1}\left(\bJ\circ\bT\right).
\end{split}
\end{equation}
Also, Lemma 3.59 in \cite{Monk:2003aa} yields,
\begin{align}
\div\bJ_\bT=\det(\d\!\bT)\div\left(\bJ\circ\bT\right),
\end{align}
whenever $\bJ\in\hdiv{\D_H}$.
\end{remark}

Due to Lemma \ref{lemma:trafo}, Problem \ref{prob:weak_nom_T} is
equivalent to Problem \ref{prob:weak_T}, in the sense that, for a fixed $\bT\in\frakT$,
$\widehat\bE_\bT\in\hncurlbf{\Dnul}$ is a solution to Problem
\ref{prob:weak_nom_T} if and only if
$\Phi_{\bT}^{-1}(\widehat\bE_\bT)\in\hncurlbf{\D_\bT}$ is a solution to
Problem \ref{prob:weak_T}; we refer to \cite{AJSZ20,Jerez-Hanckes:2017aa} for more
details.

\begin{theorem}
\label{thm:cont}
Under Assumption \ref{ass:material_trafo_smooth} and for all $\bT\in\frakT$, Problem
\ref{prob:weak_nom_T} has a unique solution
$\widehat\bE_\bT\in\hncurlbf{\Dnul}$ satisfying
\begin{align}
\label{eq:apriori_T_nom}
\norm[\hcurlbf{\Dnul}]{\widehat\bE_\bT}\le
C\frac{\omega}{\alpha}\norm[\Lp{2}{\D_\bT}]{\bJ},
\end{align}
where $\alpha>0$ is as in \eqref{eq:alpha_holdall_smooth} and
the constant $C>0$ is independent of $\bT\in\frakT$.
\end{theorem}

\begin{proof}
Under our assumptions, Proposition 2.11 in \cite{AJSZ20} ensures that
\begin{align}
\Re(e^{i\theta}\ant{\widehat\bU}{\widehat\bU})\geq
\alpha\vartheta^3\norm[\hcurlbf{\Dnul}]{\widehat\bU}^2\quad\mbox{and}\quad
\modulo{\ant{\widehat\bU}{\widehat\bV}}\leq C\norm[\hcurlbf{\Dnul}]{\widehat\bU}
\norm[\hcurlbf{\Dnul}]{\widehat\bV}
\end{align}
for all $\bU$, $\bV\in\hcurlbf{\Dnul}$ and all $\bT\in\frakT$, where
the positive continuity constant $C$ depends on $\vartheta$ but is
independent of $\bT\in\frakT$. The complex Lax-Milgram Lemma then
ensures the existence and uniqueness of the solution to Problem
\ref{prob:weak_nom_T} for each $\bT\in\frakT$ and the {\em a priori}
bound
\begin{align}
\norm[\hcurlbf{\Dnul}]{\widehat\bE_\bT}\leq \vartheta^{-3}\frac{\omega}{\alpha}\norm[\Lp{2}{\Dnul}]{\bJ_\bT},
\end{align}
where $\bJ_\bT$ is as in Remark \ref{rmk:equiv}.
Assumption \ref{ass:material_trafo_smooth} and a change of variables yield,
\begin{align}\label{eq:rhs_bound}
\norm[\Lp{2}{\Dnul}]{\bJ_\bT}\leq\vartheta^{-2}\norm[\Lp{2}{\Dnul}]{\bJ\circ\bT}
\leq\vartheta^{-\frac{5}{2}}\norm[\Lp{2}{\D_\bT}]{\bJ},
\end{align}
and \eqref{eq:apriori_T_nom} follows.
\end{proof}

\begin{remark}\label{rmk:lip_assumption}
Note that we have not yet made use of the smoothness properties of the domain $\Dnul$ nor of
the parameters $\varepsilon$, $\mu$ and $\bJ$ nor of the transformations $\bT\in\frakT$ specified
in Assumption \ref{ass:material_trafo_smooth}. In fact, all of the results in Sections \ref{ssec:pert_dom} and
\ref{ssec:pullback} hold with $N=0$ in Assumption \ref{ass:material_trafo_smooth}.
\end{remark}

\subsection{Spatial regularity}
\label{ssec:regularity}

We continue by recalling a regularity statement for the solution of
\eqref{eq:firstorder} from \cite{AlbertiRoughMaxw} 
(also, see 
\cite{MR0233555,MR657071} 
for earlier but less sharp
results establishing $\boldsymbol{H}^1$-regularity).

\begin{theorem}[Theorem 9 in \cite{AlbertiRoughMaxw}]
\label{thm:reg}
Fix $N\in\IN$, $q>3$ and $\alpha_s>0$ and let $\D\subset\IR^3$ be an open and bounded domain of class $\cC^{N,1}$. 
Assume the parameters $\varepsilon$, $\mu$ and $\bJ$ to satisfy
\begin{gather}\label{eq:regularity_cond}
\begin{gathered}
\eps,\ \mu\in \bm{W}^{N,q}(\D;\IC^{3\times 3}),
\quad \bJ\in\mathbf{W}^{N-1,q}(\div;\D),\\
\inf_{0\neq \bzeta\in\C^3}
\essinf_{\bx\in\D}\;
\min\left\{
\frac{\Re(\bzeta^\top
\mu(\bx)\bar\bzeta)}{\norm[\C^3]{\bzeta}^2},
\frac{\Re(\bzeta^\top
\eps(\bx)\bar\bzeta)}{\norm[\C^3]{\bzeta}^2}\right\}\ge\alpha_s,
\end{gathered}
\end{gather}
and that the imaginary parts of $\varepsilon$ and $\mu$ are symmetric and
let $R>0$ be such that
$$R>\max(\norm[\bm{W}^{N,q}(\D;\IC^{3\times 3})]{\eps},\norm[\bm{W}^{N,q}(\D;\IC^{3\times 3})]{\mu}).$$
Then, there exists a positive constant $C$ depending on $R$,
$\omega$, $q$, $\alpha_s$  and $\D$ such that any weak solution pair
$\bE$, $\bH\in\hcurlbf{\D}$ of \eqref{eq:firstorder} belong to
$\mathbf{W}^{N,q}(\D)$ and satisfy 
\begin{align}
\label{eq:regularity_stated}
\norm[\mathbf{W}^{N,q}(\D)]{\bE}+\norm[\mathbf{W}^{N,q}(\D)]{\bH}
&\le C(\norm[\Lp{2}{\D}]{\bE} + \norm[\Lp{2}{\D}]{\bH} +\norm[\mathbf{W}^{N-1,q}(\div;\D)]{\bJ}).
\end{align}
\end{theorem}

\begin{remark}
The fact that the constant $C$ in Theorem \ref{thm:reg} depends on
$\mu$ and $\eps$ only through $R>0$ is not explicitly stated in
\cite{AlbertiRoughMaxw}, but follows from the proof of
\cite[Thm.~9]{AlbertiRoughMaxw} and the references therein.
\end{remark}
We now adapt Theorem \ref{thm:reg} to our setting and prove,
under Assumption \ref{ass:material_trafo_smooth},
a uniform (on $\bT\in\frakT$) smoothness result for the solution of Problem
\ref{prob:weak_nom_T}.

\begin{theorem}\label{thm:unif_reg}
Let Assumption \ref{ass:material_trafo_smooth} hold. 
Then, for each $\bT\in\frakT$, the solution
$\widehat\bE_\bT\in\hncurlbf{\Dnul}$ to Problem \ref{prob:weak_nom_T}
belongs to $\Wsob{N,q}{\curl;\Dnul}$, with the bound
\begin{align}\label{eq:regularitycurl}
\norm[\Wsob{N,q}{\curl;\Dnul}]{\widehat\bE_\bT}\le C
\norm[\Wsob{N-1,q}{\div;\D_H}]{\bJ},
\end{align}
where the constant $C>0$ depends on $\frakT$ but is
independent of $\bT\in\frakT$.
\end{theorem}

\begin{proof}
Theorem \ref{thm:cont} yields the existence and uniqueness of a
solution $\widehat\bE_\bT\in\hncurlbf{\Dnul}$ to Problem
\ref{prob:weak_nom_T} for each $\bT\in\frakT$. Due to Remark
\ref{rmk:firstorder} and recalling the notation therein introduced, we have
\begin{align}
\widehat\bH_\bT:=\imath\omega^{-1}\mu_\bT^{-1}\curl\widehat\bE_\bT,
\end{align}
then $\widehat\bE_\bT$ and $\widehat\bH_\bT$ are weak solution pair 
to the system:
\begin{align}
\label{eq:firstorder_Dnul}
\begin{aligned}
\curl \widehat\bE_\bT + \ii\omega \mu_\bT \widehat\bH_\bT &= \bnul &&\text{in }\Dnul,\\
\ii\omega\eps_\bT\widehat\bE_\bT -\curl \widehat\bH_\bT &=-\bJ_\bT&&\text{in }\Dnul.
\end{aligned}
\end{align}
The product rule then yields
\begin{align}\label{eq:par_T_est}
\begin{aligned}
\norm[\Wsob{N,q}{\Dnul}]{\mu_\bT}&\leq
C(\bT)\norm[\Wsob{N,q}{\Dnul}]{\mu\circ \bT},\\
\norm[\Wsob{N,q}{\Dnul}]{\eps_\bT}&\leq
C(\bT)\norm[\Wsob{N,q}{\Dnul}]{\eps\circ \bT},\\
\norm[\Wsob{N-1,q}{\Dnul}]{\bJ_\bT}&\leq
C(\bT)\norm[\Wsob{N-1,q}{\Dnul}]{\bJ\circ \bT},\\
\norm[\Wsob{N-1,q}{\Dnul}]{\div\bJ_\bT}&\leq
C(\bT)\norm[\Wsob{N-1,q}{\Dnul}]{\div\bJ_\bT\circ \bT},
\end{aligned}
\end{align}
where the constant $C(\bT)>0$ depends continuously on
$\bT\in\frakT\Subset\bm\cC^{N,1}(\Dnul)$.
Repeated application of the chain rule
(\emph{cf.}~\cite[Lemma 1]{PCPAR_1972_b} and
\cite[Lemma 3]{PCPAR_1972}) yields,
\begin{align}
\norm[\Wsob{N,q}{\Dnul}]{\eps\circ \bT}
&\leq C \left(1+\norm[\Wsob{N,\infty}{\Dnul}]{\bT}\right)^N
\norm[\lp{\infty}{\Dnul}]{\det (\d\!\bT)^{-1}}^\frac{1}{q}
\norm[\Wsob{N,q}{\D_\bT}]{\eps}\\
&\leq C \left(1+\norm[\Wsob{N,\infty}{\Dnul}]{\bT}\right)^N
\norm[\lp{\infty}{\Dnul}]{\det (\d\!\bT)^{-1}}^\frac{1}{q}
\norm[\Wsob{N,q}{\D_H}]{\eps},\label{eq:eps_est}
\end{align}
where the constant $C>0$ is independent of $\bT\in\frakT$.
Analogously,
\begin{align}\label{eq:mu_J_divJ_est}
\begin{aligned}
\norm[\Wsob{N,q}{\Dnul}]{\mu\circ \bT}
&\leq C \left(1+\norm[\Wsob{N,\infty}{\Dnul}]{\bT}\right)^N
\norm[\lp{\infty}{\Dnul}]{\det (\d\!\bT)^{-1}}^\frac{1}{q}
\norm[\mathbf{W}^{N,q}(\D_H)]{\mu},\\
\norm[\Wsob{N-1,q}{\Dnul}]{\bJ\circ \bT}
&\leq C \left(1+\norm[\Wsob{N-1,\infty}{\Dnul}]{\bT}\right)^{N-1}
\norm[\lp{\infty}{\Dnul}]{\det (\d\!\bT)^{-1}}^\frac{1}{q}
\norm[\mathbf{W}^{N-1,q}(\D_H)]{\bJ},\\
\norm[\Wsob{N-1,q}{\Dnul}]{\div\bJ\circ \bT}
&\leq C \left(1+\norm[\Wsob{N-1,\infty}{\Dnul}]{\bT}\right)^{N-1}
\norm[\lp{\infty}{\Dnul}]{\det (\d\!\bT)^{-1}}^\frac{1}{q}
\norm[\mathbf{W}^{N-1,q}(\D_H)]{\div\bJ}.
\end{aligned}
\end{align}
The combination of the estimates in \eqref{eq:par_T_est} with those in
\eqref{eq:eps_est} and \eqref{eq:mu_J_divJ_est} imply that 
\begin{align}
\eps_\bT,\ \mu_\bT\in \bm{W}^{N,q}(\Dnul;\IC^{3\times 3}),
\quad \bJ_\bT\in\mathbf{W}^{N-1,q}(\div;\Dnul).
\end{align}
Furthermore, 
for every $\bm\zeta\in\IC^3$ and $\widehat{\bx}\in\Dnul$ we have,
due to Assumption \ref{ass:material_trafo_smooth}, that
\begin{align}\label{eq:ellip_alphas}
\modulo{\bm\zeta^\top\left(\d\!\bT^{-1}\left(\eps\circ\bT\right)\d\!\bT^{-\top}\right)(\widehat{\bx})\bm\zeta}&\geq\alpha_s\norm[\IC^3]{\d\!\bT^{-\top}(\widehat{\bx})\bm\zeta}^2\\
&\geq\alpha_s\norm[\IC^{3\times 3}]{\d\!\bT(\widehat{\bx})}^{-2}\norm[\IC^3]{\bm\zeta}^2\geq c\alpha_s\vartheta^2\norm[\IC^3]{\bm\zeta}^2,
\end{align}
with an analogous computation for
$\modulo{\bm\zeta^\top(\d\!\bT^{-1}\mu\circ\bT\d\!\bT^{-\top})(\widehat{\bx})\bm\zeta}$,
where the constant $c>0$ follows from the equivalence of norms over finite dimensional spaces
and is independent of $\bT\in\frakT$. 
Therefore, there holds that
\begin{align}
\inf_{0\neq \bzeta\in\C^3}
\essinf_{\bx\in\Dnul}\;
\min\left\{
\frac{\Re(\bzeta^\top
\mu_\bT(\widehat{\bx})\bar\bzeta)}{\norm[\C^3]{\bzeta}^2},
\frac{\Re(\bzeta^\top
\eps_\bT(\widehat{\bx})\bar\bzeta)}{\norm[\C^3]{\bzeta}^2}\right\}\ge c\vartheta^3\alpha_s,
\end{align}
where the additional power in $\vartheta$ follows from the bound
for $\det(\bT)$ in Assumption \ref{ass:material_trafo_smooth}. Theorem \ref{thm:reg}
then ensures that $\widehat\bE_\bT$ and $\widehat\bH_\bT$ belong
to $\Wsob{N,q}{\Dnul}$ together with the bound
\begin{align}
\label{eq:regularity_T}
\norm[\Wsob{N,q}{\Dnul}]{\widehat\bE_\bT}+\norm[\Wsob{N,q}{\Dnul}]{\widehat\bH_\bT}
&\le C(\norm[\Lp{2}{\Dnul}]{\widehat\bE_\bT} + \norm[\Lp{2}{\Dnul}]{\widehat\bH_\bT} +\norm[\Wsob{N-1,q}{\div ;\Dnul}]{\bJ_\bT}),
\end{align}
where the constant $C>0$ depends on
$R>\max(\norm[\bm{W}^{N,q}(\Dnul;\IC^{3\times 3})]{\eps_\bT},\norm[\bm{W}^{N,q}(\Dnul;\IC^{3\times 3})]{\mu_\bT})$,
$\omega$, $q$, $\alpha_s$, $\vartheta$ and $\Dnul$.
The product rule and the previous estimates for $\mu_\bT$ \eqref{eq:par_T_est} 
then yield
\begin{align}\label{eq:curlET_HT}
\begin{aligned}
\norm[\Wsob{N,q}{\Dnul}]{\curl\widehat\bE_\bT}=\omega\norm[\Wsob{N,q}{\Dnul}]{\mu_\bT\widehat\bH_\bT}\leq C(\bT)\norm[\Wsob{N,q}{\Dnul}]{\widehat\bH_\bT},\\
\norm[\Lp{2}{\Dnul}]{\widehat\bH_\bT}=\omega^{-1}\norm[\Lp{2}{\Dnul}]{\mu_\bT^{-1}\curl\widehat\bE_\bT}\leq c(\bT)\norm[\Lp{2}{\Dnul}]{\curl\widehat\bE_\bT},
\end{aligned}
\end{align}
where the constants $C(\bT)>0$ and $c(\bT)>0$ depend on $\mu\in\Wsob{N,q}{\D_H;\IC^{3\times 3}}$ and, continuously,
on $\bT\in\frakT\Subset\bm\cC^{N,1}(\Dnul)$.
A combination of the estimates in \eqref{eq:regularity_T},
\eqref{eq:curlET_HT} and \eqref{eq:apriori_T_nom} then yields
\begin{align}
&\norm[\Wsob{N,q}{\Dnul}]{\widehat\bE_\bT}+\norm[\Wsob{N,q}{\Dnul}]{\curl\widehat\bE_\bT}\\
&\le (1+C(\bT))\left(\norm[\Wsob{N,q}{\Dnul}]{\widehat\bE_\bT}+\norm[\Wsob{N,q}{\Dnul}]{\widehat\bH_\bT}\right)\\
&\le C(1+C(\bT))\left(\norm[\Lp{2}{\Dnul}]{\widehat\bE_\bT} + \norm[\Lp{2}{\Dnul}]{\widehat\bH_\bT} +\norm[\Wsob{N-1,q}{\div ;\Dnul}]{\bJ_\bT}\right)\\
&\le C(1+c(\bT))(1+C(\bT))\left(\norm[\Lp{2}{\Dnul}]{\widehat\bE_\bT} + \norm[\Lp{2}{\Dnul}]{\curl\widehat\bE_\bT} +\norm[\Wsob{N-1,q}{\div ;\Dnul}]{\bJ_\bT}\right)\\
&\leq C(1+c(\bT))(1+C(\bT))(1+C_\bJ(\bT))\frac{C_\vartheta}{\alpha}\norm[\Wsob{N-1,q}{\div ;\D_H}]{\bJ},\label{eq:fin_reg_est}
\end{align}
where the constant $C>0$ is as in \eqref{eq:regularity_T}, $C(\bT)>0$
and $c(\bT)>0$ are as in \eqref{eq:curlET_HT}, $C_\vartheta>0$ and
$\alpha>0$ are as in \eqref{eq:apriori_T_nom}, with the dependence on
$\omega>0$ has been absorbed by the constant $C$.
$C_\bJ(\bT)>0$ follows from combining the estimates for $\bJ_\bT$
and $\div\bJ_\bT$ in \eqref{eq:par_T_est} and \eqref{eq:mu_J_divJ_est}.
The compactness of $\frakT$ in $\bm\cC^{N,1}(\Dnul)$ together with the continuous
dependence of the constants on $\bT\in\bm\cC^{N,1}(\Dnul)$ then ensures a
uniform bound on $\bT\in\frakT$ in \eqref{eq:fin_reg_est}.
\end{proof}

Theorem \ref{thm:unif_reg} will ensure a uniform bound on the
convergence rates of finite element approximations of the fields
$\widehat\bE_\bT$ and will allow for the design of multilevel algorithms
in the approximation of the expectation of the mapping
$\bT\mapsto\widehat\bE_{\bT}$.
We continue our analysis by studying the approximation of solutions to
Problem \ref{prob:weak_nom_T} by the finite element method.

\section{Discrete solution}
\label{sec:DiscSol}
To compute a discrete finite element approximation to
the solution of Problem \ref{prob:weak_nom_T}, the test
and trial space $\hncurlbf{\D}$ in \eqref{eq:weak_T_nom}
is to be replaced with a finite dimensional subspace. Since the PDE coefficients defining
$\ant{\cdot}{\cdot}$ and $\fnt{\cdot}$ in Problem
\ref{prob:weak_nom_T} are not constant, the corresponding
stiffness matrix must itself be approximated by means of
numerical quadrature on each element of the mesh, which
introduces a further source of error. Following
\cite{AJSZ20,AJ_2020}, in this section we discuss
existence, uniqueness and the approximation properties
of such a discrete solution. However, first we provide a framework 
to accommodate non-polyhedral domains.

\subsection{Pullback to a polyhedral domain}
The regularity result in Theorem \ref{thm:unif_reg} requires the
domain $\Dnul\subset\IR^3$ to possess a $\cC^{N,1}$-boundary for
some $N\in\IN$, which precludes polyhedral domains and the usage
of standard tetrahedral meshes. We circumvent this problem by
  pulling back the respective Maxwell problems to a polyhedral domain,
  henceforth referred to as the \emph{computational domain}, satisfying
  the following assumption.

  \begin{assumption}
    \label{ass:computational_domain}
    Let $\Dcomp\subset\IR^3$ be a polyhedral domain, 
    referred to as the \emph{computational domain}.
    There exists a bijective bi-Lipschitz map $\widehat\bT$ mapping $\Dcomp$ onto
    $\Dnul$. 
    For $n\in\IN$, there are two sets of
    pairwise disjoint subsets of $\Dcomp$ and $\Dnul$,
    $\{\Dcomp_j\}_{j=1}^{n}$ and $\{\Dnul_j\}_{j=1}^{n}$,
    respectively, such that the domains $\{\Dcomp_j\}_{j=1}^{n}$ are
    polyhedral, the domains $\{\Dnul_j\}_{j=1}^{n}$ are
    Lipschitz, and it holds that
    \begin{gather}
      \Dnul=\mbox{int}\left(\bigcup_{j=1}^n\overline{\Dnul_j}\right),\qquad
      \Dcomp=\mbox{int}\left(\bigcup_{j=1}^n\overline{\Dcomp_j}\right),\\
      \widehat\bT|_{\Dcomp_j}:\Dcomp_j\to\Dnul_j,\quad
      \widehat\bT|_{\Dcomp_j}\in \bm\cC^{N,1}(\Dcomp_j),\quad
      \widehat\bT^{-1}|_{\Dnul_j}\in
      \bm\cC^{N,1}(\Dnul_j) \quad\forall
      j\in\{1,\hdots,n\},
    \end{gather}
    where $N\in\IN$ is as in Assumption \ref{ass:material_trafo_smooth}.
  \end{assumption}    
    For each $\bT\in\frakT$, we introduce the mapping
    \begin{align}
      \widetilde\bT:=\bT\circ\widehat\bT:\Dcomp\to\D_\bT,
    \end{align}
    and the set of admissible computational perturbations
    $\widetilde\frakT:=\set{\widetilde\bT}{\widetilde\bT:=\bT\circ\widehat\bT\;\forall\,\bT\in\frakT}$. Figure \ref{fig:setting} illustrates the setting of Assumption \ref{ass:computational_domain}.

\begin{definition}\label{def:pw}
  Let $\Dcomp$ be as in Assumption \ref{ass:computational_domain}.  For
  $m\in\N$ and $p\in [1,\infty]$, we introduce
  \begin{equation*}
    \Wsobpw{m,p}{\Dcomp}:=\set{\bU\in
      \Lp{p}{\Dcomp}}{\bU|_{\Dcomp_j}\in \Wsob{m,q}{\Dcomp_j}~\forall\, j\, \in\{1,\hdots,n\}}
  \end{equation*}
  with
  \begin{equation*}
    \norm[\Wsobpw{m,p}{\Dcomp}]{\bU}:=\left(\sum_{i=j}^n
      \norm[\Wsob{m,p}{\Dcomp_j}]{\bU|_{\Dcomp_j}}^p
    \right)^\frac{1}{p},
  \end{equation*}
  if $p<\infty$ and the usual adjustment in case $p=\infty$.
\end{definition}

\begin{definition}\label{def:pw_curl}
Let $\Dcomp$ be as in Assumption \ref{ass:computational_domain}. For
$m\in\N$ and $p\in [1,\infty]$, we introduce
\begin{equation*}
\Wsobpw{m,p}{\curl;\Dcomp}:=
\set{\bU\in\Wsobpw{m,p}{\Dcomp}}{\curl\bU\in \Wsobpw{m,p}{\Dcomp}},
\end{equation*}
with
\begin{equation*}
\norm[\Wsobpw{m,p}{\curl;\Dcomp}]{\bU}:=\left(\norm[\Wsobpw{m,p}{\Dcomp}]{\bU}^p+\norm[\Wsobpw{m,p}{\Dcomp}]{\curl\bU}^p\right)^\frac{1}{p},
\end{equation*}
if $p<\infty$ and the usual adjustment in case $p=\infty$. 
Furthermore, 
for any $m\in\IN$, we set
  \begin{align}
  \hscurlbfpw{m}{\Dcomp}:=\Wsobpw{m,2}{\curl;\Dcomp}.
  \end{align}
\end{definition}

\begin{remark}\label{rmk:comp}
    Under Assumptions \ref{ass:material_trafo_smooth} and
    \ref{ass:computational_domain}, compactness of
      $\frakT \Subset \bm\cC^{N,1}(\Dnul)$ and continuity of
      $\bT\mapsto \bT\circ\widehat\bT:\bm\cC^{N,1}(\Dnul)\to
      \Wsobpw{N+1,\infty}{\Dcomp}$ imply the compactness of $\widetilde\frakT=
      \set{\bT\circ\widehat\bT}{\bT\in\frakT}\subseteq
      \Wsobpw{N+1,\infty}{\Dcomp}$.
\end{remark}
\begin{figure}
  \begin{center}
    \begin{tikzpicture}[scale=1.2]
      \draw[thick] (-2.5,-1) -- (-0.5,-1) -- (-0.5,1) -- (-2.5,1) --
      (-2.5,-1); \draw (-2.5,-1) -- (-0.5,1); \draw (-2.5,1) --
      (-0.5,-1); \node at (-2.8,1) {\small $\Dcomp$};

      \node at (-1.5+0.4,0+0){\small $\Dcomp_1$}; \node at
      (-1.5+0,0+0.4){\small $\Dcomp_2$}; \node at (-1.5-0.4,0+0){\small
        $\Dcomp_3$}; \node at (-1.5+0,0-0.4){\small $\Dcomp_4$};
  
      \draw[thick] (3,0) circle (1cm); \draw (3+0.707,0+0.707) --
      (3-0.707,0-0.707); \draw (3+0.707,0-0.707) -- (3-0.707,0+0.707);
      \node at (1.7,1) {\small $\Dnul$};
  
      \node at (3+0.4,0+0){\small $\Dnul_1$}; \node at
      (3+0,0+0.4){\small $\Dnul_2$}; \node at
      (3-0.4,0+0){\small $\Dnul_3$}; \node at
      (3+0,0-0.4){\small $\Dnul_4$};

      \draw[thick, xshift = 7.5cm] plot [smooth cycle, tension = 0.6]
      coordinates {(1,0) (0.954,0.3) (0.507,0.507) (0.3,0.954) (0,1)
        (-0.707,0.707) (-1,0) (-0.954,-0.3) (-0.507,-0.507)
        (-0.3,-0.954) (0,-1) (0.707,-0.707) }; \node at (6.2,1)
      {\small $\D_\bT$};

      \draw [-,xshift = 7.5cm,in=80,out=-120] (0.507,0+0.507) to (-0.507,0-0.507);
      \draw [xshift = 7.5cm,in=-30,out=120] (0.707,0-0.707) to (-0.707,0+0.707);

      \node at (7.5+0.55+0.05,0+0.08){\small $\bT(\Dnul_1)$}; \node at
      (7.5+0,0+0.55+0.05){\small $\bT(\Dnul_2)$}; \node at
      (7.5-0.55+0.05,0+0.15){\small $\bT(\Dnul_3)$}; \node at
      (7.5+0.05,0-0.55+0.05){\small $\bT(\Dnul_4)$};

      \draw [->,in=150,out=30] (0,0) to (1.5,0); \node at
      (0.75,0.5){\small $\widehat\bT$};

      \draw [->,in=150,out=30] (4.5,0) to (6,0); \node at
      (5.25,0.5){\small $\bT\in\frakT$};

      \draw [->,in=210,out=-30] (0.,-.5) to (6,-.5); \node at
      (3,-1.7){\small $\widetilde\bT\in\widetilde\frakT$};
    \end{tikzpicture}
  \end{center}
  \caption{ Setting for domain transformations. The domains
      $\Dnul$ and $\Dcomp$, as well as the transformation
      $\widehat\bT:\Dcomp\to\Dnul$ are considered fixed. For a family
      of domain transformations $\frakT$, the physical domains are
      given as $\D_\bT=\bT(\Dnul)$ for $\bT\in\frakT$, or
      equivalently as $\widetilde\bT(\Dcomp)$ for
      $\widetilde\bT=\bT\circ\widehat\bT\in\widetilde\frakT:=\frakT\circ\widehat\bT$.  The
      pullback solution on the smooth nominal domain $\Dnul$
      can be shown to belong to a certain regularity class. This
      allows to deduce convergence rates for finite element
      approximations of the pullback solutions computed on the
      polyhedral domain $\Dcomp$.}\label{fig:setting}
\end{figure}
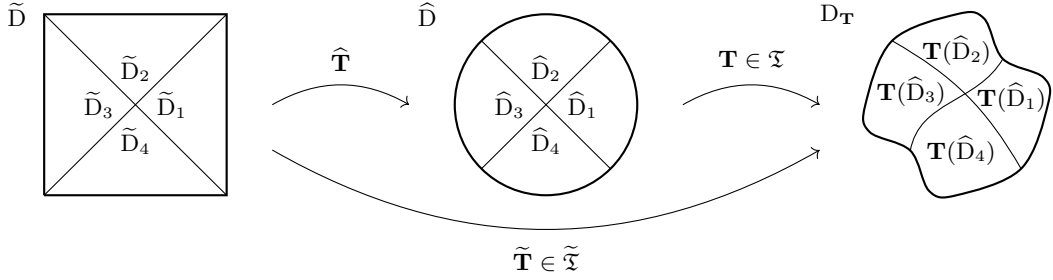

\begin{lemma}\label{lemma:regT}
  Let Assumptions \ref{ass:computational_domain} hold and let
  $\bU\in\hncurlbf{\Dnul}\cap\Wsob{m,p}{\curl;\Dnul}$ for $m\in\IN$ with $m\leq N$, where
  $N\in\IN$ is as in Assumption \ref{ass:computational_domain}, and $p>1$.
  Then, with $\widehat\bT:\Dcomp\to\Dnul$ as in Assumption
  \ref{ass:computational_domain} and
  $\Phi_{\widehat\bT}:\hncurlbf{\Dnul}\to\hncurlbf{\Dcomp}$ as
  in \eqref{eq:pullback} and Lemma \ref{lemma:trafo}, it holds that
  $\Phi_{\widehat\bT}\bU$ belongs to
  $\hncurlbf{\Dcomp}\cap\Wsobpw{m,p}{\curl;\Dcomp}$, with
  \begin{align}\label{eq:regT_lem}
    \norm[\Wsobpw{m,p}{\curl;\Dcomp}]{\Phi_{\widehat\bT}\bU}\leq C\norm[\Wsob{m,p}{\curl;\Dnul}]{\bU},
  \end{align}
  where the constant $C>0$ is independent of
  $\bU\in\hncurlbf{\Dnul}\cap\Wsob{m,p}{\curl;\Dnul}$.
\end{lemma}
\begin{proof}
  Take an arbitrary
  $\bU\in\hncurlbf{\Dnul}\cap\Wsob{m,p}{\curl;\Dnul}$. Lemma
  \ref{lemma:trafo} gives,
  \begin{align*}
    \Phi_{\widehat\bT}\bU=\d\!\widehat\bT^\top\bU\circ\widehat\bT\quad\mbox{and}\quad \curl\Phi_{\widehat\bT}\bU=\det(\d\!\widehat\bT)\d\!\widehat\bT^{-1}\curl\bU\circ\widehat\bT,
  \end{align*}
  and $\Phi_{\widehat\bT}\bU\in\hncurlbf{\Dcomp}$.
  Fix $j\in\{1,\dots,n\}$. 
By Assumption \ref{ass:computational_domain},
  $\widehat\bT\in\Wsob{N+1,\infty}{\Dcomp_j}$,
  according to \cite[Lemma 1]{PCPAR_1972_b}---also see \cite[Lemma 3]{PCPAR_1972}---it holds that
  $\bU|_{\Dnul_j}\circ \widehat\bT,\ \curl\bU|_{\Dnul_j}\circ \widehat\bT\in \Wsob{m,p}{\Dcomp_j}$.
  Repeatedly applying the chain rule---as in the proof of Theorem
  \ref{thm:unif_reg}---yields the existence of $C>0$, 
  independent of $\bU\in\hncurlbf{\Dnul}$, 
  such that
  \begin{align}
    \norm[\Wsob{m,p}{\Dcomp_j}]{\bU\circ\widehat\bT}
    &\le C \left(1+\norm[\Wsobpw{m,\infty}{\Dcomp}]{\widehat\bT}\right)^m
      \norm[\lp{\infty}{\Dcomp}]{\det(\d\!\widehat\bT)^{-1}}^\frac{1}{p}
      \norm[\Wsob{m,p}{\Dnul_j}]{\bU},\\
    \norm[\Wsob{m,p}{\Dcomp_j}]{\curl\bU\circ \widehat\bT}
    &\le C \left(1+\norm[\Wsobpw{m,\infty}{\Dcomp}]{\widehat\bT}\right)^m
      \norm[L^\infty(\Dcomp)]{\det(\d\!\widehat\bT)^{-1}}^\frac{1}{p}
      \norm[\Wsob{m,p}{\Dnul_j}]{\curl\bU}.
  \end{align}
  Furthermore
  $\d\!\widehat\bT\in \Wsob{N,\infty}{\Dcomp_j;\C^{3\times 3}}$ and
  therefore $\Phi_{\widehat\bT}\bU\vert_{\Dcomp_j}\in\Wsob{m,p}{\Dcomp_j}$ for all
  $j\in\{1,\dots,n\}$. Analogously, we have that
  $\det(\d\!\widehat\bT)\d\!\widehat\bT^{-1}\in\Wsob{N,\infty}{\widetilde \D_j;\C^{3\times
      3}}$ (upon recalling
  that $\det(\d\!\widehat\bT)\d\!\widehat\bT^{-1}$ is the cofactor
  matrix of $\d\!\widehat\bT$) and
  therefore $\curl\Phi_{\widehat\bT}\bU\vert_{\Dcomp_j}\in\Wsob{m,p}{\Dcomp_j}$ for all
  $j\in\{1,\dots,n\}$. The estimate \eqref{eq:regT_lem} then follows by the
  product rule.
\end{proof}

\begin{problem}[Computational Maxwell cavity problem]
\label{prob:weak_comp_T}
For each $\bT\in\frakT$, we seek
$\widetilde\bE_{\bT}$ $\in \hncurlbf{\Dcomp}$ such that with
\begin{align}
\label{eq:ant_fnt_comp_def}
\begin{aligned}
\act{\widetilde\bU}{\widetilde\bV}&:=\int_{\Dcomp}\left[
\mu_{\widetilde\bT}^{-1}\curl{\widetilde\bU}
\cdot \curl{\overline{{\widetilde\bV}}}
-\omega^2\eps_{\widetilde\bT}{\widetilde\bU}\cdot
\overline{{\widetilde\bV}} \right]\d\!\widetilde{\bx}\\
\fct{{\widetilde\bV}}&:=-\ii\omega\int_{\Dcomp}
\bJ_{\widetilde\bT}\cdot
\overline{{\widetilde\bV}}\d\!\widetilde{\bx},
\end{aligned}
\end{align}
for all $\widetilde\bU$, $\widetilde\bV\in\hncurlbf{\Dcomp}$,
it holds that
\begin{align}
\label{eq:weak_T_copm}
\act{\widetilde\bE_\bT}{\widetilde\bV}=\fct{{\widetilde\bV}}
\quad\forall\ \widetilde\bV\in\hncurlbf{\Dcomp},
\end{align}
where $\widetilde\bT:=\bT\circ\widehat\bT$, $\widehat\bT:\Dcomp\to\Dnul$ is as in Assumption \ref{ass:computational_domain} and $\mu_{\widetilde\bT}$,
$\eps_{\widetilde\bT}$ and $\bJ_{\widetilde\bT}$ are as in Remark \ref{rmk:equiv}.
\end{problem}

\begin{theorem}\label{thm:unif_pw_reg}
  Let Assumptions \ref{ass:material_trafo_smooth}
  and \ref{ass:computational_domain} hold. Then, for each
  $\bT\in\frakT$, there is a unique solution $\widetilde\bE_\bT\in\hncurlbf{\Dcomp}$ to
  Problem \ref{prob:weak_comp_T} that satisfies
  $\widetilde\bE_\bT\in\hncurlbf{\Dcomp}\cap\Wsobpw{N,q}{\Dcomp}$, where $N\in\IN$ and
  $q\geq 3$ are as in Assumption \ref{ass:material_trafo_smooth}, with the bound
  \begin{align}\label{eq:hnq_curl_ET}
    \norm[\Wsobpw{N,q}{\curl;\Dcomp}]{\widetilde\bE_\bT}\le C
    \norm[\Wsob{N-1,q}{\div;\D_H}]{\bJ},
  \end{align}
  where the constant $C>0$ depends on $\frakT$ but is independent of
  $\bT\in\frakT$.
\end{theorem}
\begin{proof}
  Under our assumptions, we may repeat our
  analysis in Sections \ref{ssec:pert_dom} through
  \ref{ssec:regularity} on $\Dcomp$ instead of on $\Dnul$, so
  that for each $\widetilde\bT\in\widetilde\frakT$ there is a unique
  $\widetilde\bE_{\bT}\in\hncurlbf{\Dcomp}$ that solves Problem \ref{prob:weak_comp_T}.
  Moreover, as before, it holds that $\widetilde\bE_{\bT}\equiv\Phi_{\widehat\bT}\widehat\bE_\bT$,
  where $\widehat\bT:\Dcomp\to\Dnul$ is as in Assumption \ref{ass:computational_domain},
  $\Phi_{\widehat\bT}:\hncurlbf{\Dnul}\to\hncurlbf{\Dcomp}$ is as in Lemma \ref{lemma:trafo}
  and $\widehat\bE_\bT\in\hncurlbf{\Dnul}$ is the solution of Problem \ref{prob:weak_nom_T}
  (\emph{cf.}~\cite{AJSZ20,Jerez-Hanckes:2017aa} for more details).
  Then, Theorem \ref{thm:unif_reg} yields
  $\widehat\bE_{\bT}\in\Wsob{N,q}{\curl;\Dnul}$ with
  the bound
  \begin{align}\label{eq:regularitycurl_intermediate}
    \norm[\Wsob{N,q}{\curl;\Dnul}]{\widehat\bE_{\bT}}\le C
    \norm[\Wsob{N-1,q}{\div;\D_H}]{\bJ},
  \end{align}
  where the positive constant $C$ depends on $\widetilde\frakT$ but is
  independent of $\widetilde\bT\in\widetilde\frakT$. Lemma
  \ref{lemma:regT} then yields a positive
  constant $C$ such that, 
  \begin{align}
    \norm[\Wsobpw{N,q}{\curl;\Dcomp}]{\Phi_{\widehat\bT}\widehat\bE_\bT}\leq C
    \norm[\Wsob{N,q}{\curl;\Dnul}]{\widehat\bE_{\bT}},
  \end{align}
  for each $\bT\in\frakT$, so that the result then follows from the equivalence $\widetilde\bE_{\bT}\equiv\Phi_{\widehat\bT}\widehat\bE_\bT$.
\end{proof}

\subsection{Finite elements}
\label{ssec:fe}
We now introduce discretization spaces for 
Problem \ref{prob:weak_comp_T}. 
We shall consider a sequence
of affine meshes $\{\tau_{h_i}\}_{i\in\IN}$, indexed by their
positive mesh-sizes, on the computational domain $\Dcomp$.

\begin{assumption}
\label{ass:tau}
Let $\Dcomp$ be as in Assumption \ref{ass:computational_domain}.
There exist constants $s\in (0,1)$, $C_1>0$, $C_2>0$ and a sequence of meshes
  $\{\tau_{h_i}\}_{i\in\IN}$ such that for all $i\in\IN$ the following
  conditions hold:
\begin{enumerate}
\item $\tau_{h_i}$ is a set of pairwise disjoint tetrahedrons
  generally denoted $K$ such that
\begin{align}
\Dcomp=\mbox{int}\left(\bigcup_{K\in\tau_{h_i}}\overline{K}\right),
\end{align}
\item\label{item:tau:partition} there exists a partition of
$\tau_{h_i}=\bigcup_{j=1}^n \tau_{h_i,j}$ such that,
\begin{align}
\Dcomp_j=\mbox{int}\left(\bigcup_{K\in\tau_{h_i,j}}\overline{K}\right),
\end{align}
for all $j\in\{1,\dots,n\}$,
\item $\tau_{h_i}$ is a shape-regular and quasi-uniform mesh
(\emph{cf.}~\cite[Chap.~1]{ern2004theory}),
\item\label{item:tau:sizes} %
\begin{align}\label{eq:tauexp}
{C}_{1} s^{i} \leq h_i\leq {C}_2 s^{i}.
\end{align}%
\end{enumerate}
\end{assumption}

We denote an arbitrary mesh on the sequence
$\{\tau_{h_i}\}_{i\in\IN}$ as $\tau_h$. Note that
condition \eqref{item:tau:sizes} in Assumption \ref{ass:tau}
implies $\lim_{i\to\infty}h_i=0$.

\begin{remark}
\label{rmk:tauexp}
Equation \eqref{eq:tauexp} ensures that the cardinality
$|\tau_{h_i}|$ of the mesh $\tau_{h_i}$ increases. Specifically,
it holds that
\begin{align}
\label{eq:dof_s}
C_{s,3} s^{-3i} \leq {\rm dim} (\bm{P}^c_k(\tau_{h_i}))
\leq C_{s,4} s^{-3i},
\end{align} 
for a second pair of positive constants $C_{s,3}$ and $C_{s,4}$.
This will be relevant for the multilevel results presented
in Section \ref{sec:mlapp}.
\end{remark}

In the following, we assume given a reference tetrahedron
$\rK\subset \R^3$ such that for every $K\in\tau_{h}$ there is an affine
bijective map $\bmT_K:\rK\mapsto K$. For an arbitrary tetrahedron $K$,
we shall make use of the following space of polynomial functions of
degree $k\in\IN$,
\begin{align}\label{eq:poly_spaces_K}
\begin{gathered}
\bm{P}_{k}^{c}({K}) := \bbP_{k-1}\left(K ; \mathbb{C}^{3}\right)\oplus
\set{\bm{p} \in \widetilde{\mathbb{P}}_{k}\left(K, \mathbb{C}^{3}\right)}
{\bx \cdot \bm{p}(\bx )=0\quad\forall\;\bx\in K}.%
\end{gathered}
\end{align}
The curl-conforming edge finite element (FE) on a tetrahedron $K$
is given by the triple $$(K, \bm{P}_{k}^{c}({K}), \Sigma^c_k(K)),$$
where $\Sigma^c_K$ is a set of uni-solvent linear functionals over
$\bm{P}_{k}^{c}({K})$ (\emph{cf.}~\cite[Sec.~5.5]{Monk:2003aa}).
The curl-conforming FE space\textemdash satisfying the PEC boundary
condition \eqref{eq:bc}\textemdash on an affine mesh
$\tau_h\in\{\tau_{h_i}\}_{i\in\IN}$ is then built as follows
\begin{align*}
\bm P_k^c(\tau_h):= \set{\bV\in\hncurlbf{\Dcomp}}{\bV\vert_K \in
\bm{P}_{k}^{c}(K)\quad \forall K\in\tau_h}.
\end{align*}
For the sake of brevity, we avoid specifying all properties satisfied
by the mappings $\bmT_K$ as well as those satisfied by the space
$\bm P_k^c(\tau_h)$ (see \cite{ern2004theory} and \cite{Monk:2003aa}
and references therein).

\subsection{Discrete problem, Quadrature error and Strang's Lemma}
We continue by stating the fully discrete version of Problem
\ref{prob:weak_nom_T} and briefly comment on the conditions
required of the quadrature rules used to approximate the
integrals defining $\act{\cdot}{\cdot}$ and $\fct{\cdot}$
in Problem \ref{prob:weak_comp_T} to ensure convergence
rates of the solution to the fully discrete problem. For a
more detailed analysis we refer to \cite{AJSZ20,AJ_2020}.
\subsubsection{Numerical quadrature}
On the fixed reference tetrahedron $\breve{K}$, we define
a quadrature rule $Q:\cC^0(\breve K;\C)\to\C$ as
\begin{equation}\label{eq:Q}
Q(f):=\sum_{l=1}^L \breve w_l f(\breve{\bm{b}}_l),
\end{equation}
for certain quadrature nodes
$(\breve{\bm{b}}_l)_{l=1}^{L}\subseteq \breve K$ and quadrature
weights $(\breve{w}_l)_{l=1}^L\subseteq \R\backslash\{0\}$. Given a
(nondegenerate) tetrahedron $K$ and the affine bijective element map
$\bmT_K:\breve K\to K$ we obtain a transformed quadrature rule
$Q_K:\cC(K;\C)\to\C$ on $K$ via
\begin{equation}\label{eq:QK}
Q_K(f):=\sum_{l=1}^L w_{l,K} f(\bm{b}_{l,K})\qquad\text{where}\qquad
w_{l,K}:=\modulo{\det(\d\!\bmT_K)}\breve{w}_{l},\quad
{\bm{b}}_{l,K}:=\bmT_K(\breve{\bm{b}}_l).
\end{equation}

\subsubsection{Discrete variational formulation}
Approximating all the integrals in Problem \ref{prob:weak_comp_T}
with quad\-ra\-tures $Q^\bullet_K$ as in \eqref{eq:QK}---on each element $K$
of the mesh $\tau_h$---leads to the following sesquilinear and
antilinear forms:
\begin{align}
\label{eq:aht}
\begin{aligned}
\aht{\widetilde\bU_h}{\widetilde\bV_h}:=&\sum_{K\in\tau_h}Q^1_{K}
\left(\mu_{\widetilde\bT}^{-1}\curl{\widetilde\bU_h}\cdot 
\curl{\overline{{\widetilde\bV_h}}}\right)-\omega^2Q^2_K\left({\eps_{\widetilde\bT}}{\widetilde\bU_h}\cdot
\overline{{\widetilde\bV_h}}\right),
\end{aligned}
\end{align}
and
\begin{align}
\label{eq:fht}
\fht{\widetilde\bV_h}:=-\ii\omega\sum_{K\in\tau_h}
Q^2_K\left({\bJ_{\widetilde\bT}}\cdot\overline{{\widetilde\bV_h}}\right),
\end{align}
for all $\widetilde\bU_h$, $\widetilde\bV_h\in \bm P_k^c(\tau_h)$,
where we have used the same notation as in the statement of
Problem \ref{prob:weak_comp_T},
$Q_K^1$ and $Q_K^2$ are two different
quadrature rules on each $K\in\tau_h$, constructed from
two different quadrature rules $Q^1$ and $Q^2$ over $\rK$
as indicated in equation \eqref{eq:QK}. Since the quadrature rules
require pointwise function evaluations to be well-defined, here
$\mu^{-1}:\D_H\to\C^{3\times 3}$, $\eps:\D_H\to\C^{3\times 3}$ and
$\bJ:\D_H\to\C^3$ are required to be continuous in each element $K\in\tau_h$.
Function evaluations on the boundary of an element
$K$ are understood with respect to the interior limit on the
element $K$. With the previous definitions at hand, we arrive
at the fully discrete variational problem. 

\begin{problem}[Fully discrete computational Maxwell cavity problem]
\label{prob:disc_var_ref}
For each $\bT\in\frakT$, we seek $\widetilde\bE_{\bT,h}\in\bm{{P}}^c_k(\tau_h)$ such that
\begin{align}\label{eq:weakh}
\aht{\widetilde\bE_{\bT,h}}{\widetilde\bV_h}=\fht{\widetilde\bV_h}\qquad
\forall\;\widetilde\bV_h\in\bm{{P}}^c_k(\tau_h).
\end{align}
\end{problem}

\begin{theorem}\label{thm:disc}
Let Assumptions \ref{ass:material_trafo_smooth}, \ref{ass:computational_domain} and \ref{ass:tau}
hold and assume that the weights of the quadratures $Q^1$, $Q^2$ are
positive and at least one of the following two conditions:
\begin{enumerate}
\item The nodes defining $Q^1$ and $Q^2$ are
$\mathbb{P}_{k-1}(\breve{K};\IC)$ and
$\mathbb{P}_{k}(\breve{K};\IC)$-unisolvent, respectively.
\item $Q^1$ and $Q^2$ are exact on $\bbP_{2k-2}(\breve K;\IC)$
and $\bbP_{2k}(\breve K;\IC)$, respectively.
\end{enumerate}
Then, there exists a
unique solution $\widetilde\bE_{\bT,h}\in \bm{P}^c_k(\tau_h)$ of
Problem \ref{prob:disc_var_ref} and it holds that
\begin{align}\label{eq:aprioriEh}
\norm[\hcurlbf{\Dcomp}]{\widetilde\bE_{\bT,h}}\le C\frac{\omega}{\alpha}
\norm[\Lp{2}{\D_{\bT}}]{\bJ},
\end{align}
where $\alpha>0$ is as in \eqref{eq:alpha_holdall_smooth} and
the constant $C>0$ is independent of the
mesh-size and of $\bT\in\frakT$, but depends on $\vartheta$
in \eqref{eq:vartheta_smooth}.
\end{theorem}

\begin{proof}
The discrete coercivity 
\begin{align}
\modulo{\aht{\widetilde\bU_{h}}{\widetilde\bU_h}}\geq C\alpha\norm[\hcurlbf{\Dcomp}]{\widetilde\bU_h}^2\quad\forall\,\widetilde\bU_h\in\bm{P}^c_0(\tau_h),
\end{align}
where $\alpha>0$ is as in \eqref{eq:alpha_holdall_smooth} and $C>0$ is independent of both the mesh-size and $\bT\in\frakT$, but depends on $\vartheta$
in \eqref{eq:vartheta_smooth},
was shown in \cite[Thm.~3.13]{AJSZ20}. The continuity
\begin{align}
\modulo{\aht{\widetilde\bU_{h}}{\widetilde\bV_h}}&\leq C\norm[\hcurlbf{\Dcomp}]{\widetilde\bU_h}\norm[\hcurlbf{\Dcomp}]{\widetilde\bV_h}\quad\forall\,\widetilde\bU_h,\,\widetilde\bV_h\in\bm{P}^c_0(\tau_h),
\end{align}
on the other hand, follows from \cite[Lem.~3.12]{AJSZ20}, where $C>0$ is as before and not necessarily the
same in each appearance.
Moreover, by an application of \cite[Lem.~3.12 (i)]{AJSZ20} we have that
\begin{align*}
\modulo{\fht{\widetilde\bV_h}}
&=\modulo{\omega\sum_{K\in\tau_h} 
  Q^2_K\left(\bJ_{\widetilde\bT} \cdot\overline{{\widetilde\bV_h}}\right)}
\leq C\omega\sum_{K\in\tau_h}
\normz[\Lp{2}{K}]{\bJ_{\widetilde\bT}} \normz[\Lp{2}{K}]{\widetilde\bV_h}\quad\forall\,\widetilde\bV_h\in\bm{P}^c_0(\tau_h),
\end{align*}
for $C>0$ as before.
Together with our assumptions,
the Cauchy-Schwartz inequality and \eqref{eq:rhs_bound},
the complex Lax-Milgram Lemma then
ensures the existence and uniqueness of the solution to Problem
\ref{prob:disc_var_ref} for each $\bT\in\frakT$ and the {\em a priori} bound in \eqref{eq:aprioriEh}.
\end{proof}

\begin{remark}\label{rmk:lip_assumption_dis}
Much like the results in Sections \ref{ssec:pert_dom} and \ref{ssec:pullback}, 
Theorem \ref{thm:disc} makes no use of the smoothness of $\Dnul$, 
the parameters $\varepsilon$, $\mu$ and $\bJ$ or of the transformations $\bT\in\frakT$,
and still hold true when $N=0$ in Assumption \ref{ass:material_trafo_smooth}.
\end{remark}

\subsubsection{Discretization error}
The discretization error
$\norm[\hcurlbf{\Dcomp}]{\widetilde\bE_{\bT}-\widetilde\bE_{\bT,h}}$
may be bounded with the help of Strang's Lemma. 
The following results---studied first in \cite[Sec.~3]{AJSZ20} and later generalized
in \cite{AJ_2020}--- will give sufficient conditions on the quadrature
rules defining $\aht{\cdot}{\cdot}$ and $\fht{\cdot}$ to ensure bounds
on the discretization error with respect to the mesh-size. We begin by
stating Strang's Lemma (\emph{cf.}~\cite{Ciarlet:2002aa}).

For $\widetilde\bU\in\hncurlbf{\Dcomp}$ and arbitrary $i\in\IN$
and $\bT\in\frakT$, set
\begin{align*}
\mathrm{A}_{1,\bT,i}(\widetilde{\bU})&:=
\inf\limits_{\widetilde{\bU}_{h_i}\in\bm{P}^c_k(\tau_{h_i})}\!
\norm[\hcurlbf{\Dcomp}]{\widetilde{\bU}-\widetilde{\bU}_{h_i}}
+\sup_{\substack{\widetilde{\bV}_{h_i} \in\bm{P}^c_k(\tau_{h_i})\\\widetilde\bV_{h_i}\neq 0}}\!
\frac{|\act{\widetilde{\bU}_{h_i}}{\widetilde{\bV}_{h_i}}-\aht{\widetilde{\bU}_{h_i}}{\widetilde{\bV}_{h_i}}|}
{\norm[\hcurlbf{\Dcomp}]{\widetilde{\bV}_{h_i}}},\\
\mathrm{A}_{2,\bT,i}&:=
\sup_{\substack{\widetilde{\bV}_{h_i} \in\bm{P}^c_k(\tau_{h_i})\\\widetilde\bV_{h_i}\neq 0}}
\frac{\vert{\fct{\widetilde{\bV}_{h_i}}-\fht{\widetilde{\bV}_{h_i}}}\vert}
{\norm[\hcurlbf{\Dcomp}]{\widetilde{\bV}_{h_i}}}. 
\end{align*}

\begin{lemma}\label{lemma:strang}
Under the assumptions of Theorems \ref{thm:cont} and \ref{thm:disc},
there exist unique solutions $\widetilde\bE_{\bT}\in\hncurlbf{\Dcomp}$ and
$\widetilde\bE_{\bT,h_i}\in \bm{P}^c_k(\tau_{h_i})$ of Problems
\ref{prob:weak_comp_T} and \ref{prob:disc_var_ref}, respectively,
and $c>0$ independent of the mesh-size and of
$\bT\in\frakT$ such that
\begin{align}\label{eq:strang}
\norm[\hcurlbf{\Dcomp}]{{\widetilde{\bE}_{\bT}}-{\widetilde{\bE}_{\bT,h_i}}}
\le c (\mathrm{A}_{1,\bT,i}(\widetilde{\bE}_{\bT})+\mathrm{A}_{2,\bT,i}),
\end{align}
holds for every $i\in\IN$ and $\bT\in\frakT$.
\end{lemma}

We continue by stating relevant consistency error estimates that will
permit a later application of Strang's Lemma. They correspond to
adaptations to our context of Theorems 2 and 3 in \cite{AJ_2020}
(\emph{cf.~}Lemmas 3.15 and 3.16 in \cite{AJSZ20}). 

\begin{lemma}
\label{lemma:rhserror}
Let Assumptions \ref{ass:material_trafo_smooth}, \ref{ass:computational_domain} and \ref{ass:tau}
hold. If the
quadrature rule $Q^2$ on $\rK$ is exact on polynomials of degree
$k+N-1$, where $N\in\IN$ is as in Assumption \ref{ass:material_trafo_smooth},
then, for any sequence $\{\widetilde\bV_{h_i}\}_{i\in\IN}$ with
$\widetilde\bV_{h_i}\in\bm{P}^c_k(\tau_{h_i})$ for all $i\in\IN$, it holds
that
\begin{align}
\vert{\fct{\widetilde{\bV}_{h_i}}-\fht{\widetilde{\bV}_{h_i}}}\vert\leq C h_i^{N} 
\vert \Dcomp\vert^{\frac{1}{2}-\frac{1}{q}}
\norm[\Wsob{N,q}{\D_H}]{\bJ}
\norm[{0},{\Dcomp}]{\widetilde\bV_{h_i}},
\end{align}    
for each $\bT\in\frakT$, where $q>3$ is as in Assumption \ref{ass:material_trafo_smooth} and
the constant $C>0$ is independent of
$i\in\IN$ and $\bT\in\frakT$.
\end{lemma}
\begin{proof}
Fix $i\in\IN$ and $\bT\in\frakT$ and recall
$\widetilde\bT:=\bT\circ\widehat\bT$ for $\widehat\bT:\Dcomp\to\Dnul$
as in Assumption \ref{ass:computational_domain}. An application of the chain and product
rules as in the proof of Theorem \ref{thm:unif_reg} yields, together with
Assumption \ref{ass:computational_domain}, for any $j\in\{1,\hdots,n\}$
with $n\in\IN$ as in Assumption \ref{ass:computational_domain}, 
that
\begin{align}
\norm[\Wsob{N,q}{\Dcomp_j}]{\bJ_{\widetilde\bT}}\leq C
\norm[\Wsob{N,q}{\bT(\Dnul_j)}]{\bJ}.
\end{align}
The constant $C$ may be chosen independently of $\bT\in\frakT$, so
that $\bJ_{\widetilde\bT}\in \Wsobpw{N,q}{\Dcomp}$ and
\begin{align}\label{eq:JT_bound}
\norm[\Wsobpw{N,q}{\Dcomp}]{\bJ_{\widetilde\bT}}\leq C
\norm[\Wsob{N,q}{\D_\bT}]{\bJ}\leq C
\norm[\Wsob{N,q}{\D_H}]{\bJ},
\end{align}
for $C>0$ not necessarily the same in each appearance. Then, and for any
$\widetilde\bV_{h_i}\in\bm{P}^c_k(\tau_{h_i})$,
Lemma 7 in \cite{AJ_2020} yields
\begin{align}\label{eq:rhs_error1}
\left\vert{\fct{\widetilde{\bV}_{h_i}}-\fht{\widetilde{\bV}_{h_i}}}\right\vert\leq Ch_i^N\sum\limits_{K\in\tau_{h_i}}\modulo{K}^{\half-\frac{1}{p}}\norm[\Wsob{N,p}{K}]{\bJ_\bT}\norm[0,K]{\widehat\bV_{h_i}},
\end{align}
where the positive constant is independent of $i\in\IN$ and
$\bT\in\frakT$. Then, using that $\bJ_{\widetilde\bT}\in\Wsobpw{N,p}{\Dcomp}$
with $p>2$ and Assumption \ref{ass:tau}, Hölder's inequality yields
\begin{align}\label{eq:rhs_error2}
\sum\limits_{K\in\tau_{h_i}}\modulo{K}^{\half-\frac{1}{q}}\norm[\Wsob{N,q}{K}]{\bJ_{\widetilde\bT}}\norm[0,K]{\widetilde\bV_{h_i}}\leq
\vert{\Dcomp}\vert^{\frac{1}{2}-\frac{1}{q}}
\norm[\Wsobpw{N,q}{\Dcomp}]{\bJ_{\widetilde\bT}}
\norm[{0},{\Dcomp}]{\widetilde\bV_{h_i}}.
\end{align}
Combining the estimates in
\eqref{eq:JT_bound}, \eqref{eq:rhs_error1} and \eqref{eq:rhs_error2}
yields the required result.
\end{proof}

\begin{lemma}
\label{lemma:bilerror}
Let Assumptions \ref{ass:material_trafo_smooth},
\ref{ass:computational_domain} and \ref{ass:tau} hold. If the quadrature rules
$Q^1$ and $Q^2$ on $\rK$ are exact on polynomials of
degree $k+N-2$ and $k+N-1$, respectively, where $N\in\IN$ is as in Assumption
\ref{ass:material_trafo_smooth}, then, for any pair of
sequences $\{\widetilde\bU_{h_i}\}_{i\in\IN}$ and
$\{\widetilde\bV_{h_i}\}_{i\in\IN}$ with $\widetilde\bU_{h_i}$,
$\widetilde\bV_{h_i}\in\bm{P}^c_k(\tau_{h_i})$ for all $i\in\IN$, it holds that
\begin{align}
&\vert\act{\widetilde{\bU}_{h_i}}{\widetilde{\bV}_{h_i}}-\aht{\widetilde{\bU}_{h_i}}{\widetilde{\bV}_{h_i}}\vert
\leq h_i^N C
\norm[\hscurlbfpw{N}{\Dcomp}]{\widetilde\bU_{h_i}}
\norm[\hcurlbf{\Dcomp}]{\widetilde\bV_{h_i}},
\end{align}    
for each $\bT\in\frakT$, where the constant $C>0$ is independent of
$i\in\IN$ and $\bT\in\frakT$.
\end{lemma}    
\begin{proof}
An analogous reasoning as that in the
proof of Lemma \ref{lemma:rhserror} shows that there exists a constant
$C>0$ such that
\begin{align}
\norm[\Wsob{N,\infty}{\Dcomp_j;\IC^{3\times 3}}]{\mu^{-1}_{\widetilde\bT}}&\leq
C\norm[\Wsob{N,\infty}{\bT(\Dnul_j);\IC^{3\times 3}}]{\mu^{-1}},
\\
\norm[\Wsob{N,\infty}{\Dcomp_j;\IC^{3\times 3}}]{\eps_{\widetilde\bT}}&\leq
C\norm[\Wsob{N,\infty}{\bT(\Dnul_j);\IC^{3\times 3}}]{\eps},
\end{align}
for any $\bT\in\frakT$ and $j\in\{1,\hdots,n\}$---$n\in\IN$ is as
in Assumption \ref{ass:computational_domain}. Therefore, we have that
$\eps_{\widetilde\bT},\ \mu^{-1}_{\widetilde\bT}\in \Wsobpw{N,\infty}{\Dcomp;\IC^{3\times 3}}$,
with
\begin{align}\label{eq:muT_LamT_bound}
\begin{aligned}
\norm[\Wsobpw{N,\infty}{\Dcomp;\IC^{3\times 3}}]{\mu^{-1}_{\widetilde\bT}}&\leq
C\norm[\Wsob{N,\infty}{\D_\bT;\IC^{3\times 3}}]{\mu^{-1}},
\\
\norm[\Wsobpw{N,\infty}{\Dnul;\IC^{3\times 3}}]{\eps_{\widetilde\bT}}&\leq
C\norm[\Wsob{N,\infty}{\D_\bT;\IC^{3\times 3}}]{\eps},
\end{aligned}
\end{align}
for $C>0$ as before.

Fix $i\in\IN$ and $\bT\in\frakT$. 
For any pair $\widetilde\bU_{h_i},\ \widetilde\bV_{h_i}\in\bm{P}^c_k(\tau_{h_i})$ 
\cite[Lemma 6]{AJ_2020}---also \cite[Thm.~2]{AJ_2020}---yields
\begin{align}
&\vert\act{\widetilde{\bU}_{h_i}}{\widetilde{\bV}_{h_i}}-\aht{\widetilde{\bU}_{h_i}}{\widetilde{\bV}_{h_i}}\vert\\
&\leq Ch_i^N \sum\limits_{K\in\tau_{h_i}}
C_{\eps_{\widetilde\bT},K}\norm[N,K]{\widetilde\bU_{h_i}}
\norm[{0},{K}]{\widetilde\bV_{h_i}}+
C_{\mu_{\widetilde\bT}^{-1},K}\norm[N,K]{\curl\widetilde\bU_{h_i}}
\norm[{0},{K}]{\curl\widetilde\bV_{h_i}},\label{eq:sesq_est_1}
\end{align}
where, for each $K\in\tau_{h_i}$, we have defined
\begin{align}
C_{\eps_{\widetilde\bT},K}
:=
\omega^2\norm[\Wsob{N,\infty}{K}]{\eps_{\widetilde\bT}}
\quad\mbox{and}\quad 
C_{\mu_{\widetilde\bT}^{-1},K}
:=
\norm[\Wsob{N,\infty}{K}]{\mu_{\widetilde\bT}^{-1}},
\end{align}
and the constant $C>0$ is independent of 
$\eps,\mu\in\Wsob{N,\infty}{\D_H;\IC^{3\times 3}}$, $i\in\IN$ and $\bT\in\frakT$.
Since
$\eps_{\widetilde\bT},\ \mu^{-1}_{\widetilde\bT}\in \Wsobpw{N,\infty}{\Dcomp;\IC^{3\times 3}}$,
we have that
\begin{align}
&\sum\limits_{K\in\tau_{h_i}}
C_{\eps_{\widetilde\bT},K}\norm[N,K]{\widetilde\bU_{h_i}}
\norm[{0},{K}]{\widetilde\bV_{h_i}}+
C_{\mu_{\widetilde\bT}^{-1},K}\norm[N,K]{\curl\widetilde\bU_{h_i}}
\norm[{0},{K}]{\curl\widetilde\bV_{h_i}}\\
&\leq C_{\eps_{\widetilde\bT}}\sum\limits_{K\in\tau_{h_i}}
\norm[N,K]{\widetilde\bU_{h_i}}
\norm[{0},{K}]{\widetilde\bV_{h_i}}+
C_{\mu_{\widetilde\bT}^{-1}}\sum\limits_{K\in\tau_{h_i}}
\norm[N,K]{\curl\widetilde\bU_{h_i}}
\norm[{0},{K}]{\curl\widetilde\bV_{h_i}}\\
&\leq \max(C_{\eps_{\widetilde\bT}},C_{\mu_{\widetilde\bT}^{-1}})\sum\limits_{K\in\tau_{h_i}}
\norm[N,K]{\widetilde\bU_{h_i}}
\norm[{0},{K}]{\widetilde\bV_{h_i}}+
\norm[N,K]{\curl\widetilde\bU_{h_i}}
\norm[{0},{K}]{\curl\widetilde\bV_{h_i}}\\
&\leq \max(C_{\eps_{\widetilde\bT}},C_{\mu_{\widetilde\bT}^{-1}})\sum\limits_{K\in\tau_{h_i}}
\norm[\hscurlbf{N}{K}]{\widetilde\bU_{h_i}}
\norm[\hcurlbf{K}]{\widetilde\bV_{h_i}}\\
&\leq \max(C_{\eps_{\widetilde\bT}},C_{\mu_{\widetilde\bT}^{-1}})
\norm[\hscurlbfpw{N}{\Dcomp}]{\widetilde\bU_{h_i}}
\norm[\hcurlbf{\Dcomp}]{\widetilde\bV_{h_i}},\label{eq:sesq_est_2}
\end{align}
where
\begin{align}
C_{\eps_{\widetilde\bT}}:=\omega^2\norm[\Wsobpw{N,\infty}{\Dcomp}]{\eps_{\widetilde\bT}}\quad\mbox{and}\quad C_{\mu_{\widetilde\bT}^{-1}}:=\norm[\Wsobpw{N,\infty}{\Dcomp}]{\mu_{\widetilde\bT}^{-1}}.
\end{align}
Combining the estimates in \eqref{eq:muT_LamT_bound}, \eqref{eq:sesq_est_1}
and \eqref{eq:sesq_est_2} yields the required result.
\end{proof}

In virtue of the requirements of Lemmas  \ref{lemma:rhserror} and \ref{lemma:bilerror},
we continue under the next assumption on the data $\eps$, $\mu$ and $\bJ$ as
well as on the quadrature rules used to construct 
the sesquilinear and antilinear forms in \eqref{eq:aht} and \eqref{eq:fht}.

\begin{assumption}
\label{ass:Q}
Recall $k\in\IN$ as the polynomial degree of the finite element spaces
$\bm{P}^c_k(\tau_h)$ and let $N\in\IN$ be as in Assumption
\ref{ass:material_trafo_smooth}.  We assume that $k\leq N$, that the
weights of the quadratures $Q^1$ and $Q^2$ are positive, that $Q^1$
and $Q^2$ are exact on polynomials of degree $k+N-2$ and $k+N-1$,
respectively, and at least one of the following two conditions is
satisfied
\begin{enumerate}
\item The nodes defining $Q^1$ and $Q^2$ are
$\mathbb{P}_{k-1}(\breve{K};\IC)$ and
$\mathbb{P}_{k}(\breve{K};\IC)$-unisolvent, respectively.
\item $Q^1$ and $Q^2$ are exact on $\bbP_{2k-2}(\breve K;\IC)$
and $\bbP_{2k}(\breve K;\IC)$, respectively.
\end{enumerate}
\end{assumption}

The combination of Lemmas \ref{lemma:rhserror} and \ref{lemma:bilerror}
together with Strang's Lemma (Lemma \ref{lemma:strang}) yields the
following estimate for the convergence rate of $\widetilde\bE_{\bT,h}$
to $\widetilde\bE_\bT$, 
solutions to Problems \ref{prob:disc_var_ref} and \ref{prob:weak_comp_T}, 
respectively.

\begin{theorem}\label{thm:approx}
Let Assumptions \ref{ass:material_trafo_smooth}, \ref{ass:computational_domain}, \ref{ass:tau} and \ref{ass:Q}
hold.
Then, for any $\bT\in\frakT$, there exists a unique solution of
Problem \ref{prob:disc_var_ref},
$\widetilde\bE_{\bT,h_i}\in\bm{P}^c_{k}(\tau_{h_i})$, which satisfies
\begin{align}\label{eq:str_main_claim}
&\norm[\hcurlbf{\Dcomp}]{\widetilde\bE_{\bT}-\widetilde\bE_{\bT,h_i}}
\leq C
h_i^k\norm[\Wsob{N,q}{\D_H}]{\bJ},
\end{align}
where $\widetilde\bE_{\bT}\in\hncurlbf{\Dcomp}$ is the solution of Problem \ref{prob:weak_comp_T}, $N\in\IN$ and $q>3$ are as in Assumption \ref{ass:Q} and the constant $C>0$ is independent of $i\in\IN$ and
$\bT\in\frakT$.
\end{theorem}

\begin{proof}
Fix $i\in\IN$ and $\bT\in\frakT$.
Theorems \ref{thm:unif_pw_reg} and \ref{thm:disc} ensure the existence of unique
solutions $\widetilde\bE_{\bT}$ and $\widetilde\bE_{\bT,h_i}$ for each $i\in\IN$ of Problems
\ref{prob:weak_comp_T} and \ref{prob:disc_var_ref},
respectively. Furthermore, Theorem \ref{thm:unif_pw_reg} also states the
piecewise smoothness of the solution of Problem \ref{prob:weak_comp_T}, i.e.,
$\widetilde\bE_{\bT}\in\Wsobpw{N,q}{\curl;\Dcomp}$, with the bound
\begin{align}\label{eq:str_eq_1}
\norm[\Wsobpw{N,q}{\curl;\Dcomp}]{\widetilde\bE_\bT}\le C
\norm[\Wsob{N-1,q}{\div;\D_H}]{\bJ},
\end{align}
where the constant $C>0$ is independent of $\bT\in\frakT$. Note that
our assumption that $\bJ\in\Wsob{N,q}{\D_H}$ in Assumption \ref{ass:Q}
implies that $\bJ\in\Wsob{N-1,q}{\div;\D_H}$. Moreover,
since $q>3$ we have that $\widetilde\bE_{\bT}\in\hscurlbfpw{N}{\Dcomp}$
with 
\begin{align}\label{eq:str_eq_2}
\norm[\hscurlbfpw{N}{\Dcomp}]{\widetilde\bE_\bT}\le C
\norm[\Wsobpw{N,q}{\curl;\Dcomp}]{\widetilde\bE_\bT},
\end{align}
for $C>0$ independent of $\bT\in\frakT$.
Now, let $\cI^c_k:\hscurlbf{N}{\Dcomp}\to\bm{P}^c_k(\tau_{h_i})$ be the canonical
curl-conforming interpolation operator
(\emph{cf.}~\cite[Sec.~5.5]{Monk:2003aa}), which is a bounded operator
in the $\hcurl{\Dcomp}$-norm in $\bm{P}^c_k(\tau_{h_i})$ for any
$N\in\IN$ (see \cite[Lemma 5.38]{Monk:2003aa}). Lemmas
\ref{lemma:strang} (Strang's Lemma), \ref{lemma:rhserror} and
\ref{lemma:bilerror} then yield,
\begin{align}\label{eq:str_eq_3}
\begin{aligned}
&\norm[\hcurlbf{\Dcomp}]{\widetilde\bE_\bT-\widetilde\bE_{\bT,h_i}}\\
&\leq C\left[
\norm[\hcurlbf{\Dcomp}]{\widetilde\bE_\bT-\cI^c_k(\widetilde\bE_{\bT})}+h_i^N\left(\norm[\hscurlbfpw{N}{\Dcomp}]{\cI^c_k(\widetilde\bE_{\bT})}+\vert\Dcomp\vert^{\half-\frac{1}{q}}\norm[\Wsob{N,q}{\D_H}]{\bJ}\right)\right],
\end{aligned}
\end{align}
where $C>0$ may be chosen to be independent of both $i\in\IN$
and $\bT\in\frakT$. Since the approximation and continuity properties of
$\cI^c_k$ hold on each mesh element $K\in\tau_{h_i}$
(\emph{cf.}~\cite[Lem.~5.48, Thm.~5.41 \& Rmk.~5.42]{Monk:2003aa}),
they also hold on $\Dnul_j$ for each $j\in\{1,\hdots,n\}$ by virtue of Assumption
\ref{ass:computational_domain}, so that
\begin{align}\label{eq:str_eq_4}
\begin{aligned}
\norm[\hcurlbf{\Dcomp}]{\widetilde\bE_\bT-\cI^c_k(\widetilde\bE_{\bT})}&\leq
ch_i^k\norm[\hscurlbfpw{k}{\Dcomp}]{\widetilde\bE_\bT},\\
\norm[\hscurlbfpw{N}{\Dcomp}]{\cI^c_k(\widetilde\bE_{\bT})}&\leq 
c\norm[\hscurlbfpw{N}{\Dcomp}]{\widetilde\bE_{\bT}},
\end{aligned}
\end{align}
where $c>0$ is, once again, independent of $i\in\IN$ and $\bT\in\frakT$.
A combination of the estimates in \eqref{eq:str_eq_1}, \eqref{eq:str_eq_2},
\eqref{eq:str_eq_3} and \eqref{eq:str_eq_4} yields,
\begin{align}
&\norm[\hcurlbf{\Dcomp}]{\widetilde\bE_\bT-\widetilde\bE_{\bT,h_i}}\\
&\leq C\left[
h_i^k\norm[\hscurlbfpw{k}{\Dcomp}]{\widetilde\bE_\bT}+h_i^N\left(\norm[\hscurlbfpw{N}{\Dcomp}]{\widetilde\bE_{\bT}}+\vert\Dcomp\vert^{\half-\frac{1}{q}}\norm[\Wsob{N,q}{\D_H}]{\bJ}\right)\right]\\
&\leq C(1+\vert\Dnul\vert^{\half-\frac{1}{q}})
h_i^k\norm[\Wsob{N,q}{\D_H}]{\bJ}
\end{align}
where the constant $C>0$ is not necessarily the same in each appearance
but is always independent of $i\in\IN$ and $\bT\in\frakT$.
\end{proof}

\section{Parametric solutions}
\label{sec:ParamSol}
We now introduce a rigorous framework to treat domain uncertainties
described by infinite-dimensional parametrizations of admissible
perturbations $\frakT$. Proving convergence rates for the
approximation of the parametric solution will require smoothness
results with respect to the parameters. In particular, we recall shape
holomorphy results %
\cite{Jerez-Hanckes:2017aa} and also discuss higher-order spatial
regularity of these extensions. As such, the current section serves as
preparation for the subsequent convergence results.

\subsection{Admissible parameters}
We shall allow the perturbations $\bT\in\frakT$
to be given through a random variable with values in the compact set
$\frakT\Subset\bm\cC^{N,1}(\Dnul)$ (see Assumption
\ref{ass:material_trafo_smooth}). 
Let, in the following,
\begin{equation}\label{eq:sZ}
{Z_N}:=\bm\cC^{N,1}(\Dnul;\IC^{3}),
\end{equation}
and
\begin{equation}\label{eq:Z}
Z:=\bm\cC^{0,1}(\Dnul;\IC^{3}).
\end{equation}
Note that we have continuous embedding
${Z_N}\hookrightarrow Z$. Throughout, elements of ${Z_N}$ will always
additionally be interpreted to belong to $Z$ without distinction in the
notation. Moreover, and as in \cite{AJSZ20}, we will also require the
data $\eps$, $\mu$ and $\bJ$ to possess holomorphic extensions to
an open set in $\IC^3$ containing the hold-all domain $\D_H$. 
We shall work under the following 
assumption on $\frakT$ and on the data.

\begin{assumption}\label{ass:param}
  There exists an open set $O_{\D_H}\subset\IC^3$, such that
  $\D_H\subset O_{\D_H}$, and holomorphic extensions of $\epsilon$,
  $\mu$ and $\bJ$ (for which we use the same notation) to $O_{\D_H}$
  satisfying, for some $\theta\in\IR$, $\alpha>0$ and $\alpha_s>0$,
  the following bounds:

\begin{gather}
\label{eq:alpha_holdall_m}
\inf_{0\neq \bzeta\in\C^3}
\essinf_{\bx\in O_{\D_H}}\;
\min\left\{
\frac{\Re(\bzeta^\top
e^{\ii\theta}\mu(\bx)^{-1}\bar\bzeta)}{\norm[\C^3]{\bzeta}^2},
\frac{-\Re(\bzeta^\top
e^{\ii\theta}\omega^2\eps(\bx)\bar\bzeta)}{\norm[\C^3]{\bzeta}^2}\right\}\ge\alpha,\\
\inf_{0\neq \bzeta\in\C^3}
\essinf_{\bx\in O_{\D_H}}\;
\min\left\{
\frac{\Re(\bzeta^\top
\mu(\bx)\bar\bzeta)}{\norm[\C^3]{\bzeta}^2},
\frac{\Re(\bzeta^\top
\eps(\bx)\bar\bzeta)}{\norm[\C^3]{\bzeta}^2}\right\}\ge\alpha_s.
\end{gather}
\end{assumption}

\subsection{Holomorphic extension in $\hncurlbf{\Dcomp}$}
We now show that the solution map possesses certain holomorphic
extensions. The term \emph{solution map} here refers to the function
mapping each perturbation
$\bT\in\frakT$ to the solution of either Problem
\ref{prob:weak_comp_T} or \ref{prob:disc_var_ref}. By \emph{holomorphic}
we mean that this map is complex Fr\'echet differentiable as a function
between two complex Banach spaces.

Before proceeding, we rewrite the $\bT\in\frakT$-dependent quantities
in the sesquilinear and antilinear forms defining Problems
\ref{prob:weak_comp_T} and \ref{prob:disc_var_ref} so that the
dependence on $\bT\in\frakT$ is made explicit.
Then, with the notation introduced in Remark \ref{rmk:equiv}, it holds that
\begin{gather}\label{eq:mu_Lam_J_tildT}
\begin{gathered}
\mu_{\widetilde\bT}
=\det(\d\!\widehat\bT)\d\!\widehat\bT^{-1}\det(\d\!\bT\circ\widehat\bT)(\d\!\bT\circ\widehat\bT)^{-1}(\mu\circ\bT\circ\widehat\bT)(\d\!\bT\circ\widehat\bT)^{-\top}\d\!\widehat\bT^{-\top},\\
\eps_{\widetilde\bT}
=\det(\d\!\widehat\bT)\d\!\widehat\bT^{-1}\det(\d\!\bT\circ\widehat\bT)(\d\!\bT\circ\widehat\bT)^{-1}(\eps\circ\bT\circ\widehat\bT)(\d\!\bT\circ\widehat\bT)^{-\top}\d\!\widehat\bT^{-\top},\\
\bJ_{\widetilde\bT}=\det(\d\!\widehat\bT)\d\!\widehat\bT^{-1}\det(\d\!\bT\circ\widehat\bT)(\d\!\bT\circ\widehat\bT)^{-1}(\bJ\circ\bT\circ\widehat\bT).
\end{gathered}
\end{gather}

The structure of the coefficients in \eqref{eq:mu_Lam_J_tildT} 
is only slightly different from the structure of the 
coefficients considered in \cite[Sec.~4]{AJSZ20}. 
Specifically, the coefficients only differ
by the composition with the fixed transformation
$\widehat{\bT}:\Dcomp\to\Dnul$ and by the product with
$\widehat{\bT}$-depending quantities. Therefore, the proofs of
Theorems \ref{thm:holcont} and \ref{thm:holdisc}, establishing
holomorphic extensions of the continuous and discrete solution
maps. Indeed, mapping each $\bT\in\frakT$ to the solutions of Problems
\ref{prob:weak_comp_T} and \ref{prob:disc_var_ref} are only slight
variations of the proofs of Theorems 4.5 and 4.8 in \cite{AJSZ20} and
are omitted for brevity.
\subsubsection{Exact solution}
\label{sec:ExSol}
\begin{theorem}
\label{thm:holcont}
Let Assumptions \ref{ass:material_trafo_smooth}, \ref{ass:computational_domain} and
\ref{ass:param} hold. Then, there
exists an open set $O_{\frakT}\subseteq Z$, with
$\frakT\subseteq O_{\frakT}$, and a holomorphic function
$\widetilde\frakE:O_{\frakT}\to \hncurlbf{\Dcomp}$ such that, for every
$\bT\in O_{\frakT}$, there exists a unique solution
$\widetilde\bE_\bT\in \hncurlbf{\Dcomp}$ of Problem
\ref{prob:weak_comp_T} and $\widetilde\bE_\bT=\widetilde\frakE(\bT)$.
Moreover, it holds that
\begin{equation}\label{eq:exactM}
\sup_{{\bT}\in O_{\frakT}}\norm[\hcurl{\Dcomp}]{\widetilde\frakE(\bT)}
<\infty.
\end{equation}
\end{theorem}

The bound in \eqref{eq:exactM} is not stated explicitly in
\cite[Thm.~4.5]{AJSZ20}, but %
can be achieved by choosing the open superset
  $O_{\frakT}$ of the compact set $\frakT$ small
  enough
(as in \cite[Thm.~4.8]{AJSZ20}).

\subsubsection{Discrete solution}
Using the discrete holomorphy result \cite[Thm.~4.8]{AJSZ20},
we obtain a discrete version of Theorem \ref{thm:holcont}.  Recall that
$k$ denotes the fixed polynomial degree of our approximation spaces
introduced in Section \ref{ssec:fe}.

\begin{theorem}[Theorem 4.8 in \cite{AJSZ20}]\label{thm:holdisc}
Let the assumptions of Theorem \ref{thm:holcont}, as well as Assumptions \ref{ass:tau} and \ref{ass:Q} hold.
Then, there exists an open set $O_{\frakT}\subseteq Z$ independent
of the mesh $\tau_h\in\{\tau_{h_i}\}_{i\in\IN}$, with
$\frakT\subseteq O_{\frakT}$,
and holomorphic functions
$\widetilde\frakE_{h}:O_{\frakT}\to \bm{P}^c_k(\tau_h)$ such that, for every
$\bT\in O_{\frakT}$, there exists a unique solution
$\widetilde\bE_{\bT,h}\in \bm{P}^c_k(\tau_h)$ of Problem \ref{prob:disc_var_ref}
and $\widetilde\bE_{\bT,h}=\widetilde\frakE_h(\bT)$. Moreover,
\begin{equation}\label{eq:exactM_h}
\sup_{\bT\in O_{\frakT}}\norm[\hcurlbf{\Dcomp}]{\widetilde\frakE_h(\bT)} <\infty.
\end{equation}
\end{theorem}
\begin{remark}
The fact that we may take the open set $O_{\frakT}$ as a subset of $Z$ in
Theorems \ref{thm:holcont} and \ref{thm:holdisc}
follows from Remarks \ref{rmk:lip_assumption} and \ref{rmk:lip_assumption_dis}.
\end{remark}

\subsection{Extension in \texorpdfstring{$\Wsobpw{N,q}{\curl;\Dcomp}$}{W}}

In addition to establishing existence of holomorphic extensions in
$\hncurlbf{\Dcomp}$ with quantified  bounds on the size of the holomorphy domains,  proving dimension-independent convergence rate
bounds of multilevel algorithms requires higher order spatial regularity. 
For our analysis it will suffice that the solution map allows an extension
as a mapping to $\Wsobpw{N,q}{\curl;\Dcomp}$, but we shall not require
holomorphy with respect to this topology.

\subsubsection{Exact Solution}

\begin{proposition}
\label{prop:contcont}
Let Assumptions \ref{ass:material_trafo_smooth}, \ref{ass:computational_domain} and \ref{ass:param} hold. Then, with $O_{\frakT}\subseteq Z$ and $\widetilde\frakE$ as in Theorem \ref{thm:holcont}, there
exists an open set $O_{{N},\frakT}\subset{Z_N}$ with
$\frakT\subseteq O_{{N},\frakT}\subseteq O_{\frakT}$ such that for every $\bT\in O_{{N},\frakT}$ it holds that $\widetilde\frakE(\bT)\in\Wsobpw{N,q}{\curl;\Dnul}$, where $N\in\IN$ and $q>3$ are as in Assumption \ref{ass:Q}. 
In addition, one has
\begin{equation}\label{eq:exact_smooth}
\sup_{\bT\in O_{{N},\frakT}}\norm[\Wsobpw{N,q}{\curl;\Dcomp}]{\widetilde\frakE(\bT)} <\infty.
\end{equation}
\end{proposition}
\begin{proof}
In the proof of Theorem \ref{thm:unif_reg} (see \eqref{eq:ellip_alphas}),
we showed that for every $\bT\in\frakT$ and $\bx\in\D_H$ 
there holds that
\begin{align}
\Re(\bm\zeta^\top\eps_{\bT}(\bx)\bm\zeta)\geq\alpha_s\norm[\IC^3]{\d\!\bT^{-\top}(\bx)\bm\zeta}^2\geq\alpha_s\norm[\IC^{3\times 3}]{\d\!\bT(\bx)}^{-2}\norm[\IC^3]{\bm\zeta}^2\geq c\alpha_s\vartheta^2\norm[\IC^3]{\bm\zeta}^2,\\
\Re(\bm\zeta^\top\mu_{\bT}(\bx)\bm\zeta)\geq\alpha_s\norm[\IC^3]{\d\!\bT^{-\top}(\bx)\bm\zeta}^2\geq\alpha_s\norm[\IC^{3\times 3}]{\d\!\bT(\bx)}^{-2}\norm[\IC^3]{\bm\zeta}^2\geq c\alpha_s\vartheta^2\norm[\IC^3]{\bm\zeta}^2,
\end{align}
for all $\bT\in\frakT$ and $x\in\D_H$,
where $\alpha_s>0$ is as in Assumption \ref{ass:param}, 
$\vartheta\in (0,1)$ as in Assumption \ref{ass:material_trafo_smooth}, 
and $c>0$ is a constant independent of $\bT\in\frakT$.
An analogous computation shows that 
\begin{align}
-\Re(\bm\zeta^\top e^{\imath\theta}\eps_{\bT}(\bx)\bm\zeta)\geq\alpha\norm[\IC^3]{\d\!\bT^{-\top}(\bx)\bm\zeta}^2\geq\alpha\norm[\IC^{3\times 3}]{\d\!\bT(\bx)}^{-2}\norm[\IC^3]{\bm\zeta}^2\geq c\alpha\vartheta^2\norm[\IC^3]{\bm\zeta}^2,\\
\Re(\bm\zeta^\top e^{\imath\theta}\mu_{\bT}^{-1}(\bx)\bm\zeta)\geq\alpha\norm[\IC^3]{\d\!\bT^{-\top}(\bx)\bm\zeta}^2\geq\alpha\norm[\IC^{3\times 3}]{\d\!\bT(\bx)}^{-2}\norm[\IC^3]{\bm\zeta}^2\geq c\alpha\vartheta^2\norm[\IC^3]{\bm\zeta}^2,
\end{align}
where $\alpha>0$ and $\theta\in\IR$ are as in Assumption \ref{ass:param}
and $\vartheta\in(0,1)$ and $c>0$ are as before.
Under our assumptions, 
we can find a bounded open set around 
each $\bT\in\frakT$, denoted $N_{\bT}$,
such that,
\begin{align}
\label{eq:unif_coerc_hol}
\begin{gathered}
    \inf_{{\textbf{L}}\in
      N_{\bT}} \inf_{0\neq \bzeta\in\C^3} 
    \essinf_{\bx\in\Dnul}\;
    \min\left\{
      \frac{\Re(\bzeta^\top
        e^{\ii\theta}\mu_{{\textbf{L}}}(\widehat\bx)^{-1}\bar\bzeta)}{\norm[\C^3]{\bzeta}^2},
      \frac{-\Re(\bzeta^\top
        e^{\ii\theta}\varepsilon_{{\textbf{L}}}(\widehat\bx)\bar\bzeta)}{\norm[\C^3]{\bzeta}^2}\right\}\ge\tilde\alpha:=\frac{c\alpha\vartheta^3}{2},\\
\inf_{{\textbf{L}}\in
      N_{\bT}} 
      \inf_{0\neq \bzeta\in\C^3}
\essinf_{\bx\in\Dnul}\;
\min\left\{
\frac{\bzeta^\top
\mu_{{\textbf{L}}}(\widehat\bx)\bar\bzeta}{\norm[\C^3]{\bzeta}^2},
\frac{\bzeta^\top
\varepsilon_{{\textbf{L}}}(\widehat\bx)\bar\bzeta}{\norm[\C^3]{\bzeta}^2}\right\}\ge\tilde\alpha_s:=\frac{c\alpha_s\vartheta^3}{2},
\end{gathered}
\end{align}
where, for ${\textbf{L}}\in N_{\bT}$,
$\mu_{{\textbf{L}}}$ and $\eps_{{\textbf{L}}}$ are as in Remark \ref{rmk:equiv}.
The compactness of $\frakT$ implies that we can cover $\frakT$ by a
finite number of such sets $N_{\bT}$, whose union yields an open and
bounded set denoted $O_{{N},\frakT}$ on which there hold the uniform
coercivity conditions in \eqref{eq:unif_coerc_hol}. Decreasing the
open set $O_{{N},\frakT}$ if necessary, together with an application
of Theorem \ref{thm:unif_pw_reg} yields the uniform bound in
  \eqref{eq:exact_smooth} (\emph{cf.}~Theorem \ref{thm:unif_reg} for
  the dependence of the constant $C>0$ in \eqref{eq:hnq_curl_ET} on
  $\bT\in\frakT$.
\end{proof}

\subsubsection{Discrete Solution}
For the discrete solution, Proposition \ref{prop:contcont} and Theorem
\ref{thm:approx} imply the following uniform approximation result.
Its proof results from arguments analogous to those in the proof
of Proposition \ref{prop:contcont}.

\begin{proposition}\label{prop:contdisc}
  Let Assumptions \ref{ass:material_trafo_smooth} through
  \ref{ass:param} hold.  Then, with the holomorphic mappings
  $\widetilde\frakE:\frakT\to\hncurlbf{\Dcomp}$ and
  $\widetilde\frakE_h:\frakT\to\bm{P}^c_k(\tau_h)$ introduced in
  Theorems \ref{thm:holcont} and \ref{thm:holdisc}, respectively,
  there holds that
\begin{equation}\label{eq:holo_conv}
\sup_{{\bT}\in O_{{N},\frakT}}
\norm[\hcurlbf{\Dcomp}]{\widetilde\frakE(\bT)-\widetilde\frakE_{h_i}(\bT)}\le C B_{{N},{\frakT},\bJ} h_i^k,
\end{equation}
with
$$B_{{N},{\frakT},\bJ}:=\sup_{\bT\in O_{{N},{\frakT}}}\norm[\hscurlbfpw{N}{\Dcomp}]{\widetilde\frakE(\bT)}+\norm[\bm{\cC}^{N}(O_{\D_H})]{\bJ},$$
for each $i\in\IN$, 
where $N\in\IN$ and $q>3$ are as in Assumption \ref{ass:material_trafo_smooth}, 
$k\leq N$ is as in Assumption \ref{ass:Q}, $O_{{N},{\frakT}}\subset{Z_N}$
is as in Proposition \ref{prop:contcont} and the constant $C>0$ is independent of the mesh-size.
\end{proposition}

\subsection{Parameter discretization}
\label{sec:ParEst}
Numerical computations require to suitably discretize the parameter
set $\frakT$ in Assumption \ref{ass:param}. To this end, set
$U:=[-1,1]^\mathbb{N}$ and assume that
$$\{\widetilde\bT_j\}_{j\in\N_0}\subset {Z_N}$$ is a summable
sequence in ${Z_N}$ (see \eqref{eq:sZ}), i.e.,
\begin{equation*}
\{\norm[{Z_N}]{\widetilde\bT_j}\}_{j\in\N_0}\in\ell^{1}(\N_0).
\end{equation*}
For $\bsy\in U$, as in \cite{AJSZ20} we define
\begin{equation}\label{eq:random_var}
\sigma(\bsy):=\bT_0 + \sum_{j\in\N}y_j \bT_j\in {Z_N}\hookrightarrow Z,
\end{equation}
and we consider the set of admissible domain perturbations to be
given as
\begin{equation}\label{eq:frakTexpand}
  \frakT:=\set{\sigma(\bsy)}{\bsy\in U}.
\end{equation}
The continuity of $\sigma:U\to\frakT$ and the compactness of $U$, 
with the product topology (see \cite[Sec.~12.2]{royden2010real})
yields the compactness of $\frakT$ in ${Z_N}$.
Hence, for every $\bsy\in U$, $\sigma(\bsy)\in\frakT$ is an
expansion of a perturbation $\bT\in{Z_N}$ in terms of the
sequence $\bsy$. Throughout what follows, we interpret 
item (\ref{it:Assum_1_frakT}) in Assumption \ref{ass:material_trafo_smooth}
as an Assumption on the sequence $\{\norm[{Z_N}]{\widetilde\bT_j}\}_{j\in\N_0}$,
i.e., we assume that $\{\norm[{Z_N}]{\widetilde\bT_j}\}_{j\in\N_0}$ is such that
item (\ref{it:Assum_1_frakT}) in Assumption \ref{ass:material_trafo_smooth} holds.
Moreover, introducing the infinite product probability
measure $\upmu=\otimes_{j\in\IN}\frac{\lambda}{2}$, where $\lambda$
denotes the %
Lebesgue measure on $[-1,1]$, we can consider
$\sigma$ in \eqref{eq:random_var} as a random variable on
the probability space
$(U,\mathfrak{B},\upmu)$, where $\mathfrak{B}$ is the infinite product Borel
$\sigma$-algebra on $U$. Under Assumption \ref{ass:param}, Problem
\ref{prob:weak_comp_T} has a unique solution for each $\bsy\in U$, which
we shall denote by $\widetilde\bE(\bsy)\in\hncurlbf{\Dcomp}$, more
precisely, with the solution map $\widetilde\frakE$ from Theorem
\ref{thm:holcont},
\begin{equation}
\label{eq:par_cont_sol}
\widetilde\bE(\bsy):=\widetilde\frakE(\sigma(\bsy))\in\hncurlbf{\Dcomp}.
\end{equation}
Similarly, the discrete problem (Problem \ref{prob:disc_var_ref})
possesses a unique solution
\begin{equation}
\label{eq:par_disc_sol}
\widetilde\bE_{h}(\bsy):=\widetilde\frakE_h(\sigma(\bsy))\in\bm{P}^c_k(\tau_h),
\end{equation}
where $\widetilde\frakE_h$ is as in Theorem \ref{thm:holdisc}. Thus,
our goal is to approximate the expected values of the mappings
$\widetilde\bE:U\to \hncurlbf{\Dcomp}$ and
$\widetilde\bE_h:U\to \bm{P}^c_k(\tau_h)$.

Furthermore,
Theorems \ref{thm:holcont} and \ref{thm:holdisc}, together with the
continuity of $\sigma: U\to \frakT$ in \eqref{eq:random_var} and the
compactness of $U$ with the product topology (\emph{cf.}~\cite[Sec.~12.2]{royden2010real})
yield the following result.

\begin{proposition}
\label{prop:integ_E_Eh}
Let Assumptions \ref{ass:material_trafo_smooth} through \ref{ass:param} hold.
Further assume that, for $\sigma:U\to {Z_N}$ as in \eqref{eq:random_var},
it holds that $\{\norm[{Z_N}]{\bT_j}\}_{j\in\IN}$ belongs to $\ell^1(\IN)$.
Then, the mappings $\widetilde\bE$ and $\widetilde\bE_h$ in
\eqref{eq:par_cont_sol} and \eqref{eq:par_disc_sol}, respectively,
are such that $\widetilde\bE\in\lp{2}{U,\mathfrak{B},\upmu;\hocurl{\Dcomp}}$
and $\widetilde\bE_h\in\lp{2}{U,\mathfrak{B},\upmu;\bm{P}^c_k(\tau_h)}$.
Moreover, under the assumptions of Proposition \ref{prop:contdisc}, there holds that
\begin{gather}
\norm[\hcurlbf{\Dcomp}]{\mathbb{E}(\widetilde\bE-\widetilde\bE_{h_i})}\leq C{B_{{N},{\frakT},\bJ}}h_i^k,\\
\norm[\lp{2}{U;\hcurlbf{\Dcomp}}]{\widetilde\bE_{h_{i+1}}-\widetilde\bE_{h_i}}
\leq C{B_{{N},{\frakT},\bJ}}h_{i+1}^k,
\end{gather}
for each $i\in\IN$,
where the constant $C>0$ is independent of the mesh-size, $N\in\IN$, $q>3$ and $k\leq N$ are as in Assumption \ref{ass:Q}, and
$O_{{N},{\frakT}}\subset{Z_N}$ is as in Proposition \ref{prop:contcont}.
\end{proposition}
\begin{proof}
From our assumption that $\{\norm[{Z_N}]{\bT_j}\}_{j\in\IN}$
belongs to $\ell^1(\IN)$ there follows that the mapping $\sigma:U\to {Z_N}$ is continuous, while
Theorems \ref{thm:holcont} and \ref{thm:holdisc} imply the continuity of the mappings $\widetilde\frakE:\frakT\to\hncurlbf{\Dcomp}$ and $\widetilde\frakE_h:\frakT\to\bm{P}^c_k(\tau_h)$, respectively.
Since, the composition $\widetilde\bE$ and $\widetilde\bE_h$ are defined as the compositions of 
$\widetilde\frakE$ and $\widetilde\frakE_h$, respectively, with the continuous mapping $\sigma$, there holds that
they are continuous as well. An application of Corollary A.2.3 in \cite{JZdiss} then gives the Bochner
integrability of $\widetilde\bE:U\to\hncurlbf{\Dcomp}$ and $\widetilde\bE_h:U\to\bm{P}^c_k(\tau_h)$,
with
\begin{gather}
\left(\int_{U}\norm[\hcurlbf{\Dcomp}]{\widetilde\bE(\bsy)}^2\;\d\!\upmu(\bsy)\right)^{\half}\leq
\sup_{{\bT}\in O_{\frakT}}\norm[\hcurl{\Dcomp}]{\widetilde\frakE(\bT)}
<\infty,\\
\left(\int_{U}\norm[\hcurlbf{\Dcomp}]{\widetilde\bE_h(\bsy)}^2\;\d\!\upmu(\bsy)\right)^{\half}\leq
\sup_{{\bT}\in O_{\frakT}}\norm[\hcurl{\Dcomp}]{\widetilde\frakE_h(\bT)}
<\infty,
\end{gather}
where $O_{\frakT}\subset Z$ is as in Theorems \ref{thm:holcont} and \ref{thm:holdisc}. Moreover,
by an application of Proposition \ref{prop:contdisc} and for each $i\in\IN$, we have that
\begin{align}
\norm[\hcurlbf{\Dcomp}]{\mathbb{E}(\widetilde\bE-\widetilde\bE_{h_i})}&\leq
\int_U\norm[\hcurlbf{\Dcomp}]{\widetilde\bE(\bsy)-\widetilde\bE_{h_i}(\bsy)}\;\d\!\upmu(\bsy)\leq C{B_{{N},{\frakT},\bJ}}h_i^k,
\end{align}
and, for each $i\in\IN$, 
\begin{align}
\norm[\lp{2}{U;\hcurlbf{\Dcomp}}]{\widetilde\bE_{h_{i+1}}-\widetilde\bE_{h_i}}&\leq \norm[\lp{2}{U;\hcurlbf{\Dcomp}}]{\widetilde\bE-\widetilde\bE_{h_{i+1}}}+\norm[\lp{2}{U;\hcurlbf{\Dcomp}}]{\widetilde\bE-\widetilde\bE_{h_{i}}}\\
&\leq C{B_{{N},{\frakT},\bJ}}(h_{i+1}^k+h_i^k)\\
&\leq C{B_{{N},{\frakT},\bJ}}\left(1+\frac{C_2}{sC_1}\right)h_{i+1}^k,
\end{align}
where $O_{{N},{\frakT}}\subset{Z_N}$ is as in Proposition \ref{prop:contcont}, $N\in\IN$ is as in Assumptions \ref{ass:material_trafo_smooth} and \ref{ass:computational_domain}, $q>3$ is as
in Assumption \ref{ass:Q}, $C>0$ is as in Proposition \ref{prop:contdisc} and $0<C_1\leq C_2$ are as in Assumption \ref{ass:tau}.
\end{proof}

\section{Multilevel Approximation}\label{sec:mlapp}

\subsection{Multilevel Monte Carlo}
\label{ssec:mlmc}
To present the Multilevel Monte Carlo (MLMC) method, 
we briefly recall some useful definitions and results. 
For further details, we refer to
\cite{ABChS_2011,KCMG_2011,gantner2017computational} 
and the references therein.

\subsubsection{Single level Monte Carlo method}
\label{sec:SLMC}
Let $X$ be a separable Hilbert space. 
For a Bochner integrable random variable $f:U\to X$, 
the Monte Carlo (MC) method attempts to approximate 
the expected value of $f$ by its
mean over a finite set of $\cN\in\IN$ independent and identically
distributed sample evaluations over $U$,
\begin{align*}
{Q}^{\text{MC}}_{\cN}(f):=\frac{1}{\cN}\sum\limits_{i=1}^{\cN}f(\bsy_{(i)}),
\end{align*}
where $\{\bsy_{(i)}\}_{i=1}^{\cN}$ are independent uniform random variables on a probability space $(\Omega,\mathfrak{F},\upnu)$ with values in $U$ (so that ${Q}^{\text{MC}}_{\cN}(f):\Omega\to X$ is, in itself, a random variable on $(\Omega,\mathfrak{F},\upnu)$ with values in $X$).
If $f$ belongs to $\lp{2}{U,\mathfrak{B},\upmu;X}$,
then the following convergence estimate for the MC method
holds (\emph{cf.}~\cite[Lemma 4.1]{ABChS_2011}),
\begin{align}\label{eq:mcl2u}
\norm[\lp{2}{{\Omega;X}}]{Q^{\text{MC}}_\cN(f)-\mathbb{E}(f)}
\leq \cN^{-\half}\norm[\lp{2}{{U;X}}]{f}.
\end{align}

Moreover, 
if each evaluation $f(\bsy_{(i)}(\zeta))$, for $\zeta\in\Omega$,
cannot be computed exactly but
can only be approximated by a numerical method yielding an approximation
to the random variable $f$, denoted $f_h$---where the subindex $h$
signifies, as before, the precision of the method---, 
then it holds that (\emph{cf.}~\cite[Section 2.1]{KCMG_2011})
\begin{align}
\norm[\lp{2}{{\Omega};X}]{Q^{\text{MC}}_\cN(f_h)-\mathbb{E}(f)}
&\leq
\norm[\lp{2}{{\Omega};X}]{Q^{\text{MC}}_\cN(f_h)-\mathbb{E}(f_h)}+
\norm[\lp{2}{{\Omega};X}]{\mathbb{E}(f_h-f)}\\
&\leq
\cN^{-\half}\norm[\lp{2}{U;X}]{f_h}+\norm[X]{\mathbb{E}(f_h-f)}.
\end{align}
The previous estimate implies that the number of samples $\cN$
needs to be chosen proportional to $\norm[X]{\mathbb{E}(f_h-f)}^{-2}$\textemdash
assuming $\norm[\lp{2}{U;X}]{f_h}$ remains bounded\textemdash
to balance the error contributions in the upper bound.
\subsubsection{Multilevel Monte Carlo method}
\label{sec:MLMC}
The MLMC method differs from the standard single level MC Galerkin discretization
in that it employs simultaneously different discretization levels of
the random variable $f$, namely $\{f_{h_i}\}_{i=1}^{L}$ for $L\in\IN$, 
to approximate
$\mathbb{E}(f)$:
\begin{align}\label{eq:MLMC_defi}
Q^{\text{MLMC}}_L(f):=\sum\limits_{i=1}^{L}Q^{\text{MC}}_{\cN_{L,i}}(f_{h_i}-f_{h_{i-1}}),
\end{align}
where $\{\cN_{L,i}\}_{i=1}^{L}\subset\IN$ corresponds to the number
of samples at each level, $f_{h_0}\equiv 0$
and we assume that $\norm[X]{\mathbb{E}(f_{h_i}-f)}$ decreases as 
the sub-index $i$ increases, i.e.~the precision of the method
increases with $i$. Then, under the same assumptions as before,
\begin{align}
&\norm[\lp{2}{{\Omega};X}]{Q^{\text{MLMC}}_L(f)-\mathbb{E}(f)}\\
&\leq\norm[\lp{2}{{\Omega};X}]{Q^{\text{MLMC}}_L(f)
-\sum\limits_{i=1}^L\mathbb{E}(f_{h_i}-f_{h_i-1})}
+\norm[\lp{2}{{\Omega};X}]{\mathbb{E}(f_{h_L}-f)}\\
&\leq\sum\limits_{i=1}^{L}\cN_{L,i}^{-\half}\norm[\lp{2}{{\Omega};X}]{f_{h_i}-f_{h_{i-1}}}
+\norm[X]{\mathbb{E}(f_{h_L}-f)}.\label{eq:mlmc_error_dec}
\end{align}

\subsubsection{Multilevel Monte Carlo for Problem \ref{prob:weak_nom_T}}
We now return to our specific setting and compute convergence
estimates for the approximation, via the MLMC method, of the
expected value of $\widetilde\bE:U\to\hncurlbf{\Dcomp}$, as in
\eqref{eq:par_cont_sol}, through the approximations given by
$\widetilde\bE_h:U\to\bm{P}^c_k(\tau_h)$ as in
\eqref{eq:par_disc_sol}. Under its respective assumptions, Proposition
\ref{prop:integ_E_Eh} ensures that both ${\widetilde\bE(\by)}$ and
${\widetilde\bE_h(\by)}$ belong to
$\lp{2}{U,\mathfrak{B},\upmu;\hncurlbf{\Dnul}}$ as well as the error
estimates required to effectively bound the MLMC error. 
The
number of samples $\cN_{L,i}$ \eqref{eq:mlmc_error_dec} of each level
will be chosen so that the convergence rate of the MLMC method is the
same as that of the FE approximation of $\widetilde\bE$, namely
$h^k$.  To estimate the total work of the method, we see that the
computation of the MLMC estimator at level $L$ requires us to solve
for $\widetilde\bE_{h_i}-\widetilde\bE_{h_{i-1}}$ at $\cN_{L,i}$
random points in $U$, corresponding to
$[{\rm dim} (\bm{P}^c_k(\tau_{h_i}))+{\rm dim}
(\bm{P}^c_k(\tau_{h_{i-1}}))]\cN_{L,i}$ degrees of
freedom\footnote{Here, we take $\widetilde\bE_{h_0}\equiv\bnul$ and
  ${\rm dim} (\bm{P}^c_k(\tau_{h_0}))=0$, as before.}.  Hence,
we %
define the following quantity as an estimate of the
  computational complexity of the method
\begin{align}\label{eq:MLMC_tot_work}
\text{work}(Q^{\text{MLMC}}_{L}):=\sum\limits_{i=1}^{L}\cN_{L,i}\left[{\rm dim} (\bm{P}^c_k(\tau_{h_i}))+{\rm dim}
(\bm{P}^c_k(\tau_{h_{i-1}}))\right].
\end{align} 

\begin{theorem}
\label{thm:MLMC_E}
Let Assumptions \ref{ass:material_trafo_smooth} through \ref{ass:param}
hold and let $\widetilde\bE:U\to\hncurlbf{\Dcomp}$
be as in \eqref{eq:par_cont_sol}. Then,
tere is a choice of $\{\cN_{L,i}\}_{i=1}^{L}$ for each $L\in\IN$, such that 
\begin{align}\label{eq:MLMC_error}
\norm[\lp{2}{{\Omega};\hcurl{\Dcomp}}]
{Q^{\rm{MLMC}}_L(\widetilde\bE)-\mathbb{E}(\widetilde\bE)}
\leq C{B_{{N},{\frakT},\bJ}}h_L^{k},
\end{align}
where the constant $C>0$ is independent of $L$---hence, of the mesh-size---and
$N\in\IN$, $q>3$ and $k\leq N$ are as in Assumption \ref{ass:Q}.
Furthermore, the total work is bounded by
\begin{align*}
\text{work}(Q^{\rm{MLMC}}_L)
=
\begin{cases}
\mathcal{O} \left({\rm dim} (\bm{P}^c_k(\tau_{h_L}))
\log\left({\rm dim} (\bm{P}^c_k(\tau_{h_L}))\right)^{2+2\delta}\right)
\quad &\text{if }k=1,\\%\half<r<\frac{3}{2},\\
\mathcal{O}\left({\rm dim} (\bm{P}^c_k(\tau_{h_L}))^{\frac{2}{3}k}\right)
\quad &\text{if }k>1,
\end{cases}
\end{align*}
for any $\delta>0$.
\end{theorem}

\begin{proof}
From \eqref{eq:mlmc_error_dec} and Proposition \ref{prop:integ_E_Eh}
it follows that
\begin{align*}
&\norm[\lp{2}{\Omega;\hcurlbf{\Dcomp}}]{
Q^{\rm{MLMC}}_L(\widetilde\bE)-\mathbb{E}(\widetilde\bE)}\\
&\leq \norm[\hcurlbf{\Dcomp}]{\mathbb{E}(\widetilde\bE-\widetilde\bE_{h_{L}})}
+\sum\limits_{i=1}^{L}\cN_{L,i}^{-\half}\norm[\lp{2}{U;\hcurlbf{\Dcomp}}]
{\widetilde\bE_{h_i}-\widetilde\bE_{h_{i-1}}}\\
&\leq C{B_{{N},{\frakT},\bJ}} \left( h_L^k+\sum\limits_{i=1}^{L}\cN_{L,i}^{-\half}h_i^k\right),
\end{align*}
where $C>0$ is as in Proposition \ref{prop:integ_E_Eh}.
We take, for each $L\in\IN$ and $i\in\{1,\hdots,L\}$,
$\cN_{L,i}={\mathcal{O}}((h_i/h_L)^{2k}i^{2+2\delta})$ for arbitrary $\delta>0$, so that 
\begin{align*}
\sum\limits_{i=1}^{L}\cN_{L,i}^{-\half}h_i^k\leq Ch_L^k\sum\limits_{i=1}^{L}i^{-(1+\delta)},
\end{align*}
where $C>0$ is independent of $L\in\IN$, and \eqref{eq:MLMC_error} follows upon noticing that the last sum
is bounded for all $L\in\IN$ and $\delta>0$. We continue
with the bounding the total work of the MLMC method. From
\eqref{eq:MLMC_tot_work} and Remark \ref{rmk:tauexp}, it follows that
\begin{align*}
\text{work}(Q^{\rm{MLMC}}_L)\leq C_{s,4}(1+s^{3})\sum\limits_{i=1}^{L}\cN_{L,i} s^{-3i},
\end{align*} 
where $s\in (0,1)$ is as in Assumption \ref{ass:tau}.
Recalling item (\ref{item:tau:sizes}) from 
Assumption \ref{ass:tau}, we may express
our choice for $\cN_{L,i}$ as $\cN_{L,i}=O(s^{2k(i-L)}i^{2+2\delta})$ 
so that
\begin{align*}
\text{work}(Q^{\rm{MLMC}}_L)&\leq  C_{s,4}(1+s^{3})\sum\limits_{i=1}^{L}i^{2+2\delta}s^{2k(i-L)-3i}\\
&\leq  C_{s,4}(1+s^{3})s^{-2kL}\sum\limits_{i=1}^{L}i^{2+2\delta}s^{(2k-3)i}.
\end{align*} 
If $k>1$, then $2k-3>0$---recall $k\in\IN$---and the claimed bound on 
the total work follows from \eqref{eq:dof_s}. 
If $k=1$, 
\begin{align*}
\text{work}(Q^{\rm{MLMC}}_L)&\leq  C_{s,4}(1+s^{3})s^{-2kL}\sum\limits_{i=1}^{L}i^{2+2\delta}s^{(2k-3)i}\\
&=  C_{s,4}(1+s^{3})s^{-3L}\sum\limits_{i=1}^{L}i^{2+2\delta}s^{(3-2k)(L-i)}\\
&\leq  C_{s,4}(1+s^{3})s^{-3L}\sum\limits_{j=0}^{L-1}(L-j)^{2+2\delta}s^{(3-2k)j}\\
&\leq  C_{s,4}(1+s^{3})s^{-3L}L^{2+2\delta}\sum\limits_{j=0}^{L-1}s^{(3-2k)j},
\end{align*} 
where the last sum is bounded for all $L\in\IN$ and 
the bound on the total work follows, again, from \eqref{eq:dof_s}.
\end{proof}

We may then use Theorem \ref{thm:MLMC_E} to bound the approximation
error of the MLMC method with respect to the total required work.
\begin{corollary}\label{cor:MLMC_rate_work}
Under the assumptions of Theorem \ref{thm:MLMC_E}, there is a choice
of $\{\cN_{L,i}\}_{i=1}^L$ such that
\begin{align*}
&\norm[\lp{2}{U;\hcurlbf{\Dcomp}}]{Q^{\rm{MLMC}}_L(\widetilde\bE)-\mathbb{E}(\widetilde\bE)}\leq C {B_{{N},{\frakT},\bJ}}\text{work}(Q^{\rm{MLMC}}_L)^{-\kappa(k)},
\end{align*}
where $k\in\IN$ is the corresponding convergence rate in
Proposition \ref{prop:contdisc}, $C>0$ is as in Theorem
\ref{thm:MLMC_E}, and
\begin{align*}
\kappa(k):=\begin{cases}
\frac{1}{3+\delta}\quad &\text{if }k=1,\\
\frac{1}{2}\quad &\text{if }k>1,
\end{cases}
\end{align*}
for arbitrary $\delta>0$.
\end{corollary}

\subsection{Multilevel Smolyak}
\label{sec:MLSmol}
The Smolyak algorithm provides a method for multi-dimensional
polynomial interpolation of functions on \emph{sparse-grids}.
Recalling the well-known construction requires to
introduce some standard notation first.  
The set of \emph{finitely supported multi-indices} 
is denoted by
\begin{equation*}
  \cF:=\set{\bsnu=(\nu_j)_{j\in\N}\in\N_0^\N}{|\bsnu|<\infty},
\end{equation*}
where $|\bsnu|:=\sum_{j\in\N}\nu_j$. 
For two multi-indices
$\bsnu=(\nu_j)_{j\in\N}$, $\bsmu=(\mu_j)_{j\in\N}\in\CF$ we will 
write
$\bsmu\le\bsnu$ if $\mu_j\le\nu_j$ for all $j\in\N$. 
A set
$\Lambda\subseteq\CF$ will be called \emph{downward closed} if
\begin{equation*}
  (\bsnu\in\Lambda\text{ and }\bsmu\le\bsnu)\qquad\Rightarrow\qquad
  \bsmu\in\Lambda
\end{equation*}
and it will be called \emph{finite} in case it is of finite
cardinality.

\subsubsection{Smolyak interpolation and quadrature}
Let $(\chi_j)_{j\in\N_0}$ be a sequence of so-called \emph{$\R$-Leja
points}, as constructed for example in \cite{CH13}. In particular, the
$(\chi_j)_{j\in\N_0}$ are distinct and $\chi_j\in [-1,1]$ for all
$j\in\N_0$.  These points will represent, in the following, one-dimensional interpolation and quadrature points. In principle, other
points could also be used, however this sequence has a known construction
and the favourable property that the Lebesgue constant of
$\{\chi_j\}_{j=0}^n$ grows at most polynomially as
$n\to\infty$ (see \cite{CH13}).

For $n\in\N_0$ let $I_n: \mathcal{C}^0([-1,1])\to \bbP_n$ 
be the polynomial interpolation operator interpolating a function in 
$(\chi_j)_{j=0}^n$,
i.e.
\begin{equation*}
  (I_nf)(y) = \sum_{j=1}^n f(\chi_j)\prod_{\substack{i=0\\ i\neq j}}^n
  \frac{y-\chi_i}{\chi_j-\chi_i}\qquad y\in [-1,1],
\end{equation*}
where empty products are understood to equal to one. 
For a multi-index
$\bsnu\in\CF$, we set $I_\bsnu:=\bigotimes_{j\in\N}I_{\nu_j}$, 
meaning that
$$I_\bsnu f(\bsy) 
  =
\sum_{\bsmu\le\bsnu}f((\chi_{\mu_j})_{j\in\N})\prod_{j\in\N}
 \prod_{i\neq \mu_j}^{\nu_j} (y_j-\chi_i)/(\chi_{\mu_j}-\chi_i),
$$ 
for every
$\bsy=(y_j)_{j\in\N}\in U$. For a finite downward closed set
$\Lambda\subseteq\CF$ the \emph{Smolyak interpolant}
$I_\Lambda: \mathcal{C}^0(U)\to \mathcal{C}^0(U)$ 
is defined with the so-called
\emph{combination formula} as
\begin{equation*}
  (I_\Lambda f)(\bsy) :=\sum_{\bsnu\in\cF} c_{\Lambda,\bsnu}
  I_\bsnu(\bsy),\qquad
  c_{\Lambda,\bsnu}:=\sum_{\set{\bse\in\{0,1\}^\N}{\bsnu+\bse\in\Lambda}}
  (-1)^{|\bse|}.
\end{equation*}
This
interpolation operator satisfies $I_\Lambda f = f$ for all
$f\in{\rm span}\set{\bsy^\bsnu}{\bsnu\in\Lambda}$ 
(\emph{cf.}~\cite[Lemma 1.3.3]{JZdiss}). 
Similarly, with the numerical integration 
$Q_\bsnu f:=\int_U I_\bsnu f$, 
we set
\begin{equation*}
  Q_\Lambda f :=\sum_{\bsnu\in\cF} c_{\Lambda,\bsnu} Q_\bsnu(\bsy),
\end{equation*}
which gives a quadrature rule for which
$Q_\Lambda f = \int_U f(\bsy)\dd\mu(\bsy)$ 
for all
$f\in{\rm span}\set{\bsy^\bsnu}{\bsnu\in\Lambda}$. 
For more details
about the construction and the properties of $I_\Lambda$ 
and
$Q_\Lambda$ we refer for example to \cite{MR1768951,klimke,CCS13_783,JZdiss}.
\subsubsection{Multilevel algorithm}
\label{sec:MLAlg}
To define a multilevel algorithm, we first associate to every
multi-index $\bsnu$ a work level $\wk_{\bsnu}$, or for short
$\Vwk=(\wk_\bsnu)_{\bsnu\in\CF}$. With $(\tau_{h_i})_{i\in\N}$ as in
Assumption \ref{ass:tau}, 
we shall assume for all $\bsnu\in\CF$ that
\begin{equation*}
  w_\bsnu \in \set{{\rm dim} (\bm{P}^c_k(\tau_{h_i}))}{i\in\N}\cup\{0\},
\end{equation*}
so that $w_\bsnu$ corresponds to the dimension of a FEM space. Furthermore
$|\Vwk|:=\sum_{\bsnu\in\CF} \wk_\bsnu<\infty$. For every $i\in\N$ this
then yields a finite multi-index set
\begin{equation}\label{eq:Gammaj}
  \Gamma_j(\Vwk):=\set{\bsnu\in\CF}{\wk_{\bsnu}\ge {\rm
      dim}(\bm{P}^c_k(\tau_{h_i}))}.
\end{equation}

The multilevel interpolation operator is defined as
\begin{equation}\label{eq:MLint}
  I_\Vwk^{\rm ML} \widetilde\bE:=\sum_{j\in\N}
  I_{\Gamma_i(\Vwk)}(\widetilde\bE_{h_i}-\widetilde\bE_{h_{i-1}})
\end{equation}
and the multilevel quadrature operator as
\begin{equation}\label{eq:MLquad}
  Q_\Vwk^{\rm ML} \widetilde\bE := \sum_{j\in\N}Q_{\Gamma_i(\Vwk)}(\widetilde\bE_{h_i} - \widetilde\bE_{h_{ji1}}), 
\end{equation}
where $\widetilde\bE_{h_j}\in \bm{P}^c_k(\tau_{h_j})$, 
as earlier, is the
discrete solution of \eqref{eq:weakh}. 
Here it is assumed that
$\Gamma_j(\Vwk)$ is downward closed for every $j\in\N$, 
so that
$I_{\Gamma_j(\Vwk)}$ and $Q_{\Gamma_j(\Vwk)}$ are
well defined. 
Furthermore, we point out again our convention
$\widetilde\bE_0\equiv 0$.

As a measure of the work required to compute $I_{\Vwk}^{\rm ML} \widetilde\bE$
and $Q_{\Vwk}^{\rm ML} \widetilde\bE$ 
we consider the total number of degrees of
freedom of all required function approximations. 
For example, computing 
$I_{\Gamma_j(\Vwk)}\widetilde\bE_{h_j} - \widetilde\bE_{h_{j-1}})$ 
or
$Q_{\Gamma_j(\Vwk)}(\widetilde\bE_{h_j} - \widetilde\bE_{h_{j-1}})$ 
requires to evaluate
both approximations 
$\widetilde\bE_{h_j}:U\to \bm{P}^c_k(\tau_{h_j})$ 
and
$\widetilde\bE_{h_{j-1}}:U\to\bm{P}^c_k(\tau_{h_{j-1}})$ 
to
$\bE:U\to\hncurlbf{\widetilde\D}$, 
at $|\Gamma_j(\Vwk)|$ points in $U$.
This
corresponds to
$({\rm dim} (\bm{P}^c_k(\tau_{h_j}))+{\rm dim}
(\bm{P}^c_k(\tau_{h_{j-1}})))|\Gamma_j(\Vwk)|$ degrees of freedom. 
In total
\begin{equation}\label{eq:wrk}
  \wrk(\Vwk):=\sum_{j\in\N}({\rm dim}
  (\bm{P}^c_k(\tau_{h_j}))+{\rm dim}
  (\bm{P}^c_k(\tau_{h_{j-1}})))|\Gamma_j(\Vwk)|
\end{equation}
counts the degrees of freedom of all FE approximations required for
the computation of the multilevel interpolant/quadrature.

\subsubsection{Abstract convergence theory}
We now recall an approximation result for multilevel Smolyak
interpolation and quadrature from \cite{JZdiss}, also see
\cite{ZDS18}. It provides a statement about the algebraic convergence
rate that is achievable in terms of $\wrk(\Vwk)$ in
\eqref{eq:wrk}. Implementing the method requires to determine the work
levels $\Vwk=(\wk_{\bsnu})_{\bsnu\in\CF}$ a priori. We comment on
possible choices when presenting numerical experiments in Section
\ref{sec:experiments}. We first recall the assumptions under which the
subsequent convergence results are valid.

\begin{assumption}\label{ass:holweakerML}
  The spaces $X$, $Z$ and ${Z_N}$ are complex Banach spaces and there
  holds the continuous embedding ${Z_N}\hookrightarrow Z$. For a
  summable sequence $(\psi_j)_{j\in\N} \subseteq {Z_N}$ denote
  $\sigma(\bsy) = \sum_{j\in\N} y_j \psi_j\in Z$ for
  $\bsy\in U$ and set $\cP:=\set{\sigma(\bsy)}{\bsy\in U}\subseteq Z$.
  There exists a constant $M>0$, two summability exponents
  $p\in (0,1)$, ${p_N}\in [p,1)$ and a FE method convergence rate $\alpha>0$
  such that, with $(\tau_{h_j})_{j\in\N}$ as in Assumption \ref{ass:tau}, 
  the following is satisfied
  \begin{enumerate}
  \item\label{item:weaker1}
    $(\norm[Z]{\psi_j})_{j\in\N_0}\in\ell^{p}(\N_0)$ and
    $(\norm[{Z_N}]{\psi_j})_{j\in\N_0}\in\ell^{{p_N}}(\N_0)$,
\item\label{item:weaker2} 
  there exists an open set $O_\cP\subseteq Z$ and a Fr\'echet
  differentiable function $\widetilde\frakE:O_\cP\to X$ such that
  $\cP\subseteq O_\cP$ and
  $\sup_{\parm\in O_\cP}\norm[X]{\widetilde\frakE(\parm)} \le M$,
\item\label{item:weaker3} 
  there exists an open set $O_{\cP,{N}}\subseteq{Z_N}$ and
  Fr\'echet differentiable functions $\widetilde\frakE_{h_j}:O_\cP\to X$ 
  for every $j\in\N$ such that $\cP\subseteq O_{\cP,{N}}$ 
  and for every $j\in\N$
\begin{equation}\label{eq:supzinsO}
  \sup_{\parm\in O_\cP} 
  \norm[X]{\widetilde\frakE(\parm)-\widetilde\frakE_{h_j}(\parm)}\le
  M\qquad\text{and}\qquad \sup_{\parm\in {O_{\cP,{N}}}}
  \norm[X]{\widetilde\frakE(\parm)-\widetilde\frakE_{h_j}(\parm)}\le M
  {\rm dim}(\bm{P}^c_k(\tau_{h_{j}}))^{-\alpha}.
\end{equation}
\end{enumerate}
The functions ${\widetilde\bE}:U\to X$ and ${\widetilde\bE}_{h_j}:U\to X$ 
are given by
$\widetilde\bE(\bsy) = \widetilde\frakE(\sigma(\bsy))$ 
and
$\widetilde\bE_{h_j}(\bsy) = \widetilde\frakE_{h_j}(\sigma(\bsy))$ 
for all
$\bsy\in U$ and $j\in\N$.
\end{assumption}
The following theorem follows \cite[Thm.~3.2.11]{JZdiss} and
\cite[Thm.~3.2.12]{JZdiss}. These bounds rely on the fact that
the sequence $({\rm dim} (\bm{P}^c_k(\tau_{h_j})))_{j\in\N}$ of FE 
degrees of freedom increases exponentially
in the sense of Remark~\ref{rmk:tauexp}.
\begin{theorem}\label{thm:MLint}
  Let $(\tau_{h_j})_{j\in\N}$ satisfy Assumption \ref{ass:tau}.  Let
  $p\in (0,1)$, ${p_N}\in [p,1)$ and $\alpha>0$ be such that $\hat\bE$
  and $(\hat\bE_{\sw})_{\sw\in\SW}$ satisfy Assumption
  \ref{ass:holweakerML}.  Then, there exists $C<\infty$ and
  \begin{enumerate}
    \item a sequence $(\Vwk_n)_{n\in\N}$ of
  sequences of work levels, such that $|\Vwk_n|\to\infty$ as
  $n\to\infty$ and
  \begin{equation*}
    \norm[{\boldsymbol{\mathcal{C}}}^0(U;X)]{\widetilde\bE - I_{\Vwk_n}^{\rm ML} 
      \widetilde\bE}\le C \wrk(\Vwk_n)^{-r_I},\qquad 
         r_I := \min\left\{\alpha,\frac{\alpha(p^{-1}-1)}{\alpha+p^{-1}-{p_N}^{-1}}\right\},
  \end{equation*}
\item a sequence $(\Vwk_n)_{n\in\N}$ of
  sequences of work levels, such  that
  $|\Vwk_n|\to\infty$ as $n\to\infty$ and
  \begin{equation*}
\normc[X]{\int_U \widetilde\bE(\bsy)\dd\mu(\bsy) - Q_{\Vwk_n}^{\rm ML} \widetilde\bE }
    \le C \wrk(\Vwk_n)^{-r_Q},\qquad r_Q 
     :=
    \min\left\{\alpha,\frac{\alpha(2p^{-1}-1)}{\alpha+2p^{-1}-2{p_N}^{-1}}\right\}. 
  \end{equation*}
  \end{enumerate}
  Moreover, in both cases $\Gamma_j(\Vwk_n)$ in \eqref{eq:Gammaj} is finite and
  downward closed for all $j\in\N$ and all $n\in\N$.
\end{theorem}

\subsubsection{Multilevel Interpolation}
Applying Thm.~\ref{thm:MLint} in our setting we obtain the following
theorem for sparse-grid approximation.

\begin{theorem}\label{thm:MWint}
  Fix $N\in\N$ and $q>3$. Let $k\in\N$ be less or equal to $N$ where
  $k$ denotes the polynomial degree of the FEM ansatz space. For some
  $p\in (0,1)$, ${p_N}\in [p,1)$ let (cp.~\eqref{eq:sZ}, \eqref{eq:Z})
  \begin{equation*}
    (\norm[Z]{\widetilde\bT_j})_{j\in\N_0}\in\ell^{p}(\N_0),\qquad
    (\norm[{Z_N}]{\widetilde\bT_j})_{j\in\N_0}\in\ell^{p_s}(\N_0).
  \end{equation*}
  Suppose that with $\frakT=\set{\sigma(\bsy)}{\bsy\in U}$ 
  as in
  \eqref{eq:random_var}--\eqref{eq:frakTexpand}, 
  Assumptions \ref{ass:material_trafo_smooth} through \ref{ass:param} are
  satisfied. 
  Then, there exists $C<\infty$ and a sequence $(\Vwk_n)_{n\in\N}$ of
  sequences of work levels, such that, $|\Vwk_n|\to\infty$ as
  $n\to\infty$ and for all $n\in\N$, the solution $\widetilde\bE(\bsy):=\widetilde\bE_{\sigma(\bsy)}$ of
  Problem \ref{prob:weak_comp_T} satisfies
  \begin{equation*}
    \norm[{\boldsymbol{\mathcal{C}}}^0(U;X)]{\widetilde\bE(\cdot) - (I_{\Vwk_n}^{\rm ML} \widetilde\bE)(\cdot)}\le C
    \wrk(\Vwk_n)^{-r_I}\quad\text{with}\quad r_I =
    \min\left\{\frac{k}{3},\frac{\frac{k}{3}(p^{-1}-1)}{\alpha+p^{-1}-{p_N}^{-1}}\right\},
  \end{equation*}
  and where $I_{\Vwk_n}^{\rm ML} \widetilde\bE$ in \eqref{eq:MLint} is
  defined with the discrete solutions
  $\widetilde\bE_{h_j}\in\bm{P}^c_k(\tau_{h_{j}})$ of
  \eqref{eq:weakh}. Moreover, $\Gamma_j(\Vwk_n)$ in \eqref{eq:Gammaj} is
  finite and downward closed for all $j\in\N$ and all $n\in\N$.
\end{theorem}
\begin{proof}
  We need to verify that Assumption \ref{ass:holweakerML} holds with
  $\alpha:=\frac{k}{3}$ and the spaces
  \begin{equation*}
    X:=\hncurlbf{\widetilde{\D}},\qquad
    Z:=\bm\cC^{0,1}(\Dnul;\IC^{3})\qquad\text{and}\qquad
    {Z_N}:=\bm\cC^{N,1}(\Dnul;\IC^{3}).
  \end{equation*}
  Then the statement is an immediate consequence of
  Thm.~\ref{thm:MLint}. In the following we choose the constant
  $k\in\N$ equal to $N$.

  Assumption \ref{ass:holweakerML} (\ref{item:weaker1}) holds by
  assumption. The existence of an open set $O_{\frakT}\subseteq Z$
  containing $\frakT$ and a bounded holomorphic function
  $\widetilde\frakE:O_{\frakT}\to X$ such that
  $\widetilde\bE(\bsy)=\widetilde\frakE(\sigma(\bsy))$ for all
  $\bsy\in U$ follows by Theorem \ref{thm:holcont}.  This shows
  Assumption \ref{ass:holweakerML} (\ref{item:weaker2}). Finally,
  Prop.~\ref{prop:contdisc} implies the existence of a suitable open
  set $O_{{N},\frakT}\subseteq{Z_N}$
  containing $\frakT$, and constant $M$, and bounded holomorphic maps
  $\widetilde\frakE:O_{\frakT}\to X$ such that
\begin{equation*}
  \sup_{\parm\in O_{\frakT}} \norm[\hncurlbf{\widetilde{\D}}]{\widetilde\frakE(\parm)-\widetilde\frakE_{h_j}(\parm)}\le
  M
\end{equation*}
and
\begin{equation*}
  \sup_{\parm\in {O_{\cP,{N}}}}
  \norm[\hncurlbf{\widetilde{\D}}]{\widetilde\frakE(\parm)-\widetilde\frakE_{h_j}(\parm)}\le
  C h_i^{k}\le C
  {\rm dim}(\bm{P}^c_k(\tau_{h_{j}}))^{-\frac{k}{3}}.
\end{equation*}
  This shows that Assumption \ref{ass:holweakerML} (\ref{item:weaker3})
  is satisfied.
\end{proof}

\subsubsection{Multilevel Quadrature}
In the same fashion, we obtain a result for multilevel
  quadrature.
  
\begin{theorem}\label{thm:MWquad}
  Let the assumptions of Theorem~\ref{thm:MWint} be satisfied.  Then
  there exists $C<\infty$ and a sequence $(\Vwk_n)_{n\in\N}$ of
  sequences of work levels, such that $|\Vwk_n|\to\infty$ as
  $n\to\infty$ and for all $n\in\N$ the solution $\widetilde\bE(\bsy)$ of
  Problem \ref{prob:weak_comp_T} satisfies
  \begin{equation*}
    \norm[\hncurlbf{\widetilde{\D}}]{\mathbb{E}(\widetilde\bE(\cdot)) - (Q_{\Vwk_n}^{\rm ML} \widetilde\bE)(\cdot)}\le C
    \wrk(\Vwk_n)^{-r_Q}\quad\text{where}\quad r_Q =
    \min\left\{\frac{k}{3},\frac{\frac{k}{3}(2p^{-1}-1)}{\alpha+2p^{-1}-2{p_N}^{-1}}\right\},
  \end{equation*}
  where $Q_{\Vwk_n}^{\rm ML} \widetilde\bE$ in \eqref{eq:MLquad} is
  defined with the discrete solutions
  $\widetilde\bE_{h_j}\in\bm{P}^c_k(\tau_{h_{j}})$ of
  \eqref{eq:weakh}. Moreover $\Gamma_j(\Vwk_n)$ in \eqref{eq:Gammaj}
  is finite and downward closed for all $j\in\N$ and all $n\in\N$.
\end{theorem}
\begin{proof}
  Assumption \ref{ass:holweakerML} holds by the same arguments as in
  the proof of Thm.~\ref{thm:MWint}. The statement thus follows by
  Thm.~\ref{thm:MLint}.
\end{proof}

\section{Numerical experiments}
\label{sec:experiments}
We now present a numerical experiment in order
to confirm our results in Theorems \ref{thm:MLMC_E}, \ref{thm:MLint},
\ref{thm:MWint} and \ref{thm:MWquad}.  
Our numerical implementation of
Problem \ref{prob:weak_nom_T} was carried out through the open source
softwares GetDP \cite{getdp} and GMSH \cite{geuzaine2009gmsh}.
\subsection{Problem Setting}
\label{ssec:Num_prob_sett}
For $N_c\in\IN$, we consider
$\widetilde\D := [-1,1]^3$ and parametric transformations given by,
\begin{align}
\sigma(\bsy):=\bT_0 + \Theta\sum_{j=1}^{N_c} y_j 
  \bT_j,\quad\bT_0=\Id, \quad\bT_j=j^{-\rho-1}x_3\sin(2\pi jx_1)\hat{\bm{e}}_3, 
\end{align}
where $\Theta\in (0,\half)$ is a scale parameter and 
$\rho>1$ determines the decay properties of the sequences
$(\norm[Z]{\bT_j})_{j\in\N_0}$ and $(\norm[{Z_N}]{\bT_j})_{j\in\N_0}$. 
Specifically, with ${Z_N}$ and $N\in\IN$ as in \eqref{eq:sZ} 
and $\rho>N$, we impose
\begin{align}
\begin{gathered}
(\norm[Z]{\bT_j})_{j\in\N_0}\in\ell^{p},\quad\forall\ 1> p> \frac{1}{\rho},\\
(\norm[{Z_N}]{\bT_j})_{j\in\N_0}\in\ell^{{p_N}},\quad\forall\ 1> {p_N}> \frac{1}{\rho-N}.
\end{gathered}
\end{align}
We fix the current $\bJ$ as a polynomial of first degree on $\D_H:=[-2,2]^3$
and choose $\epsilon$ and $\mu$ as factors of the identity matrix. The
quadratures ${Q}_{\rK}^1$ and ${Q}_{\rK}^2$ used to build the sesquilinear
and antilinear forms in \eqref{eq:aht} and \eqref{eq:fht}, respectively, are
$5$-points Gaussian quadrature rules---exact on polynomials of degree $3$---
and we consider only first order N\'ed\'elec elements for the discretization of
$\hncurlbf{\widetilde{\D}}$, i.e., $k=1$ in Section \ref{ssec:fe}. Theorem
\ref{thm:approx} then requires $N\geq 1$ to ensure a convergence rate
up to $\rho=1$ with respect to the mesh-size $h$---or $\frac{1}{3}$ with
respect to the dimension of $\bm{P}^c_1(\tau_h)$.

Numerical experiments---for brevity, not presented here---verify
the convergence rate of order $N=1$ with respect to the mesh size.
Moreover, for linear functionals of the electric fields
$G\in\hncurlbf{\widetilde{\D}}^*$ we may prove, 
through Theorem 4.2.14 in \cite{SauterSchwabBEM}, that 
$G(\widetilde\frakE_h(\bT))$ converges to
$G(\widetilde\frakE(\bT))$ at twice the rate with respect to the mesh-size
($\rho=2$ in our context) at which $\widetilde\frakE_h(\bT)$ 
converges to $\widetilde\frakE(\bT)$. 
We choose $G$ as
\begin{align}
G(\bU):=\int_\Dnul\bm{g}(\bx)\cdot\overline{\bU(\bx)}\d\!\bx,
\end{align}
for all $\bU\in\hcurlbf{\widetilde{\D}}$, where $\bm{g}$
is chosen as a polynomial in $\bbP_2(\D_H;\IR^3)$. Hence, for any $\bU\in\bm{P}^c_1(\tau_h)$,
we may compute $G(\bU_h)$ exactly through the use of appropriate Gaussian
quadrature rules. Henceforth, we concern ourselves only with the approximation
of $G(\widetilde\bE(\bsy))$ for $\bsy\in U$ and of its expected value over $U$.
Moreover, for a fair comparison between the MLMC method
and the multilevel Smolyak algorithm, 
we truncate the dimension of the sample space, so that we consider $U:=[-1,1]^{50}$. 

\subsection{Number of samples for the MLMC method}

For our multilevel algorithms, we employ five meshes of the domain $\widetilde{\D}$, 
with
$323$, $3'208$, $16'009$, $117'370$ and $707'141$ degrees of freedom. Table
\ref{tbl:mlmc_sample} indicates the number of samples taken on each mesh for
the multilevel Monte Carlo method. The expected value of the squared error
is then computed as the average over 6 realizations of the method.
\begin{table}[h!]
\begin{tabular}{|c|c|c|c|c|c|}\hline
{DoF} & $L=1$ & $L=2$ & $L=3$ & $L=4$ & $L=5$ \\ \hline
323    & 1 & 9 & 69 & 1025 & 11271 \\
3'208   & - & 2 & 13 & 203  & 2273  \\
16'009  & - & - & 4  & 48   & 537   \\
11'7370 & - & - & -  & 6    & 69    \\
707'141 & - & - & -  & -    & 9 \\ \hline
\end{tabular}
\caption{Number of samples per mesh (identified with their corresponding
degrees of freedom) used in each realization of the MLMC
method. We remark that these do not correspond to the $\{\cN_{L,i}\}_{i=1}^L$
in \eqref{eq:MLMC_defi}, but to the total number of computations carried on
each mesh (e.g., for $L=3$ we take $\cN_{3,1}=60$, $\cN_{3,2}=9$ and $\cN_{3,3}=4$).}
\label{tbl:mlmc_sample}
\end{table}

\subsection{Interpolation and quadrature results}
Figure \ref{fig:interp} displays the interpolation error:
\begin{align}\label{eq:int_error}
\norm[\cC^0(U,\IC)]{G({\widetilde\bE}(\cdot)) - (I_{\Vwk_n}^{\rm ML} G({\widetilde\bE}))(\cdot)},
\end{align}
with respect to the total work of the Smolyak algorithm (as in \eqref{eq:wrk}),
where the supremum in the computation of the $\cC^0(U)$-norm in
\eqref{eq:int_error} is approximated by taking the maximum of 
$G(\widehat\bE(\bsy)) - (I_{\Vwk_n}^{\rm ML} G(\widehat\bE))(\bsy)$ on random points in $U$. 
Figure \ref{fig:quad}, on the other hand, displays the quadrature errors
of the multilevel Smolyak algorithm
\begin{align}
\normc[{\lp{2}{
\Omega,\IC}}]{ Q_{\Vwk_n}^{\rm ML} G({\widetilde\bE})-\mathbb{E}\left(G({\widetilde\bE})\right)},
\end{align}
against the total work of the algorithm (as in \eqref{eq:wrk}), and of
the MLMC method, %
\begin{align}
\norm[\lp{2}{{
\Omega,\IC}}]{
Q^{\rm{MLMC}}_L(G({\widetilde\bE}))-\mathbb{E}\left(G({\widetilde\bE})\right)},
\end{align}
against its corresponding total work (as in \eqref{eq:MLMC_tot_work}),
where $\mathbb{E}(G({\widetilde\bE}))$ is estimated through an overkill computation of the
multilevel Smolyak algorithm. The figures display only the results computed with the first four meshes,
so that the comparison against the overkill computation shows both the meshing error---coming from
the finite element discretization---and the quadrature error---arising from both algorithms.
  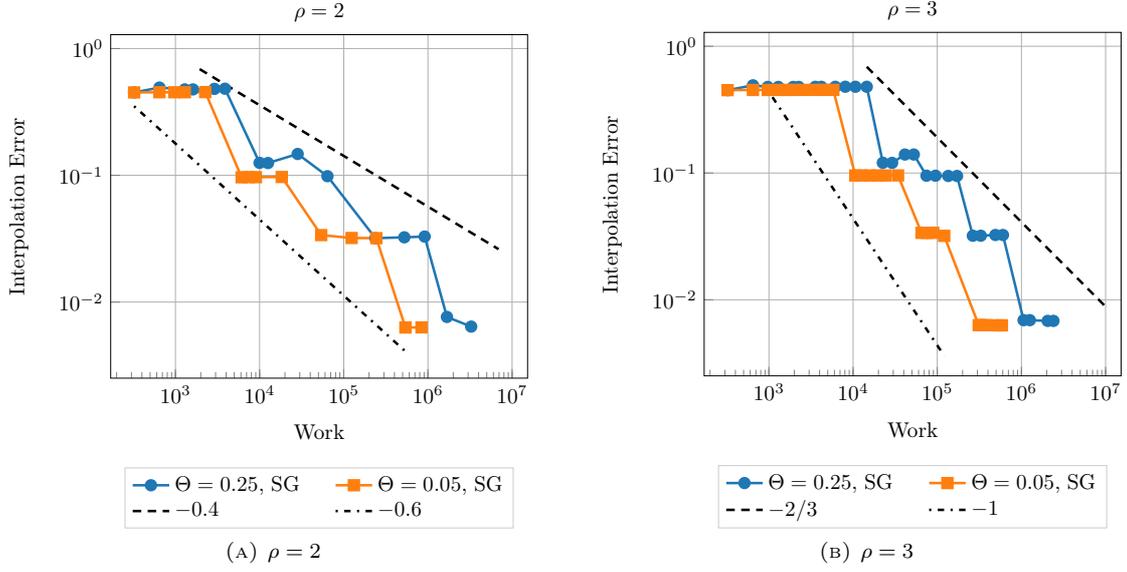
\begin{figure}[ht!]
    \begin{center}
      \subfloat[$\rho=2$]{\begin{tikzpicture}[scale=0.8]

%\definecolor{color1}{rgb}{1,0.498039215686275,0.0549019607843137}
%\definecolor{color0}{rgb}{0.12156862745098,0.466666666666667,0.705882352941177}
%\definecolor{color6}{rgb}{0.149019607843137,0.537254901960784,0.894117647058824}
%\definecolor{color7}{rgb}{0.880392156862745,0.003921568627451,0.041176470588235}
%\definecolor{color8}{rgb}{0.549019607843137,0.537254901960784,0.094117647058824}
%\definecolor{color9}{rgb}{1,1,1}
\definecolor{color1}{rgb}{1,0.498039215686275,0.0549019607843137}
\definecolor{color0}{rgb}{0.12156862745098,0.466666666666667,0.705882352941177}
\definecolor{color6}{rgb}{0.149019607843137,0.537254901960784,0.894117647058824}
\definecolor{color7}{rgb}{0.880392156862745,0.003921568627451,0.041176470588235}
\definecolor{color8}{rgb}{0.549019607843137,0.537254901960784,0.094117647058824}
\definecolor{color9}{rgb}{0.003921568627451,0.880392156862745,0.09141176470588235}

\begin{axis}[
title={$\rho=2$},
xlabel={Work},
ylabel={Interpolation Error},
% ylabel={$|\int_{[-1,1]^{50}} G({\wt
%     \bE}_{T(\bsy),h}\dd\mu_{50}(\bsy)-{\rm Quad}(G({\wt
%     \bE}_{T(\cdot),h}))|$}, 
% ylabel={$|\int_U G(\hat{\bf E}_{h}({\bf y})d\mu({\bf y})-Q_{\Lambda}G(\hat{\bf E}_{h})|$},
xmin=170.132826152243, xmax=15269368.650583,
xmode=log,
xtick={10,100,1000,10000,100000,1000000,10000000,100000000,1000000000,10000000000},
xticklabels={${10^{1}}$,${10^{2}}$,${10^{3}}$,${10^{4}}$,${10^{5}}$,${10^{6}}$,${10^{7}}$,${10^{8}}$,${10^{9}}$,${10^{10}}$},
y grid style={lightgray!92.02614379084967!black},
ymajorgrids,
ymin=2.51149449901867e-3, ymax=1.2868471525574,
ymode=log,
ytick={1e-17,1e-15,1e-13,1e-11,1e-09,1e-07,1e-05,0.001,0.01,0.1,1,10,1000,100000},
yticklabels={${10^{-17}}$,${10^{-15}}$,${10^{-13}}$,${10^{-11}}$,${10^{-9}}$,${10^{-7}}$,${10^{-5}}$,${10^{-3}}$,${10^{-2}}$,${10^{-1}}$,${10^{0}}$,${10^{1}}$,${10^{3}}$,${10^{5}}$},
tick align=inside,%\ra{modified from "outside" on rev1}
tick pos=left,
xmajorgrids,
x grid style={lightgray!92.02614379084967!black},
ymajorgrids,
y grid style={lightgray!92.02614379084967!black},
legend style={draw=white!80.0!black},
% legend style={at={(-0.05,-0.4)},anchor=west},%\ra{added on rev1}
legend cell align={left},
legend entries={
  {$\Theta=0.25$, SG},
  {$\Theta=0.05$, SG},
  {$-0.4$},{$-0.6$}
  },
legend columns = 2
%transpose legend
]
\addplot [very thick, color0, mark=*, mark size=2.3, mark options={solid}]
table [row sep=\\]{%SMK,0.25
323	0.450632738540904 \\
646	0.492045426892736 \\
1292	0.476079413834599 \\
1615	0.475951005145567 \\
2907	0.480578963365201 \\
3876	0.48206076837298 \\
9991	0.12537479330222 \\
12575	0.125006383233828 \\
28336	0.147361749650599 \\
63991	0.09841078638774 \\
237926	0.0319152822583472 \\
523840	0.0324291268142763 \\
913113	0.0328750940071262 \\
1681922	0.00763392192519853 \\
3263446	0.00640504312232883 \\
%8199145	0.0012474219700274 \\
%20711167	0.000916981715673785 \\
%47788902	0.000439391674683577 \\
%110859919	0.00028106564304492 \\
%297948103	0.000228229550117204 \\
};

\addplot [very thick, color1, mark=square*, mark size=2.3, mark options={solid}]
table [row sep=\\]{%SMK 0.05
323	0.450663828414317 \\
646	0.451979609433218 \\
969	0.452007332832455 \\
1292	0.452045166224998 \\
2261	0.453440594277697 \\
6115	0.0964814224961098 \\
7407	0.0968138045266976 \\
9022	0.0967159467679391 \\
18345	0.0969806866957486 \\
53947	0.0337661770947925 \\
123921	0.0320360276024126 \\
243630	0.0319470629536987 \\
540028	0.0063034244045132 \\
838619	0.00631719697527546 \\
%2310640	0.000226362840832556 \\
%4216286	0.000154025338368221 \\
%9767842	0.00013472591917514 \\
%20003522	6.05431171679018e-05 \\
%45722115	3.63455918743739e-05 \\
%117933741	1.24919001620133e-05 \\
};

\addplot [very thick, black, dash pattern=on 4pt off 3pt]
table [row sep=\\]{%rate 0.4
1938 0.688645822328153 \\
6909902  0.02613055723\\
};

\addplot [very thick, black, dash pattern=on 1pt off 3pt on 3pt off 3pt]
table [row sep=\\]{%rate 0.6
323	0.350840726049586\\
540028 0.004084902994\\
};

\end{axis}

\end{tikzpicture}}
      \hfill
      \subfloat[$\rho=3$]{\begin{tikzpicture}[scale=0.8]

%\definecolor{color1}{rgb}{1,0.498039215686275,0.0549019607843137}
%\definecolor{color0}{rgb}{0.12156862745098,0.466666666666667,0.705882352941177}
%\definecolor{color6}{rgb}{0.149019607843137,0.537254901960784,0.894117647058824}
%\definecolor{color7}{rgb}{0.880392156862745,0.003921568627451,0.041176470588235}
%\definecolor{color8}{rgb}{0.549019607843137,0.537254901960784,0.094117647058824}
%\definecolor{color9}{rgb}{1,1,1}
\definecolor{color1}{rgb}{1,0.498039215686275,0.0549019607843137}
\definecolor{color0}{rgb}{0.12156862745098,0.466666666666667,0.705882352941177}
\definecolor{color6}{rgb}{0.149019607843137,0.537254901960784,0.894117647058824}
\definecolor{color7}{rgb}{0.880392156862745,0.003921568627451,0.041176470588235}
\definecolor{color8}{rgb}{0.549019607843137,0.537254901960784,0.094117647058824}
\definecolor{color9}{rgb}{0.003921568627451,0.880392156862745,0.09141176470588235}

\begin{axis}[
title={$\rho=3$},
xlabel={Work},
ylabel={Interpolation Error},
% ylabel={$|\int_{[-1,1]^{50}} G({\wt
%     \bE}_{T(\bsy),h}\dd\mu_{50}(\bsy)-{\rm Quad}(G({\wt
%     \bE}_{T(\cdot),h}))|$}, 
% ylabel={$|\int_U G(\hat{\bf E}_{h}({\bf y})d\mu({\bf y})-Q_{\Lambda}G(\hat{\bf E}_{h})|$},
xmin=170.132826152243, xmax=15269368.650583,
xmode=log,
xtick={10,100,1000,10000,100000,1000000,10000000,100000000,1000000000,10000000000},
xticklabels={${10^{1}}$,${10^{2}}$,${10^{3}}$,${10^{4}}$,${10^{5}}$,${10^{6}}$,${10^{7}}$,${10^{8}}$,${10^{9}}$,${10^{10}}$},
y grid style={lightgray!92.02614379084967!black},
ymajorgrids,
ymin=2.51149449901867e-3, ymax=1.2868471525574,
ymode=log,
ytick={1e-17,1e-15,1e-13,1e-11,1e-09,1e-07,1e-05,0.001,0.01,0.1,1,10,1000,100000},
yticklabels={${10^{-17}}$,${10^{-15}}$,${10^{-13}}$,${10^{-11}}$,${10^{-9}}$,${10^{-7}}$,${10^{-5}}$,${10^{-3}}$,${10^{-2}}$,${10^{-1}}$,${10^{0}}$,${10^{1}}$,${10^{3}}$,${10^{5}}$},
tick align=inside,%\ra{modified from "outside" on rev1}
tick pos=left,
xmajorgrids,
x grid style={lightgray!92.02614379084967!black},
ymajorgrids,
y grid style={lightgray!92.02614379084967!black},
legend style={draw=white!80.0!black},
% legend style={at={(-0.05,-0.4)},anchor=west},%\ra{added on rev1}
legend cell align={left},
legend entries={
  {$\Theta=0.25$, SG},
  {$\Theta=0.05$, SG},
  {$-2/3$},{$-1$}
  },
legend columns = 2
%transpose legend
]
\addplot [very thick, color0, mark=*, mark size=2.3, mark options={solid}]
table [row sep=\\]{%SMK,0.25
323	0.450658790357831 \\
646	0.491707608721829 \\
969	0.477665490281295 \\
1292	0.477516444322703 \\
1938	0.477466585301284 \\
2261	0.477590863513772 \\
3553	0.478460908494066 \\
4199	0.47850403373688 \\
6137	0.478814511563118 \\
8075	0.479314290339491 \\
10659	0.47911508204383 \\
14535	0.47916713131277 \\
22588	0.120409445461039 \\
29048	0.12022492669053 \\
41300	0.139903962781706 \\
52605	0.139952468578181 \\
74525	0.0952214469888753 \\
95197	0.0952216115005821 \\
135161	0.0948038117746541 \\
171660	0.0948054203352641 \\
264668	0.0320136160780451 \\
328622	0.0320134504439138 \\
490532	0.0324442793172223 \\
602936	0.0324440564176664 \\
1061596	0.0069065432230731 \\
1258949	0.00690649082155005 \\
2053834	0.00681587631027597 \\
2400413	0.0068158941483519 \\
%4755836	0.000553957688848173 \\
%5361784	0.000553927510009425 \\
%9484915	0.000644246167238119 \\
%10538541	0.00064423108803339 \\
%19380112	0.000138123945726841 \\
%21211522	0.000138126791226642 \\
%40115359	0.000180687771236993 \\
%43302077	0.000180687263564272 \\
%84762936	3.18376767150496e-05 \\
%90295603	3.18372998753084e-05 \\
%181437516	3.98463060153395e-05 \\
};

\addplot [very thick, color1, mark=square*, mark size=2.3, mark options={solid}]
table [row sep=\\]{%SMK 0.05
323	0.450664764566494 \\
646	0.452005973356029 \\
969	0.452019843792242 \\
1292	0.452470493233708 \\
1615	0.452474923607599 \\
1938	0.452613849450092 \\
2584	0.452684350441069 \\
3553	0.452726196482525 \\
4522	0.452767494555836 \\
5814	0.452765700504013 \\
10637	0.0956198665841664 \\
12575	0.0956314540833716 \\
15805	0.0956285312438773 \\
19681	0.0956333041475547 \\
24203	0.095631564792386 \\
34194	0.09575546497984 \\
65017	0.0337359052271384 \\
75676	0.0337356066781103 \\
89565	0.0337358354271369 \\
120808	0.0319460951714406 \\
310815	0.00631064524622363 \\
341500	0.00631063227621097 \\
381875	0.00631062227293822 \\
518224	0.00628877256365695 \\
586377	0.00628876716862084 \\
%1728003	0.000110305983116096 \\
%1844283	0.000110309457731966 \\
%2051545	0.000110292447635539 \\
%2749104	0.000123275962536208 \\
%5364027	3.53385233143368e-05 \\
%5692841	3.53382607277465e-05 \\
%6121462	3.53386995906259e-05 \\
%10370031	2.5868786993311e-05 \\
%22213050	4.26058714452735e-06 \\
%23150396	4.26058565048818e-06 \\
%24365199	4.26058435647274e-06 \\
%45834068	2.89183501262686e-06 \\
%47879627	2.89182272029317e-06 \\
%105111606	8.55666811342147e-07 \\
};

\addplot [very thick, black, dash pattern=on 4pt off 3pt]
table [row sep=\\]{%rate 2/3
14535 0.688645822328153 \\
10558360  0.00852200198\\
};

\addplot [very thick, black, dash pattern=on 1pt off 3pt on 3pt off 3pt]
table [row sep=\\]{%rate 1
969	0.450840726049586\\
120808  0.003616189851\\
};

\end{axis}

\end{tikzpicture}}      
      \caption{Interpolation error. The theoretical asymptotic
        convergence rates are $0.4-\epsilon$ for $\rho=2$ and
        $\frac{2}{3}$ for $\rho=3$.}\label{fig:interp}
    \end{center}
  \end{figure}

  \begin{figure}[ht!]
    \begin{center}
      \subfloat[$\rho=2$]{\begin{tikzpicture}[scale=0.8]

%\definecolor{color1}{rgb}{1,0.498039215686275,0.0549019607843137}
%\definecolor{color0}{rgb}{0.12156862745098,0.466666666666667,0.705882352941177}
%\definecolor{color6}{rgb}{0.149019607843137,0.537254901960784,0.894117647058824}
%\definecolor{color7}{rgb}{0.880392156862745,0.003921568627451,0.041176470588235}
%\definecolor{color8}{rgb}{0.549019607843137,0.537254901960784,0.094117647058824}
%\definecolor{color9}{rgb}{1,1,1}
\definecolor{color1}{rgb}{1,0.498039215686275,0.0549019607843137}
\definecolor{color0}{rgb}{0.12156862745098,0.466666666666667,0.705882352941177}
\definecolor{color6}{rgb}{0.149019607843137,0.537254901960784,0.894117647058824}
\definecolor{color7}{rgb}{0.880392156862745,0.003921568627451,0.041176470588235}
\definecolor{color8}{rgb}{0.549019607843137,0.537254901960784,0.094117647058824}
\definecolor{color9}{rgb}{0.003921568627451,0.880392156862745,0.09141176470588235}

\begin{axis}[
title={$\rho=2$},
xlabel={Work},
ylabel={Quadrature Error},
% ylabel={$|\int_{[-1,1]^{50}} G({\wt
%     \bE}_{T(\bsy),h}\dd\mu_{50}(\bsy)-{\rm Quad}(G({\wt
%     \bE}_{T(\cdot),h}))|$}, 
% ylabel={$|\int_U G(\hat{\bf E}_{h}({\bf y})d\mu({\bf y})-Q_{\Lambda}G(\hat{\bf E}_{h})|$},
xmin=170.132826152243, xmax=15269368.650583,
xmode=log,
xtick={10,100,1000,10000,100000,1000000,10000000,100000000,1000000000,10000000000},
xticklabels={${10^{1}}$,${10^{2}}$,${10^{3}}$,${10^{4}}$,${10^{5}}$,${10^{6}}$,${10^{7}}$,${10^{8}}$,${10^{9}}$,${10^{10}}$},
y grid style={lightgray!92.02614379084967!black},
ymajorgrids,
ymin=2.51149449901867e-3, ymax=1.2868471525574,
ymode=log,
ytick={1e-17,1e-15,1e-13,1e-11,1e-09,1e-07,1e-05,0.001,0.01,0.1,1,10,1000,100000},
yticklabels={${10^{-17}}$,${10^{-15}}$,${10^{-13}}$,${10^{-11}}$,${10^{-9}}$,${10^{-7}}$,${10^{-5}}$,${10^{-3}}$,${10^{-2}}$,${10^{-1}}$,${10^{0}}$,${10^{1}}$,${10^{3}}$,${10^{5}}$},
tick align=inside,%\ra{modified from "outside" on rev1}
tick pos=left,
xmajorgrids,
x grid style={lightgray!92.02614379084967!black},
ymajorgrids,
y grid style={lightgray!92.02614379084967!black},
legend style={draw=white!80.0!black},
% legend style={at={(-0.05,-0.4)},anchor=west},%\ra{added on rev1}
legend cell align={left},
legend entries={
  {$\Theta=0.25$, SG},
  {$\Theta=0.05$, SG},
  {$\Theta=0.25$, MC},
  {$\Theta=0.05$, MC},
  {$-1$},{$-0.5$}
  },
legend columns = 2
%transpose legend
]
\addplot [very thick, color0, mark=*, mark size=2.3, mark options={solid}]
table [row sep=\\]{%SMK,0.25
323	0.402312505892222 \\
969	0.45542544131929 \\
1292	0.45542544131929 \\
2261	0.456117204029827 \\
3230	0.45577686706695 \\
5814	0.455533870057241 \\
6460	0.455476192796935 \\
12575	0.0985721910825851 \\
13544	0.0985656037404836 \\
18712	0.0985702751131878 \\
21296	0.0985695395432991 \\
34495	0.0899961099486953 \\
38694	0.0899966493286837 \\
59300	0.0907550394982991 \\
67052	0.09075526353319 \\
115897	0.0287981032793189 \\
128494	0.028798191787069 \\
200215	0.0288782344944705 \\
223148	0.0288782473069499 \\
506244	0.00492732525853758 \\
544035	0.00492739328392738 \\
909063	0.00554934870557569 \\
973017	0.00554936692094848 \\
%2614265	0.000377240102470627 \\
%2717625	0.000377233123481659 \\
%4693636	0.000256887449606677 \\
%4859335	0.000256883932453287 \\
%10059369	0.000112105146453338 \\
%10581192	0.000112104664949211 \\
%20890398	6.43509278679018e-05 \\
%21703311	6.43504337124993e-05 \\
%41183286	6.68206448516576e-06 \\
%44119607	6.68199435993049e-06 \\
%84513665	4.61891682724867e-06 \\
%89526910	4.61891995445091e-06 \\
};

\addplot [very thick, color1, mark=square*, mark size=2.3, mark options={solid}]
table [row sep=\\]{%SMK 0.05
323	0.448685483162927 \\
969	0.450903883468434 \\
1292	0.450903883468434 \\
1938	0.450901421868823 \\
5469	0.0941257015467412 \\
6115	0.0941242300900871 \\
6761	0.0941239477099704 \\
20208	0.0320773421109409 \\
20854	0.0320772512294424 \\
22146	0.0320771978692138 \\
22792	0.0320771780525034 \\
134746	0.00613006834897145 \\
136684	0.00613007379388888 \\
138299	0.00613007360772831 \\
138945	0.00613007253896311 \\
141206	0.00613007313102616 \\
%797628	2.44167681551041e-05 \\
%800858	2.44166491141656e-05 \\
%804734	2.44164306762283e-05 \\
%806995	2.44164381482472e-05 \\
%1249504	6.07578465407493e-06 \\
%1255318	6.07573303920006e-06 \\
%1260163	6.07569070526711e-06 \\
%1266623	6.07569100350203e-06 \\
%3247154	9.86105968541613e-07 \\
%3254260	9.86102331298545e-07 \\
%3264919	9.86100919608979e-07 \\
%3277193	9.86101588223902e-07 \\
%3289467	9.86101849779948e-07 \\
%6514488	1.99504081392516e-07 \\
%6533868	1.99503886655116e-07 \\
%6554217	1.99503709737676e-07 \\
%6575858	1.99503651453072e-07 \\
%13492456	8.56277245317478e-09 \\
%13524433	8.56278460715642e-09 \\
%13562870	8.5627825232171e-09 \\
%13609705	8.56278544983439e-09 \\
%33696596	4.44025740980494e-10 \\
%33825591	4.44019703255104e-10 \\
%33893098	4.44041824129574e-10 \\
%33969649	4.44062971545634e-10 \\
%34056859	4.44085499772951e-10 \\
%77487688	6.34146494129582e-14 \\
%77814081	3.45834650755028e-14 \\
};

\addplot [very thick, color9, mark=pentagon*, mark size=3.0, mark options={solid}]
table [row sep=\\]{%MC 0.25
646.0 0.42585759247970845\\
18646.0 0.086187843198320499\\
256054.0 0.037512323709095192\\
4909902.0 0.0054660023434726576\\
};

\addplot [very thick, color7, mark=oplus, mark size=3.0, mark options={solid}]
table [row sep=\\]{%MC 0.05
646.0 0.4497529750525896\\
18646.0 0.093668281577365703\\
256054.0 0.032021520072653221\\
4909902.0 0.0061819228902711409\\
};

\addplot [very thick, black, dash pattern=on 1pt off 3pt on 3pt off 3pt]
table [row sep=\\]{%rate 1
323	0.450840726049586\\
38694  0.003763414341\\
};

\addplot [very thick, black, dash pattern=on 4pt off 3pt]
table [row sep=\\]{%rate 0.5
1938 0.688645822328153 \\
6909902  0.01153285574\\
};

\end{axis}

\end{tikzpicture}}
      \hfill
      \subfloat[$\rho=3$]{\begin{tikzpicture}[scale=0.8]

%\definecolor{color1}{rgb}{1,0.498039215686275,0.0549019607843137}
%\definecolor{color0}{rgb}{0.12156862745098,0.466666666666667,0.705882352941177}
%\definecolor{color6}{rgb}{0.149019607843137,0.537254901960784,0.894117647058824}
%\definecolor{color7}{rgb}{0.880392156862745,0.003921568627451,0.041176470588235}
%\definecolor{color8}{rgb}{0.549019607843137,0.537254901960784,0.094117647058824}
%\definecolor{color9}{rgb}{1,1,1}
\definecolor{color1}{rgb}{1,0.498039215686275,0.0549019607843137}
\definecolor{color0}{rgb}{0.12156862745098,0.466666666666667,0.705882352941177}
\definecolor{color6}{rgb}{0.149019607843137,0.537254901960784,0.894117647058824}
\definecolor{color7}{rgb}{0.880392156862745,0.003921568627451,0.041176470588235}
\definecolor{color8}{rgb}{0.549019607843137,0.537254901960784,0.094117647058824}
\definecolor{color9}{rgb}{0.003921568627451,0.880392156862745,0.09141176470588235}

\begin{axis}[
title={$\rho=3$},
xlabel={Work},
ylabel={Quadrature Error},
% ylabel={$|\int_{[-1,1]^{50}} G({\wt
%     \bE}_{T(\bsy),h}\dd\mu_{50}(\bsy)-{\rm Quad}(G({\wt
%     \bE}_{T(\cdot),h}))|$}, 
% ylabel={$|\int_U G(\hat{\bf E}_{h}({\bf y})d\mu({\bf y})-Q_{\Lambda}G(\hat{\bf E}_{h})|$},
xmin=170.132826152243, xmax=15269368.650583,
xmode=log,
xtick={10,100,1000,10000,100000,1000000,10000000,100000000,1000000000,10000000000},
xticklabels={${10^{1}}$,${10^{2}}$,${10^{3}}$,${10^{4}}$,${10^{5}}$,${10^{6}}$,${10^{7}}$,${10^{8}}$,${10^{9}}$,${10^{10}}$},
y grid style={lightgray!92.02614379084967!black},
ymajorgrids,
ymin=2.51149449901867e-3, ymax=1.2868471525574,
ymode=log,
ytick={1e-17,1e-15,1e-13,1e-11,1e-09,1e-07,1e-05,0.001,0.01,0.1,1,10,1000,100000},
yticklabels={${10^{-17}}$,${10^{-15}}$,${10^{-13}}$,${10^{-11}}$,${10^{-9}}$,${10^{-7}}$,${10^{-5}}$,${10^{-3}}$,${10^{-2}}$,${10^{-1}}$,${10^{0}}$,${10^{1}}$,${10^{3}}$,${10^{5}}$},
tick align=inside,%\ra{modified from "outside" on rev1}
tick pos=left,
xmajorgrids,
x grid style={lightgray!92.02614379084967!black},
ymajorgrids,
y grid style={lightgray!92.02614379084967!black},
legend style={draw=white!80.0!black},
% legend style={at={(-0.05,-0.4)},anchor=west},%\ra{added on rev1}
legend cell align={left},
legend entries={
  {$\Theta=0.25$, SG},
  {$\Theta=0.05$, SG},
  {$\Theta=0.25$, MC},
  {$\Theta=0.05$, MC},
  {$-1$},{$-0.5$}
  },
legend columns = 2
%transpose legend
]
\addplot [very thick, color0, mark=*, mark size=2.3, mark options={solid}]
table [row sep=\\]{%SMK,0.25
323	0.403244067824014 \\
969	0.456358870260461 \\
1292	0.456358870260461 \\
1615	0.457297688610492 \\
2584	0.457236089707229 \\
2907	0.457236490590579 \\
3553	0.457198588117167 \\
8376	0.10023143301944 \\
8699	0.10023143301944 \\
10314	0.100223549637679 \\
12252	0.100220251145777 \\
19637	0.0916326778424319 \\
21575	0.09163220208324 \\
28336	0.091631906010247 \\
30597	0.0916317943145863 \\
58491	0.0296029715657371 \\
63982	0.0296029410935194 \\
72358	0.0296029239703631 \\
75588	0.0296029368934816 \\
264162	0.00571923040687097 \\
273852	0.00571923341686841 \\
313330	0.00571996172049556 \\
324958	0.00571996285008823 \\
%1396123	0.000183650966803092 \\
%1410658	0.000183648208794773 \\
%1652511	0.000165718308273626 \\
%1678674	0.000165717979848015 \\
%%3743962	5.31686406425617e-05 \\
%%3774324	5.31685881937075e-05 \\
%%5208144	4.71083873073011e-05 \\
%%5252072	4.71083683059114e-05 \\
%%9023058	2.49401576093786e-06 \\
%%9087658	2.49401856851982e-06 \\
%%12074421	1.79529302666228e-06 \\
%%12151295	1.79529290984298e-06 \\
%%22909341	2.85725972316706e-07 \\
%%23027882	2.85725759402708e-07 \\
%%32573171	1.27017822610575e-07 \\
%32728211	1.27018745405802e-07 \\
%56453836	1.05292882846587e-13 \\
};

\addplot [very thick, color1, mark=square*, mark size=2.3, mark options={solid}]
table [row sep=\\]{%SMK 0.05
323	0.448685483162927 \\
969	0.450903883468434 \\
1292	0.450903883468434 \\
1938	0.450901421868823 \\
5469	0.0941257015467412 \\
6115	0.0941242300900871 \\
6761	0.0941239477099704 \\
20208	0.0320773421109409 \\
20854	0.0320772512294424 \\
22146	0.0320771978692138 \\
22792	0.0320771780525034 \\
134746	0.00613006834897145 \\
136684	0.00613007379388888 \\
138299	0.00613007360772831 \\
138945	0.00613007253896311 \\
141206	0.00613007313102616 \\
%797628	2.44167681551041e-05 \\
%800858	2.44166491141656e-05 \\
%804734	2.44164306762283e-05 \\
%806995	2.44164381482472e-05 \\
%1249504	6.07578465407493e-06 \\
%1255318	6.07573303920006e-06 \\
%1260163	6.07569070526711e-06 \\
%1266623	6.07569100350203e-06 \\
%3247154	9.86105968541613e-07 \\
%3254260	9.86102331298545e-07 \\
%3264919	9.86100919608979e-07 \\
%3277193	9.86101588223902e-07 \\
%3289467	9.86101849779948e-07 \\
%6514488	1.99504081392516e-07 \\
%6533868	1.99503886655116e-07 \\
%6554217	1.99503709737676e-07 \\
%6575858	1.99503651453072e-07 \\
%13492456	8.56277245317478e-09 \\
%13524433	8.56278460715642e-09 \\
%13562870	8.5627825232171e-09 \\
%13609705	8.56278544983439e-09 \\
%33696596	4.44025740980494e-10 \\
%33825591	4.44019703255104e-10 \\
%33893098	4.44041824129574e-10 \\
%33969649	4.44062971545634e-10 \\
%34056859	4.44085499772951e-10 \\
%77487688	6.34146494129582e-14 \\
%77814081	3.45834650755028e-14 \\
};

\addplot [very thick, color9, mark=pentagon*, mark size=3.0, mark options={solid}]
table [row sep=\\]{%MC 0.25
646.0 0.44167388872951496\\
18646.0 0.092548515994830563\\
256054.0 0.036099377202510732\\
4909902.0 0.007973014411982721\\
};

\addplot [very thick, color7, mark=oplus, mark size=3.0, mark options={solid}]
table [row sep=\\]{%MC 0.05
646.0 0.45139181163055642\\
18646.0 0.093575090620750492\\
256054.0 0.03183834691846546\\
4909902.0 0.0058976327236292354\\
};

\addplot [very thick, black, dash pattern=on 1pt off 3pt on 3pt off 3pt]
table [row sep=\\]{%rate 1
323	0.450840726049586\\
38694  0.003763414341\\
};

\addplot [very thick, black, dash pattern=on 4pt off 3pt]
table [row sep=\\]{%rate 0.5
1938 0.688645822328153 \\
6909902  0.01153285574\\
};

\end{axis}

\end{tikzpicture}}      
      \caption{Quadrature error. Theoretical asymptotic convergence
        rates are $\frac{1}{2}$ for MLMC and $\frac{2}{3}$ for ML
        Smolyak quadrature.}\label{fig:quad}
    \end{center}
  \end{figure}
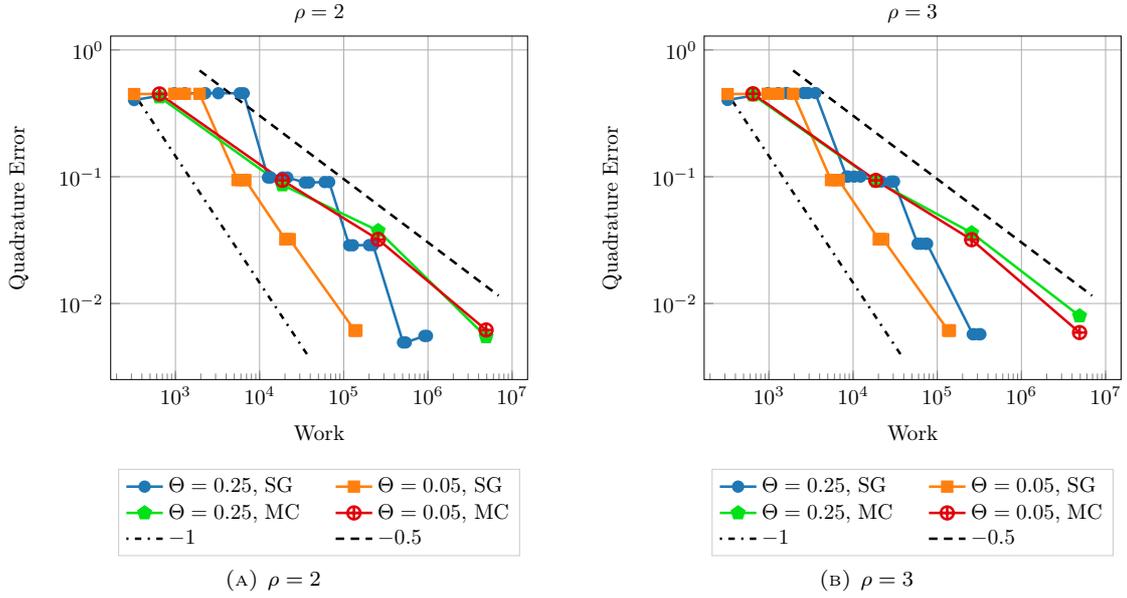
{
For $\rho=2$, we have that for $p\in (1/2,1)$ 
\begin{align}
(\norm[Z]{\bT_j})_{j\in\N_0}\in\ell^{p}.
\end{align}
However, we cannot prove a summability property of $(\norm[Z]{\bT_j})_{j\in\N_0}$ (recall $N=1$). 
Considering, however, $\rho=2+\epsilon$ for small
$\epsilon>0$ yields,
\begin{align}
\begin{gathered}
(\norm[Z]{\bT_j})_{j\in\N_0}\in\ell^{p},\quad\forall\ 1> p> \frac{1}{2+\epsilon},\\
(\norm[{Z_N}]{\bT_j})_{j\in\N_0}\in\ell^{{p_N}},\quad\forall\ 1> {p_N}> \frac{1}{1+\epsilon},
\end{gathered}
\end{align}
and the convergence rate of the multilevel Smolyak interpolation operator
is given by (\emph{cf.}~Theorem \ref{thm:MLint}),
\begin{align}
\min\left\{\frac{2}{3},\frac{\frac{2}{3}(2+\half\epsilon-1)}{\frac{2}{3}+2+\half\epsilon-1-\half\epsilon}\right\}
=\min\left\{\frac{2}{3},\frac{\frac{2}{3}(1+\half\epsilon)}{\frac{5}{3}}\right\}=\frac{2}{5}\left(1+\half\epsilon\right).
\end{align}
On the other hand, the convergence rate for the multilevel Smolyak quadrature will be given by (\emph{cf.}~Theorem \ref{thm:MWquad}),
\begin{align}
\min\left\{\frac{2}{3},\frac{\frac{2}{3}(4+\epsilon-1)}{\frac{2}{3}+4+\epsilon-2-\epsilon}\right\}
=\min\left\{\frac{2}{3},\frac{\frac{2}{3}(3+\epsilon)}{\frac{8}{3}}\right\}
=\min\left\{\frac{2}{3},\frac{3}{4}\left(1+\frac{1}{3}\epsilon\right)\right\}=\frac{2}{3}.
\end{align}
An analogous computation for the case $\rho=3$ yields the convergence rate
of $\kappa=\frac{2}{3}$ for both the multilevel Smolyak interpolation and
quadrature operators.
\section{Conclusions and future work}
\label{sec:Concl}
We have extended our original work \cite{AJSZ20} concerning shape UQ
for Maxwell's lossy cavity problem to multilevel versions of MC and
Smolyak quadrature and interpolation. Theoretically, regularity
results for pullback solutions on the nominal domain are required in
suitable Sobolev spaces. Algorithmically, we have then shown much
better convergence rates and computational costs of parametric
implementations of edge FE in the nominal domain. Our numerical
experiments confirm our theoretical findings and pave the way for
other EM applications or other approximation methods for the
approximation of parametric solution manifolds such as deep neural
networks, see e.g.~\cite{CSJZAAA}.

\bibliographystyle{plain}
\bibliography{main}
\end{document}